\PassOptionsToPackage{lowtilde}{url}		
\documentclass[11pt]{amsart}

\usepackage[english]{babel}
\usepackage{amsfonts}
\usepackage{amssymb}
\usepackage{amsthm}
\usepackage[usenames,dvipsnames]{color}
\usepackage[fleqn]{mathtools}                       
\usepackage{mathrsfs}				    
\usepackage{verbatim}				    
\usepackage[all]{xy}    	          
\SelectTips{cm}{11}                  
\usepackage[T1]{fontenc}        
\usepackage[numbers]{natbib}
\usepackage{graphicx}	
\usepackage[normalem]{ulem}          
\usepackage{microtype}               
\usepackage{tikz-cd}

\usepackage{enumitem}
\setenumerate[1]{label=(\roman*),ref=\thesubsection.{\roman*}}    

\usepackage{fullpage}

\newcommand\twodigits[1]{%
  \ifnum#1<10 0\number#1 \else #1\fi
}

\usepackage{hyperref}
\hypersetup{
  colorlinks   = true,          
  urlcolor     = blue,          
  linkcolor    = blue,          
  citecolor   = red             
}
\newcommand{\N}{{\mathbb{N}}}
\newcommand{\Z}{{\mathbb{Z}}}
\newcommand{\Q}{{\mathbb{Q}}}
\newcommand{\R}{\mathbb{R}}			

\newcommand{\F}{{\mathbb F}}

\newcommand{\Fp}{{\mathbb{F}_p}}		

\newcommand{\PP}{{\mathbb P}}
\renewcommand{\AA}{{\mathbb A}}

\newcommand{\calH}{{\mathcal H}}

\newcommand{\calP}{{\mathcal P}}
\newcommand{\calQ}{{\mathcal Q}}

\newcommand{\calW}{{\mathcal W}}

\newcommand{\frakS}{{\mathfrak S}}
\newcommand{\frakH}{{\mathfrak H}}

\newcommand{\frakW}{{\mathfrak W}}

\newcommand{\id}{\mathrm{id}}
\newcommand{\res}{\mathrm{res}}
\newcommand{\Res}{\mathrm{Res}}
\newcommand{\ord}{\mathrm{ord}}
\newcommand{\Disc}{\mathrm{Disc}}
\newcommand{\Spec}{\mathrm{Spec}}
\newcommand{\rdeg}{\mathrm{rdeg}}
\newcommand{\Moore}[2]{\Delta_{#1}(\underline{#2})}
\newcommand{\ul}[1]{\underline{#1}}

\DeclareMathOperator{\coeff}{coeff}

\DeclareMathOperator{\Aut}{Aut}
\DeclareMathOperator{\Proj}{Proj}
\DeclareMathOperator{\GL}{GL}
\DeclareMathOperator{\sgn}{sgn}
\DeclarePairedDelimiter\floor{\lfloor}{\rfloor}

\newtheoremstyle{plain2}    
  {}            
  {}            
  {\itshape}    
  {}            
  {\bfseries}   
  {.}           
  {5pt plus 1pt minus 1pt}  
  {{\thmnumber{(#2)} \thmname{#1}{\thmnote{ (#3)}}}}          
\newcommand{\daniele}[1]{{\color{Red} \sf $\clubsuit\clubsuit\clubsuit$ Daniele: [#1]}}

 \numberwithin{equation}{section}
\newtheorem{theorem}[equation]{Theorem}
\newtheorem*{theorem*}{Theorem}
\newtheorem{corollary}[equation]{Corollary}
\newtheorem{lemma}[equation]{Lemma}

\newtheorem{proposition}[equation]{Proposition}
\newtheorem*{proposition*}{Proposition}

\theoremstyle{definition}
\newtheorem{definition}[equation]{Definition}
\newtheorem*{definition*}{Definition}
\newtheorem{example}[equation]{Example}

\newtheorem*{question}{Question}

\theoremstyle{remark}
\newtheorem{remark}[equation]{Remark}
\newtheorem{conv}[equation]{Convention}

\newtheoremstyle{stepstyle}
  {}     {}   
  {\normalfont}  
  {\parindent}       
  {\itshape} 
  {}         
  {5pt plus 1pt minus 1pt} 
  {{\thmname{#1} \thmnumber{#2}:{\thmnote{#3}}}}          
\theoremstyle{stepstyle}

\newtheoremstyle{point}
  {}     {}   
  {\normalfont}  
  {}       
  {\bfseries} 
  {}         
  {5pt plus 1pt minus 1pt} 
  {{\thmname{#1}\thmnumber{#2}.\thmnote{ #3.}}}          
\theoremstyle{point}
\newtheorem{point}[subsection]{}

\newtheoremstyle{point*}
  {}     {}   
  {\normalfont}  
  {}       
  {\bfseries} 
  {}         
  {5pt plus 1pt minus 1pt} 
  {{\thmname{#1}\thmnote{ #3.}}}          
\theoremstyle{point*}
\newtheorem{point*}[subsubsection]{}

\title{On the arithmetic and geometry of spaces $L_{m+1,n}$}
\author{Michel Matignon} 
\author{Guillaume Pagot}
\author{Daniele Turchetti}
\date{\today}

\usepackage{etoolbox}
\makeatletter
\patchcmd{\@maketitle}
  {\ifx\@empty\@dedicatory}
  {\ifx\@empty\@date \else {\vskip3ex \centering\footnotesize\@date\par\vskip1ex}\fi
   \ifx\@empty\@dedicatory}
  {}{}
\patchcmd{\@adminfootnotes}
  {\ifx\@empty\@date\else \@footnotetext{\@setdate}\fi}
  {}{}{}
\makeatother

\begin{document}

\begin{abstract}
Let $p$ be a prime number. Motivated by the local lifting problem for $(\Z/p\Z)^n$ with $n>1$, we prove several new results on certain $\F_p$-vector spaces of logarithmic differential forms on the projective line in characteristic $p$, called \emph{spaces $L_{m+1,n}$}.
Expanding the previous work by the first two authors, we prove positive and negative results for the existence of spaces $L_{m+1,n}$ in many situations.
Moreover, we classify all spaces $L_{4p,2}$ for any $p$, and all spaces $L_{15,2}$ for $p=3$.
Among the novel tools we use, Moore determinants and computational algebra play a prominent role.
\end{abstract}

\maketitle

\setcounter{tocdepth}{1}
\section{Introduction}

Algebraic geometry in positive characteristic is a rich subject that has been intensively studied ever since the foundations of algebraic geometry were rigorously established.\footnote{Most notably, studying varieties over finite fields was the main motivation that prompted Weil to write his seminal book \emph{Foundations of Algebraic Geometry} \cite{Weil46}. Weil's theory was subsequently generalized extensively by Grothendieck into the theory of schemes, which forms the foundation of modern algebraic geometry.}
As a result, a plethora of positive characteristic methods have been developed. 
These mathematical tools not only unlock novel insights into the geometry of algebraic varieties but also have surprising applications to other disciplines.
For example, they form the foundations for the development of advanced coding techniques and cryptographic protocols used nowadays. 
Moreover, the richness of algebraic geometry in positive characteristic has been also exploited to prove results in characteristic zero.
Roughly speaking, this is made possible by the use of two complementary procedures: \emph{lifting to characteristic $0$}, that assigns an object in characteristic zero with a given object in positive characteristic, and \emph{reduction modulo $p$}, that assigns an object in positive characteristic with a given object in characteristic zero.

In this paper, we focus on a particular phenomenon in positive characteristic, namely the existence of the so-called \emph{spaces} $L_{m+1,n}$.
\begin{definition*}
Given $k$ an algebraically closed field of characteristic $p>0$ and strictly positive integers $n,m \in \N$, a set $\Omega$ of differential forms on the curve $\PP^1_k$ 
is called a \emph{space $L_{m+1,n}$} if it satisfies the following conditions
\begin{itemize}
\item The set $\Omega$ is a $n$-dimensional vector space over $\F_p$;
\item Every $\omega \in \Omega - \{0\}$ is logarithmic;
\item Every $\omega \in \Omega - \{0\}$ has a unique zero, which is at $\infty$ and has order $m-1$.
\end{itemize}
\end{definition*}

Classically, the motivation for studying spaces $L_{m+1,n}$ arose from
the following local lifting problem:

Let $G$ be a finite group, and $k[\![z]\!]/k[\![t]\!]$ a $G$-galois
extension. Does there exist a finite extension $R$ of the ring of Witt
vectors $W(k)$ with uniformizer $\pi$ and a $G$-galois extension
$R[\![Z]\!]/R[\![T]\!]$ such the $G$ action on $R[\![Z]\!]$ reduces
modulo $\pi $ to the given $G$ action
on $k[\![z]\!]$?

If this problem has a solution for a $G$-extension
$k[\![z]\!]/k[\![t]\!]$, we say that the extension \emph{lifts to
characteristic $0$} and that $R[\![Z]\!]/R[\![T]\!]$ is a \emph{lift} of the extension.
In this case, we can consider the cover $f:~\Spec(R[\![Z]\!])\longrightarrow \Spec (R[\![T]\!])$: it is unramified at the prime ideal $(\pi)$ and the branch points of $R[\![Z]\!]/R[\![T]\!]$ are defined as the divisors $b$ of $\Spec (R[\![T]\!])$ such that $f$ is ramiﬁed at $f^{-1}(b)$.
After extending $R$, we can assume that all branch points of this cover are $R$-rational. Then, the set of these branch points is also called the \emph{branch locus} of $R[\![Z]\!]/R[\![T]\!]$.

The local lifting problem has a relatively long tradition and can be studied from many angles: the interested reader is referred to the surveys \cite{Obus12a} and \cite{Obus19} for a thorough discussion.
In this paper we are interested in the elementary
abelian case, that is when $G\cong (\Z/p\Z)^n$. 
When $n=2$, \cite[Theorem 5.1]{GreenMatignon98} gives necessary and sufficient conditions for lifting, involving the ramification filtration and combinatorics of the branch locus of lifts of intermediate p-cyclic extensions. 
This was recently generalized by Yang \cite{Yang25} to the case $n>2$, and was used as a guide to build lifts for some $(\Z/2\Z)^3$-extensions.
Subsequent work of Pagot \cite{Pagot25}, combining Yang's results and the notion of \emph{Hurwitz trees}, goes further, solving the local lifting problem for $p=2$ and $n=2,3$.

The investigation of spaces $L_{m+1,n}$ of the present paper results in necessary and sufficient conditions for the existence of lifts of $(\Z/p\Z)^n$-extensions with branch locus of a very specific type, for arbitrary $p$ and $n>1$.
In fact, thanks to the results obtained in \cite{GreenMatignon99}, \cite{HenrioThese} (for $n=1$), \cite{Matignon99} and \cite{PagotThese} (for $n>1$), we know that the existence of a space $L_{m+1, n}$ for a given triple $(p,n,m)$ is equivalent to the existence of a lift of a $(\Z/p\Z)^n$--~extension $k[\![z]\!]/k[\![t]\!]$ with a unique ramification jump of conductor $m$ and an \emph{equidistant branch locus}, that is, the branch locus $\{(T-t_i) \in \Spec(R[\![T]\!]) : i \in I\}$ is such that the set $\{v_R(t_i-t_j) : i,j\in I, i\neq j \}$ consists of a unique element.
Up to change of the parameter $T$, we can always assume that there exists a unique integer $r > 0$ such that $v_R(t_i-t_j)=v_R(t_i)=r$ for every $i, j\in I$ with $i\neq j$.
Given such a lift, one can build a space $L_{m+1,n}$ whose set of poles is $\{\frac{t_i}{\pi^r}\mod \pi: 0\leq i \leq m\} \subset k$.
Conversely, for every space $L_{m+1, n}$, a construction by Pagot (\cite[Th\'eor\`eme 11]{Pagot02}) shows the existence of a lift with a unique ramification jump of conductor $m$ and an equidistant branch locus that induces the given space $L_{m+1, n}$.
Moreover, for such a lift to exist, we know that $m+1$ is necessarily of the form $\lambda p^{n-1}$ for some $\lambda>0$, a fact that leads naturally to the following question
\begin{question}
For which triples $(p,n,\lambda)$ does there exist a space $L_{\lambda p^{n-1},n}$?
\end{question}
 
 In this paper we make progress on this largely open question, as well as introduce tools that can be used to see it in a new light.
 Let us recall what was known prior to this paper: if $n=2$ and $p=2$ there are spaces $L_{2\lambda,2}$ for every $\lambda>0$ (\cite[Th\'eor\`eme 8]{Pagot02}).
If $n=2$ and $p\geq 3$ then there are no spaces $L_{p,2}$ and no spaces $L_{3p,2}$, while spaces $L_{2p,2}$ exist if, and only if $p=3$ (\cite[Th\'eor\`eme 9]{Pagot02}).
If $n=2$ and $\lambda \geq 4$ the only examples of spaces $L_{\lambda p,2}$ that were known before this paper have the special property that $\lambda$ is a multiple of $p-1$ and obey a strict geometric constraint. Similarly, in the case $n \geq 3$ the known examples (see Section \ref{sec:standard}) satisfy $p-1 | \lambda$ and are of a very special nature.

Our contributions to the question above solve the problem of existence of spaces $L_{\lambda p^{n-1},n}$ in four cases:
\begin{itemize}
\item The case $n=2$, $p=3$ and $\lambda=4,5$;
\item The case $n=2$ and $p >3 \lambda$;
\item The case $n=2$ and $\lambda=4$;
\item The case $n\geq 2$ and $p=2$.
\end{itemize}
More precisely, we provide a complete classification of spaces $L_{\lambda p, 2}$ for $p=3$ and $\lambda = 4,5$, we show that for $p>3\lambda$ there are no spaces $L_{\lambda p,2}$, and we show that spaces $L_{4p,2}$ exist only for $p\leq 5$, and that for $p=5$ they all are of a very special kind.
Moreover, we prove new results on the existence of spaces $L_{\lambda p^{n-1},n}$ for $n \geq 3$ that, when applied to the case $p=2$, lead to construction of large classes of spaces $L_{\lambda 2^{n-1}, n}$ when $\lambda$ is either even or congruent to $1$ modulo $2^n-2$. 
In proving the results above, we make use of techniques that have not been exploited in this context before, and that we believe to be of independent interest, such as computational commutative algebra, \'etale pullbacks, and Moore determinants.
The first main result presented in the paper is the case $n=2$ and $p >3 \lambda$.

\begin{theorem*}[cf. Theorem \ref{thm:generic}]
There are no spaces $L_{\lambda p,2}$ when $p>3\lambda$.
\end{theorem*}

The proof of this theorem relies on a result of Pagot\footnote{Proposition \ref{prop:Pagot}, first appearing as \cite[Proposition 7]{Pagot02}}, that shows how the existence of $\Omega$ a space $L_{\lambda p,2}$ for fixed $\lambda$ and $p$ is equivalent to the existence of two polynomials $Q_1, Q_2 \in k[X]$ of degree $\lambda$ satisfying three conditions.
The first condition is that $Q_1$ and $Q_2$ are $\F_p$-linearly independent.
The second condition is that the set of poles of $\Omega$ coincides with the set of zeroes of the polynomial $Q_1 Q_2^p - Q_1^p Q_2$.
The third condition can be expressed as a multivariate polynomial system, whose indeterminates are the coefficients of $Q_1$ and $Q_2$.
This system is in general very complicated, but a subsystem of necessary conditions can be extracted thanks to the fact that the nonzero elements of $\Omega$ are logarithmic and hence their poles have residues in $\F_p^\times$.
In the context of our theorem, the necessary conditions can be expressed in terms of the coefficients of $Q_1$ and $Q_2$ and become tractable enough to get a contradiction under the assumptions of the theorem.

This result is a big progress towards the classification of spaces $L_{\lambda p,2}$: thanks to it, the existence of a space $L_{\lambda p,2}$ for a fixed $\lambda$ needs to be checked only at a finite number of primes, which in principle can be done by solving the polynomial system in the coefficients of $Q_1$ and $Q_2$ given by Proposition \ref{prop:Pagot}.
The use of computational tools, such as one of the many algorithms to compute Gr\"obner bases, is essential to perform this task, but not enough to conclude.
In fact, the resulting computations are of a high complexity and a straightforward implementation does not yield a solution in a reasonable time when $\lambda >3$.
Hence, some simplifications are needed to reduce the number of variables in the polynomial system under consideration.
A first simplification is made possible by the fact that a space $L_{\lambda p^{n-1}, n}$ can be transformed by applying to $\PP^1_k$ a homography that fixes $\infty$.
The resulting space is again a space $L_{\lambda p^{n-1}, n}$ that is said to be \emph{equivalent} to the first.
Another useful construction is obtained by considering the action of the relative Frobenius morphism on a space $L_{\lambda p^{n-1}, n}$ (see Section \ref{sec:Frobetale}).
The resulting space is again a space $L_{\lambda p^{n-1}, n}$ that is said to be \emph{Frobenius equivalent} to the first.

When $p=3$ and $n=2$, by performing the above simplifications and applying considerations of symmetry under change of variables, we can find enough relations between the coefficients of $Q_1$ and $Q_2$ to make the polynomial system treatable in the cases $\lambda=4$ and $\lambda=5$.
In this way, we not only find instances of, but we can also classify all possible spaces $L_{12,2}$ and $L_{15,2}$.

\begin{theorem*}[cf. Theorem \ref{thm:class12,2}]
Let $p=3$, $n=2$, and $\lambda=4$.
Up to equivalence, a pair $(Q_1, Q_2)$ of polynomials satisfies the conditions of Proposition \ref{prop:Pagot} (and hence determines a space $L_{12,2}$) if, and only if, it is of the form
\[Q_{1} =a(X^4+(a^4-a^2-1)X^2+a^8)\] 
\[Q_{2}=X^4-(a^4+a^2-1)X^2+1,\]
for some $a \in k$ such that $a^2 \notin \F_3$.
\end{theorem*}

\begin{theorem*}[cf. Theorem \ref{thm:class15,2}]
Let $p=3$, $n=2$, and $\lambda=5$.
Up to equivalence and Frobenius equivalence, a pair ($Q_1, Q_2$) of polynomials satisfies the conditions of Proposition \ref{prop:Pagot} (and hence determines a space $L_{15,2}$) if, and only if, it is either of the form

\[Q_{1} =(\mu^2-\mu-1)X^5+X^3-(\mu^2-\mu-1)X^2-\mu X\] 
\[Q_{2}=-\mu X^5 +(\mu^2+\mu-1)X^3+(\mu^2+\mu)X^2 +(\mu^2-1)X-(\mu^2+\mu+1),\]

where $\mu \in \F_{27}$ is such that $\mu^3-\mu+1=0$, or of the form

\[Q_{1} =a(X^5-X^3-X^2+aX-(a+1))\] 
\[Q_{2}=(-a-1)(X^5-(a+1)X^3+(a+1)X^2+X+a),\]

where $a\in \F_9$ is such that $a^2+1=0$.
\end{theorem*}

In particular, we have infinite equivalence classes of spaces $L_{12,2}$ and only finitely many equivalence classes of spaces $L_{15,2}$.
Finding a geometric explanation of this phenomenon would be very interesting.

In the case $p>3$ the above simplifications are not enough to conclude. 
However, a finer strategy can be employed to obtain a classification of all spaces $L_{4 p,2}$ for all prime $p$'s, even though it requires more work.
In this case Theorem \ref{thm:generic} allows us to consider only the cases $p=3$ (discussed above) and $p=5,7,11$, which can be solved through a combination of elementary arguments and the support of computational methods, finally yielding the following result.
\begin{theorem*}[cf. Theorem \ref{thm:classL4p,2}]
Let $p$ be a prime number and let $\Omega$ be a space $L_{4p,2}$.
Then $p \in \{2,3,5\}$ and if $p=5$ then we obtain a full classification of spaces $L_{20,2}$.
\end{theorem*}
The proof of this result consists in splitting the analysis of the coefficients of the polynomials $Q_1, Q_2$ into three subcases.
In each of these subcases we show that $Q_1$ and $Q_2$ are necessarily biquadratic polynomials.
Then, either by solving the polynomial system given by Proposition \ref{prop:Pagot} via Grobner bases computations or, where possible, by explicit computing certain weighted power sums in the zeroes of $Q_1 + jQ_2$ and showing that they must satisfy a certain relation, we find a contradiction in all cases unless $p=5$ and $\Omega$ is of a very special kind (the one considered in Section \ref{sec:standard}).\\

We then turn our attention to the case of spaces $L_{\lambda p^{n-1},n}$.
The crucial new tools that we need in this context are Moore determinants and Dickson invariants.
The \emph{Moore determinant} $\Moore{n}{a}$ of a $n$-tuple of elements $\ul{a}:= (a_1, \dots, a_n)$ in a field of characteristic $p$ is defined as the determinant of the associated Moore matrix, namely we have
\[\Moore{n}{a}:= \begin{vmatrix}
a_1 & a_2 & \dots & a_n \\
a_1^p & a_2^p & \dots & a_n^p \\
\vdots & \vdots & \dots & \vdots \\
a_1^{p^{n-1}} & a_2^{p^{n-1}} & \dots & a_n^{p^{n-1}} \\
\end{vmatrix}.  \]
The Moore determinant is a basic object in arithmetic in characteristic $p$, due to its relationship with the theory of additive polynomials (cf. \cite[Chapter 1]{Goss96}).
In this paper, we apply a mix of classical and recent results (obtained in \cite{FresnelMatignon23}) on Moore determinants to describe effectively the differential forms belonging to a space $L_{\lambda p^{n-1},n}$.
Moore determinants are also important to define the Dickson invariants $c_{n,i} \in \F_p[X_1, \dots, X_n]$.
More precisely, for a fixed $n$ and $0\leq i \leq n$, one defines these by the formula
\[ \frac{\Delta_{n+1}(\ul{X},T)}{\Moore{n}{X}} = \sum_{i=0}^n (-1)^{n-i}c_{n,i}T^{p^i}.\]
Dickson invariants play a central role in the invariant theory of finite groups and were defined by Dickson \cite{Dickson11} to prove that the invariant algebra $\F_p[X_1, \dots, X_n]^{GL_n(\F_p)}$ is a polynomial algebra generated by the $c_{n,i}$'s.
Our exposition of Moore determinants and Dickson invariants is self-interested and self-contained: we introduce all and only the results we need in the dedicated Appendix \ref{app:Moore} and include there further references for the keen reader.

When applied to spaces $L_{\lambda p^{n-1},n}$, these techniques prove to be very fruitful, allowing to prove a new condition for the existence of spaces $L_{\lambda p^{n-1},n}$ which for $n=2$ boils down to Proposition \ref{prop:Pagot}.
\begin{theorem*}[combining Theorems \ref{thm:Pagotn0}, \ref{thm:Pagotn} and \ref{prop:Dickson}]
Let $n\geq 2$ and let $\Omega$ be a $n$-dimensional $\F_p$-vector space of differential forms on $\PP^1_k$, generated by elements $\omega_1, \dots, \omega_n$.
Then $\Omega$ is a space $L_{\lambda p^{n-1},n}$ if, and only if, there exist a $n$-tuple of polynomials $\ul{Q}:= (Q_1, \dots, Q_n) \in k[X]^n$ of degree $\lambda$ such that
\begin{itemize}
\item The leading coefficients $q_1, \dots, q_n$ of $Q_1, \dots, Q_n$ satisfy $\Moore{n}{q} \neq 0$.
\item We have
\[\omega_i = \frac{(-1)^{i-1}\Delta_{n-1}(\ul{\widehat{Q_i}})}{\Moore{n}{Q}} dX\] for every $i=1, \dots, n$, where $\ul{\widehat{Q_i}}$ denotes the $n-1$-tuple obtained by removing $Q_i$ from $\ul{Q}$.
\item The $n-1$th Dickson invariant evaluated at the $n$-tuple $\ul{Q}$ satisfies the equation \[c_{n,n-1}(\ul{Q})^{(p-1)}=-1,\]
where the exponent $(p-1)$ denotes the $p-1$-th derivative.
\end{itemize}
\end{theorem*}

This theorem allows us to approach the classification of spaces $L_{\lambda p^{n-1},n}$ from a better viewpoint. 
Namely, we can start with a $n$-tuple of polynomials $\ul{Q}$ and then check if this satisfies the conditions of the theorem.
This approach leads to advances in the classification of spaces $L_{\lambda p^{n-1},n}$ thanks to the fact that the condition $c_{n,n-1}(\ul{Q})^{(p-1)}=-1$ is equivalent to a condition involving the products of two determinants (see Proposition \ref{prop:det}).
In this way, we can achieve new results in the case $p=2$ (for every $n$ and $\lambda$) and in the case $\lambda=1$ (for every $n$ and $p\neq 2$).
These are the next two theorems.

\begin{theorem*}[cf. Theorem \ref{thm:p=2generaln} and Corollary \ref{cor:WRp=2}]
Let $p=2$ and $\lambda$ be either even or $\lambda \equiv 1 \mod (2^n-2)$.
Then there exist infinitely many equivalence classes of spaces $L_{\lambda 2^{n-1}, n}$.
\end{theorem*}

\begin{theorem*}[Theorem \ref{thm:lambda=1}]
Let $p>2$. Then, there exist no space $L_{p^{n-1},n}$.
\end{theorem*}

Then, we turn our attention to the examples of spaces $L_{\lambda p^{n-1},n}$ for $n\geq 3$ that were known in the literature prior to the present paper.
We remark that all these spaces share a very special structure: each of them is an \'etale pullback of a space $L_{(p-1)p^{n-1},n}$ whose set of poles is the set of nonzero elements in a $n$-dimensional $\F_p$-subvector space of $k$.
We call \emph{standard} any space $L_{(p-1)p^{n-1},n}$ whose poles satisfy the property above, and we apply Theorems \ref{thm:Pagotn0} and \ref{thm:Pagotn}  in the special case of standard spaces, shedding new light on their arithmetic properties.
As a result, we are able to characterize the proper subspaces of standard spaces as \'etale pullbacks of standard spaces of lower dimension.

\begin{proposition*}[cf. Proposition \ref{prop:standardsubsp}]
Let $\Omega$ be a standard $L_{(p-1) p^{n-1},n}$ space and let $\widetilde{\Omega}$ be a proper $r$-dimensional subspace of $\Omega$.
Then, there exists a standard $L_{(p-1)  p^{r-1},r}$ space $\Omega'$ and a degree $p^{n-r}$ morphism $\sigma: \PP^1_k \to \PP^1_k$ ramified only at $\infty$ such that $\widetilde{\Omega}$ is the pullback of $\Omega'$ under $\sigma$.
\end{proposition*}
When $p=2$, we can show that all standard $L_{2^{n-1},n}$ spaces and their étale pullbacks arise from the construction 
of Section \ref{subsec:applyingPagotn}.

Finally, we remark that the spaces $L_{12,2}$ and $L_{15,2}$ discovered in this paper are, to our knowledge, the first known examples of non-standard spaces $L_{\lambda p^{n-1},n}$ with $p\neq 2$ that are not equivalent to \'etale pullbacks of standard spaces.
Since $p=3$ in these examples, their \'etale pullbacks generate examples of non-standard spaces $L_{36d,2}$ and $L_{45d,2}$ for every positive integer $d$.
The spaces $L_{45d,2}$ when $d$ is odd can not be equivalent to \'etale pullbacks of standard spaces, by a simple argument of pole counts, and therefore we have an infinite class of examples not arising from standard spaces.
We don't know to what extent this generalizes to other values of $p$ and $n$.
More specifically, the following questions remain open:
\begin{itemize}
\item Are there spaces $L_{\lambda p,2}$ for $p \geq 5$, that are not \'etale pullbacks of standard spaces?
\item Are there spaces $L_{\lambda p^{n-1},n}$ for $n \geq 3$, that are not \'etale pullbacks of standard spaces?
\end{itemize}

Answers to these questions would result in great progress in the understanding the role of \'etale pullbacks in generating examples of spaces $L_{\lambda p^{n-1},n}$, and more generally in the structure of these spaces when $p, \lambda, n$ vary.

\subsection*{Structure of the paper}

In Section 2, we present known results and useful constructions on $L_{\lambda p^{n-1},n}$ that are used in all the other sections.
In Section 3, we recall an important characterization of spaces $L_{\lambda p, 2}$, due to Pagot (Proposition \ref{prop:Pagot}), and we prove the non existence of spaces $L_{\lambda p, 2}$ when $p > 3 \lambda$.
In Section 4, we generalize Proposition \ref{prop:Pagot} to the case of spaces $L_{\lambda p^{n-1},n}$ and we apply this to the case $p=2$ to construct our new examples of spaces $L_{\lambda 2^{n-1},n}$.
In Section 5, we define and study \emph{standard $L_{\lambda p^{n-1},n}$-spaces}.
In Section 6, we fix $p=3$ and $n=2$ and provide a complete classification of spaces $L_{12, 2}$ and $L_{15, 2}$.
In Section 7, we fix $\lambda=4$ and $n=2$ and we deal with the outstanding cases in the classification of spaces $L_{4p,2}$.
Finally, in Appendix A we collect all the results on Moore determinants that are used throughout the paper (mostly in Sections 4 and 5).

\subsection*{Notation and conventions}
Let $k$ be an algebraically closed field of characteristic $p>0$.
A meromorphic differential form on $\PP^1_k$ can be written as $\omega=f(X)dX$ for $f(X) \in k(X)$, the field of rational fractions in one variable over $k$.
Given a point $a \in \PP^1_k(k)=k \cup \{\infty \}$ and a differential form $\omega=f(X)dX$ on $\PP^1_k$, the order of $\omega$ at $a$ is equal to the order $\ord_a(f(X))$ if $a \neq \infty$, and it is equal to $\ord_0(f(\frac{1}{X})) - 2$ if $a = \infty$.
When $\ord_a(\omega)>0$, the point $a$ is a \emph{zero} of $\omega$, and when $\ord_a(\omega)<0$, the point $a$ is a \emph{pole} of $\omega$.

A differential form on $\PP^1_k$ is called \emph{logarithmic} if it is of the form $\omega=\frac{dF}{F}$ for some $F \in k(X)$.
We usually denote by $\Omega$ a space $L_{m+1,n}$ and by $\{ \omega_1,\dots, \omega_n \}$ a basis for this space.
For a given subset $S \subset \Omega$, we denote by $\calP(S)$ the subset of $k$ consisting of elements that are poles of at least a non-zero differential form in $S$.
If $\omega \in \Omega - \{0\}$, we write $\calP(\omega)$ for the set of poles of $\omega$.
One deduces from the definition that the set $\calP(\omega)$ consists of $m+1$ simple poles.
Finally, given a finite set of logarithmic differential forms $\{\omega_1, \dots, \omega_n\}$, we denote by $\langle \omega_1, \dots, \omega_n \rangle_{\Fp}$ the $\Fp$-vector space that they generate.
Most of the times, this will not be a space $L_{m+1,n}$, but it will be one under certain predetermined conditions.

\subsection*{Acknowledgments}
Thanks to Jean Fresnel's careful reading and insightful comments, this paper is better than it otherwise would have been. We dedicate it to his memory.

\section{Preliminaries}
In this section, we collect the preliminary results later used to show existence, non-existence and classification results of spaces $L_{m+1,n}$.
We first recall known results on the number of poles in such spaces, as well as proving new lemmas on the combinatorics of the arrangements of such poles.
Then we introduce useful constructions: Frobenius twists and \'etale pullbacks of spaces $L_{m+1,n}$.
Finally, we briefly discuss known results in the case $n=1$.
The main result in this section that was not previously known is Corollary \ref{coro:poles}, stating that a space $L_{m+1,n}$ is characterized by its set of poles.

\subsection{The Jacobson-Cartier condition}
Let us recall here the \emph{Jacobson-Cartier} condition for verifying that a meromorphic differential form on $\PP^1_k$ is logarithmic.
Let $\omega = f(X) dX \in \Omega(k(X))$ be such a form.
Since $k$ is perfect, we have that $k(X)=\oplus_{i=0}^{p-1} k(X)^p X^i$, and hence a unique writing 
\[f(X)=\oplus_{i=0}^{p-1} f_i(X)^p X^i.\]
Note that the rational function $f_{p-1}$ is invariant by translation by any element $a\in k$.
In fact, we can also write $f(X)=\oplus_{i=0}^{p-1} g_{i}(X)^p (X-a)^i$, and by comparing the coefficients we find that $f_{p-1}=g_{p-1}$.
It is a classical result that $\omega$ is logarithmic if, and only if, $f(X)=f_{p-1}(X)$\footnote{In the case of curves this result is due to Jacobson. The study of the correspondence $f \to f_{p-1}$ as an operation over differential forms over curves is addressed in subsequent work of Tate, that was later generalized to higher dimensions by Cartier. Because of this, such correspondence is often known under the name of \emph{Cartier operator}. We therefore deem it reasonable to refer to the result on curves as to the ``Jacobson-Cartier'' condition. We refer to \cite[\S 10, \S 11]{Serre58} for a discussion of the topic that brings all these different perspectives together.}.
In our case, that of differential forms over the projective line, this fact has an elementary proof, that we provide below.

\begin{proposition}\label{prop:Cartier}
Let $\omega = f(X) dX \in \Omega(k(X))$ with $f(X)=\oplus_{i=0}^{p-1} f_i(X)^p X^i$ be a non-zero differential form.
Then $\omega=\frac{dF}{F}$ for some $F\in k(X)$ if, and only if, $f(X)=f_{p-1}(X)$.
\end{proposition}
\begin{proof}
Let $\{x_1,\dots x_r\}$ be the set of poles of $\omega$, which is non-empty because $\omega \neq 0$.
To have $\omega=\frac{dF}{F}$ it is necessary and sufficient that the $x_i$'s are simple and their residues are in $\F_p^\times$. 
When this is the case, $\omega=\sum_{i=1}^{r} \frac{a_i}{X-x_i}dX$.
So, assuming that $\omega$ is logarithmic (and hence $a_i \in \F_p^\times$), we find that
\[f_{p-1}(X)=\left(\sum_{i=1}^{r} \frac{a_i}{X-x_i} \right)_{p-1} = \sum_{i=1}^{r} \left(\frac{a_i}{X-x_i} \right)_{p-1} = \sum_{i=1}^{r} \frac{a_i}{X-x_i}= f(X).\] 
Conversely, suppose that $f(X)=f_{p-1}(X)$.
Then we can consider the pole expansion for the meromophic function $f(X)$, namely:
\[f(X) = E(X)+\sum_{i=1}^r g_{x_i}(X), \; \mbox{where} \; E(X)\in k[X]\; \mbox{and} \; g_{x_i}(X):= \sum_{j>0} \frac{a_{ij}^p}{(X-x_i)^j} \;\mbox{with}\; a_{ij} \in k.\]
Our assumption that $f(X)=f_{p-1}(X)$ can be rewritten as 
\[E(X)=E_{p-1}(X) \;\; \mbox{and} \;\; g_{x_i}(X)=(g_{x_i})_{p-1}(X).\]
So, if we write $E(X)=\sum_{i=0}^{p-1} E_i(X)^p X^i$ with $E_i(X) \in k[X]$ we have $\deg(E)=\max_i(p \deg E_i +i)$, so that $\deg(E)=\deg(E_{p-1})$ implies $E=0$.
Similarly, by writing $\frac{a_{ij}^p}{(X-x_i)^j}= \frac{a_{ij}^p}{(X-x_i)^{qp+s}}$ for suitable $q$ and $0\leq s < p$, we have that $\frac{a_{ij}^p}{(X-x_i)^j}=\frac{a_{ij}^p(X-x_i)^{p-s}}{(X-x_i)^{(q+1)p}}.$ 
If $s\neq 1$, then $\left(\frac{a_{ij}^p}{(X-x_i)^j}\right)_{p-1}=0$, hence we just need to consider the case $s=1$, where we have $\left(\frac{a_{ij}^p}{(X-x_i)^j}\right)_{p-1}=\frac{a_{ij}}{(X-x_i)^{1+\frac{j-1}{p}}}$.
The condition $g_{x_i}(X)=(g_{x_i})_{p-1}(X)$ implies then that $a_{ij}^p=a_{i(1+(j-1)p)}$ which is to say that $a_{i1}\in \F_p$ and $a_{ij}=0$ for $j>1$.
Summarizing, we find that $E(X)=0$ and $g_{x_i}(X)=\frac{a_{i1}}{X-x_i}$ with $a_{i1}\in \F_p$ for every $i$, which is to say that $\omega$ is a logarithmic differential form.
\end{proof}
\begin{corollary}\label{cor:1/P}
Let $\omega=f(X)dX \in \Omega(k(X))$ be such that $f(X)=\frac{1}{P(X)}$ with $P(X)\in k[X]$.
Then $\omega=\frac{dF}{F}$ for some $F\in k(X)$ if, and only if, the $p-1$-th derivative $\left({P}^{p-1}\right)^{(p-1)}$ of the polynomial $P^{p-1}$ is equal to $-1$.
\end{corollary}
\begin{proof}
By Proposition \ref{prop:Cartier}, we have that $\omega=\frac{dF}{F}$ if, and only if, $f=f_{p-1}$.
Since the $p-1$-th derivative of $f$ satisfies $f^{(p-1)}=(p-1)!f_{p-1}^p=-f_{p-1}^p$, then it is equivalent to require that $f^{(p-1)}=-f^p$.
In the case where $f(X)=\frac{1}{P(X)}$, we find $f^{(p-1)}=\left( \frac{P^{p-1}}{P^p}\right)^{(p-1)}=\frac{{(P^{p-1})}^{(p-1)}}{P^p}.$
On the other hand, $-f^p=-\frac{1}{P^p}$, and so the condition $f=f_{p-1}$ becomes $\left({P}^{p-1}\right)^{(p-1)}=-1$.
\end{proof}

\subsection{The combinatorics of poles of a space $L_{m+1,n}$}\label{ssec:combpoles}
The condition that $n$ logarithmic differential forms $\omega_1, \dots, \omega_n$ generate $\Omega$ a space $L_{m+1,n}$ put combinatorial restrictions on the set $\calP(\Omega)$ consisting of all the poles of the elements of $\Omega$. 
In the first part of this section, we rewrite some of the known restrictions using our notation and including alternative proofs for the reader's convenience.
The first is a lemma proved in \cite[Lemme 5 and Lemme 6]{Pagot02}.
\begin{lemma}\label{lem:Pagotpoles}
Let $\Omega$ be a space $L_{m+1,n}$. Then, there is an integer $\lambda>0$ such that $m+1=\lambda p^{n-1}$. Moreover, $|\calP(\Omega)|=\lambda \frac{p^{n}-1}{p-1}$. \qed
\end{lemma}

Thanks to this result, in the rest of the paper we can restrict to study spaces $L_{\lambda p^{n-1},n}$ for different values of $p$ and $\lambda$, knowing that this hypothesis does not constitute a loss of generality.

Let $\Omega$ be a space $L_{\lambda p^{n-1},n}$, fix a basis $\{\omega_1, \dots, \omega_n\}$ of $\Omega$ and pick a pole $x\in \calP(\Omega)$.
We denote by $h_i$ the residue of $\omega_i$ at $x$ for $i=1, \dots, n$, setting $h_i=0$ if $x$ is not a pole of $\omega_i$.
Let $\omega=a_1\omega_1+\dots+a_n\omega_n \in \Omega$ with $a_i \in \F_p$.
Then we have that $x \in \calP(\omega)$ if, and only if, $a_1h_1+\dots+a_nh_n \neq 0$.
As a result, the set $\mathfrak{H}(x)=\{\sum_{i=1}^n a_i \omega_i \in \Omega : \sum_{i=1}^n a_i h_i =0\}$ is a hyperplane of $\Omega$ consisting of all the differential forms that do not have $x$ as a pole.
Conversely, given a hyperplane $H \subset \Omega$, we define $X_H := \calP(\Omega) - \calP(H)$, the subset of $\calP(\Omega)$ consisting of all poles that do not occur as poles of any element in $H$. 
If we denote by $\calH(\Omega)$ the set of hyperplanes of $\Omega$, the correspondence $x \mapsto \frakH(x)$ defines a function $\frakH:\calP(\Omega) \to \calH(\Omega)$ with the property that $\frakH^{-1}(H)=X_H$ for every $H \in \calH(\Omega)$.

\begin{lemma}\label{lem:hyperplanes}
We have that $|X_H|=\lambda$ for every $H \in \mathcal{H}(\Omega)$.
As a result, the map $\frakH$ is surjective and it is injective if, and only if, $\lambda=1$.
\end{lemma}
\begin{proof}
Note that a hyperplane $H\in \mathcal{H}(\Omega)$ is a space $L_{(\lambda p) p^{n-2},n-1}$.
By Lemma \ref{lem:Pagotpoles} we have that $|\calP(\Omega)|=\lambda \frac{p^n-1}{p-1}$, $|\calP(H)|=\lambda p \frac{p^{n-1}-1}{p-1}$, and we know that $\calP(H)\subset \calP(\Omega)$.
From this it follows that
\[|X_H|= |\calP(\Omega) - \calP(H)|= |\calP(\Omega)| - |\calP(H)| =\lambda \frac{p^n-1}{p-1} - \lambda p \frac{p^{n-1}-1}{p-1} = \lambda \frac{p^n-1-p^n+p}{p-1}=\lambda.\]
\end{proof}

The result of Lemma \ref{lem:hyperplanes} can be rephrased by observing that $\calP(\Omega)= \cup_H X_H$ is a union of sets of cardinality $\lambda$, indexed by the $\frac{p^n-1}{p-1}$ hyperplanes of $\Omega$.
In the light of \ref{lem:Pagotpoles}, we see that this is in fact a disjoint union.

The following Corollary is essentially equivalent to \cite[Lemme 6]{Pagot02}. 
We include here a proof that uses the notation above for the reader's convenience.

\begin{corollary}\label{cor:lambdapn}
Let $\Omega$ be a space $L_{\lambda p^{n-1},n}$ and let $\{ \omega_1, \dots, \omega_n \}$ be a basis of $\Omega$.
Then, for every $1 \leq r\leq n$ we have that 
$|\calP(\omega_1) \cap \dots \cap \calP(\omega_r)|=\lambda (p-1)^{r-1} p^{n-r}$.
\end{corollary}
\begin{proof}
We begin by remarking that, for every $\omega\in \Omega - \{0\}$ and $H \in \calH(\Omega)$ either $\frakH^{-1}(H) \subset \calP(\omega)$ or $\frakH^{-1}(H) \cap \calP(\omega)=\emptyset$.
As a result, $\calP(\omega_1) \cap \dots \cap \calP(\omega_r)$ is a union of sets of the form $X_H$, and therefore
$\calP(\omega_1) \cap \dots \cap \calP(\omega_r)=\frakH^{-1}(\frakH( \calP(\omega_1) \cap \dots \cap \calP(\omega_r)))$. 
We then have 
\[|\calP(\omega_1) \cap \dots \cap \calP(\omega_r)|=\lambda |\frakH \big( \calP(\omega_1) \cap \dots \cap \calP(\omega_r) \big)|.\]
We can then conclude with a hyperplane-counting argument: we have that $\frakH(\calP(\omega_1))$ is the set of hyperplanes of $\Omega$ not containing $\omega_1$.
More in general, since $\calP(\omega_1) \cap \dots \cap \calP(\omega_r)$ is a union of sets of the form $X_H$, we have that $\frakH(\calP(\omega_1) \cap \dots \cap \calP(\omega_r))=\frakH(\calP(\omega_1)) \cap \dots \cap \frakH(\calP(\omega_r))$.
This latter is the set of hyperplanes of $\Omega$ not containing $\omega_i$ for any $i\in \{1, \dots, r\}$ and so its cardinality is $\frac{p^{n-r} (p-1)^r}{p-1}$.
As a result, we also have that 
\[|\calP(\omega_1) \cap \dots \cap \calP(\omega_r)|=\lambda p^{n-r}(p-1)^{r-1}.\]
\end{proof}

When $n=2$, the result above specializes to the following.

\begin{corollary}\label{cor:lambdap}
Let $\Omega$ be a space $L_{\lambda p,2}$, and let $\omega, \omega' \in \Omega -  \{0\}$.
Then, 
\begin{enumerate}
\item$\langle \omega \rangle_{\Fp} = \langle \omega' \rangle_{\Fp}$ if, and only if, $|\calP(\omega)\cap \calP(\omega')|=\lambda p$
\item $\langle \omega, \omega' \rangle_{\Fp} = \Omega$ if, and only if, 
$|\calP(\omega)\cap \calP(\omega')|=\lambda (p-1)$.
\end{enumerate}
\qed
\end{corollary}

\begin{remark}
     The results above can be interpreted as special cases of a theorem by Yang (\cite[Theorem 3.12]{Yang25}), which analyzes the combinatorics of the branch loci of the lifts of $(\Z/p\Z)^n$-extensions $k[\![z]\!]/k[\![t]\!]$.
     For every $\Omega$ a space $L_{m+1,n}$ with set of poles $\calP(\Omega)$ we know by \cite[Th\'eor\`eme 11]{Pagot02} that there is a lift $R[\![Z]\!]/R[\![T]\!]$ of a $(\Z/p\Z)^n$-extension $k[\![z]\!]/k[\![t]\!]$ whose branch locus specializes to $\calP(\Omega)$ and such that the specializations of the branch loci of its $p$-cyclic subextensions are all of the form $\calP(\omega)$ for $\omega\in \Omega$. Under this correspondence, Yang's criteria in the equidistant case are equivalent to Lemma \ref{lem:Pagotpoles} and Corollary \ref{cor:lambdap}.
\end{remark}

Let us now show that the set of poles of a space $L_{m+1,n}$ completely characterizes such a space, starting with the case $n=1$.
\begin{lemma}\label{lem:poles1}
Let $\omega, \omega'$ be logarithmic differential forms on $\PP^1_k$ with a unique zero at $\infty$ of order $m-1$, and let $\calP(\omega)$ and $\calP(\omega')$ be the respective set of poles.
Then, $\calP(\omega)=\calP(\omega')$ if, and only if $\langle \omega \rangle_{\F_p} = \langle \omega' \rangle_{\F_p}$.
\end{lemma}
\begin{proof}
The non-trivial part is to prove that $\calP(\omega)=\calP(\omega')$ implies that the two forms generate the same $\F_p$-vector space.
Let $\calP(\omega)=\calP(\omega')=\{x_0,\dots,x_m\}$ and let us show that there exists $c\in \F_p^\times$ such that $\omega'=c\omega$. 
As these differential forms are logarithmic with no zeroes outside $\infty$, we have the unique writings $\omega=\sum_{i=0}^{m} \frac{a_i}{X-x_i}dX$ and $\omega'=\sum_{i=0}^{m} \frac{b_i}{X-x_i}dX$, with $a_i, b_i \in \F_p^\times$.\\

Suppose by contradiction that there is no $c\in \F_p^\times$ such that $a_i=c b_i$ for every $i$.
Then there is a $j\in \F_p^\times$ such that $a_i+jb_i=0$ for some but not all $i$'s.
Then, the form $\omega+j \omega'$ is a non-zero logarithmic differential form with a zero of order at least $m-1$ at infinity, and at the same time it has at most $m$ simple poles.
This is not possible, since the degree of any (non-zero) meromorphic differential form on $\PP^1_k$ is $-2$.
\end{proof}


\begin{proposition}\label{prop:poles}
Let $n\geq 2$, let $\Omega$ be a space $L_{\lambda p^{n-1},n}$ and let $\omega'$ be a logarithmic differential form having a unique zero at infinity of order $\lambda p^{n-1} - 2$ and such that $\calP(\omega') \subset \calP(\Omega)$.
Then $\omega' \in \Omega$.
\end{proposition}
\begin{proof}
Suppose by contradiction that $\omega' \not \in \Omega$. 
Let $\omega \in \Omega - \{0\}$ be a non-zero differential form.
We begin our argument with a proof by contradiction that
\begin{equation}\label{eq:interpoles}
 |\calP(\omega) \cap \calP(\omega')| \leq \lambda(p^{n-1}-p^{n-2}).
\end{equation}

For $j=1,\dots, p-1$, we set $Y_j:=\{x \in \calP(\omega)\cap \calP(\omega') : \res_x(\omega') = j\cdot \res_x(\omega) \}$ in such a way that 
\[\calP(\omega)\cap \calP(\omega') = \bigcup_{j=1}^{p-1} Y_j\]
is a disjoint union. 
If we assume by contradiction that $ |\calP(\omega) \cap \calP(\omega')| >\lambda(p^{n-1}-p^{n-2})$, then there exists at least a value of $j$ for which $|Y_j|>\frac{\lambda (p^{n-1}-p^{n-2})}{p-1}=\lambda p^{n-2}.$
For such a $j$, we set $\omega_j:=\omega'-j \omega$.

Since $\omega' \not \in \Omega$, we have that $\omega_j \neq 0$ and then that $\omega_j$ is a logarithmic differential form with $\calP(\omega_j) \cap Y_j = \emptyset$.
Moreover, since both $\omega$ and $\omega'$ have a zero of order $\lambda p^{n-1} -2$ at $\infty$, then the order of $\infty$ as a zero of $\omega_j$ is at least $\lambda p^{n-1} -2$, resulting in $|\calP(\omega_j)| \geq \lambda p^{n-1}$, because all the poles are simple.
By construction, we also have that $\calP(\omega_j)=\calP(\omega) \cup \calP(\omega') - Y_j$.
Combining this information, we estimate the cardinality $|\calP(\omega_j)|$ as
\begin{align*}
\lambda p^{n-1}& \leq |\calP(\omega_j)| = |\calP(\omega)| + |\calP(\omega')| - |\calP(\omega) \cap \calP(\omega')|- |Y_j|< 2 \lambda p^{n-1} - \lambda(p^{n-1}-p^{n-2}) - \lambda p^{n-2} \\
& = \lambda p^{n-1},
\end{align*}
leading to a contradiction and proving the validity of the inequality (\ref{eq:interpoles}).

We then show that inequality (\ref{eq:interpoles}) can not hold for every $\omega \in \Omega$ using the hypothesis that $\omega' \not \in \Omega$.
First of all, we fix a basis $\omega_1, \dots, \omega_n$ of $\Omega$, we consider the dual basis $\omega_1^\star, \dots, \omega_n^\star$ of $\Omega^\star$
and we index every element of $\Omega$ as $\omega_{\ul{a}} := \sum a_i \omega_i$ for every $\ul{a} \in \F_p^n$.
In this way, we have that $\omega' \neq \omega_{\ul{a}}$ for every $\ul{a} \in \F_p^n$.
For all $\ul{\epsilon}:= \epsilon_1 \omega_1^\star+\dots+ \epsilon_n \omega_n^\star \in \PP(\Omega^\star)$, we set
\[H_{\ul{\epsilon}}:=\{\omega_{\ul{a}} \in \Omega : \ul{\epsilon}(\omega_{\ul{a}}) = 0\},\]  
\[X_{\ul{\epsilon}}:= X_{H_{\ul{\epsilon}}} = \{x \in \calP(\Omega):  \res_x(\omega)=0 \; \forall \; \omega \in H_{\ul{\epsilon}}\}\] and 
\[N_{\ul{\epsilon}}:=|\calP(\omega') \cap X_{\ul{\epsilon}}|.\]
Since $\calP(\omega')\subset \calP(\Omega)$, we have that $\calP(\omega')= \bigsqcup_{\ul{\epsilon}} (\calP(\omega') \cap X_\epsilon)$ and hence that $\sum_{\ul{\epsilon}} N_{\ul{\epsilon}}=\lambda p^{n-1}.$

We note that, for every $\ul{a} \in \PP(\F_p^n)$
the set $\calP(\omega_{\ul{a}})$ is the union
\[\calP(\omega_{\ul{a}}) = \bigcup_{\substack{\ul{\epsilon} \in \PP(\Omega^\star)\\  \ul{\epsilon}(\omega_{\ul{a}})\neq 0}} X_{\ul{\epsilon}}, \]
and in particular it depends only on the class $[\ul{a}] \in \PP(\F_p^n)$.
To conclude the proof, we count $\sum_{[\ul{a}] \in \PP(\F_p^n)} \big|\calP(\omega_{\ul{a}})\cap \calP(\omega')\big|$ in two different ways:

On the one hand, we have
\begin{align*}
    \sum_{[\ul{a}] \in \PP(\F_p^n)} \big|\calP(\omega_{\ul{a}})\cap \calP(\omega')\big| &= \sum_{[\ul{a}] \in \PP(\F_p^n)}\sum_{\substack{\ul{\epsilon} \in \PP(\Omega^\star)\\  \ul{\epsilon}(\omega_{\ul{a}})\neq 0}} N_{\ul{\epsilon}}  = 
    \sum_{\ul{\epsilon} \in \PP(\Omega^\star)}\sum_{\substack{[\ul{a}] \in \PP(\F_p^n)\\ \ul{\epsilon}(\omega_{\ul{a}})\neq 0}} N_{\ul{\epsilon}} \\
    & = \sum_{\ul{\epsilon} \in \PP(\Omega^\star)} p^{n-1} N_{\ul{\epsilon}} = \lambda p^{n-1} p^{n-1} = \lambda p^{2n-2}.
\end{align*}

On the other hand, since $\omega' \neq \omega_{\ul{a}}$ for every $\ul{a} \in \PP(\F_p^n)$, we can apply the inequality (\ref{eq:interpoles}) with $\omega=\omega_{\ul{a}}$ for every $\ul{a} \in \PP(\F_p^n)$ to get that
\[ \sum_{[\ul{a}] \in \PP(\F_p^n)} \big|\calP(\omega_{\ul{a}})\cap \calP(\omega')\big| \leq \lambda (p^{n-1} - p^{n-2})(p^{n-1}+\dots+1)= \lambda( p^{2n-2}- p^{n-2}).\]

We obtain the contradiction $\lambda p^{2n-2} \leq \lambda( p^{2n-2}- p^{n-2})$ and we conclude that $\omega' \in \Omega$.

\end{proof}

From Lemma \ref{lem:poles1} and Proposition \ref{prop:poles} we deduce the very useful corollary that a space  $L_{\lambda p^{n-1},n}$ is characterized by its set of poles.
\begin{corollary}\label{coro:poles}
Let $n \geq 1$ and let $\Omega$ and $\Omega'$ be spaces $L_{\lambda p^{n-1},n}$. Then $\Omega=\Omega'$ if, and only if, $\calP(\Omega)=\calP(\Omega')$.
\end{corollary}

We conclude this section by establishing a notion of equivalence between two spaces $L_{\lambda p^{n-1},n}$, which will be employed in later sections.

\begin{definition}\label{defn:equiv}
    If $\Omega, \Omega'$ are two spaces $L_{\lambda p^{n-1},n}$, we say that they are \emph{equivalent} if there is an automorphism $\sigma \in \Aut_k(\PP^1_k)$ such that $\sigma(\infty)=\infty$ and $\Omega' = \sigma^\star \Omega$.
\end{definition}

An immediate consequence of Corollary \ref{coro:poles} is that $\Omega$ is equivalent to $\Omega'$ if, and only if, there exist $a\in k^\times, b\in k$ such that $\calP(\Omega')=a\calP(\Omega)+b$.

\subsection{Frobenius action and \'etale pullbacks}\label{sec:Frobetale}

Given a space $L_{\lambda p^{n-1}, n}$, there are two constructions that we can apply to construct more spaces.

The first construction exploits the action of the relative Frobenius on a space $L_{\lambda p^{n-1}, n}$. 
Recall that $\Omega(k(X))$ denotes the $k$-algebra of meromorphic differential forms on $\PP^1_k$. Consider the relative Frobenius operator $\Phi: \Omega(k(X)) \to \Omega(k(X))$ acting on the coefficients of a form by raising them to the power $p$:
\[\Phi\left(\frac{\sum a_i X^i}{\sum b_i X^i} dX\right) = \frac{\sum a_i^p X^i}{\sum b_i^p X^i} dX.\]

Then we have the following result:

\begin{lemma}
Let $\Omega$ be a space $L_{\lambda p^{n-1},n}$.
Then $\Phi(\Omega)$ is again a space $L_{\lambda p^{n-1},n}$. 
Moreover, for every choice of $p, \lambda$, and $n$, $\Phi$ is bijective when restricted to the set of spaces $L_{\lambda p^{n-1},n}$.
\end{lemma}
\begin{proof}
Let $\omega=\sum_{i=1}^{\lambda p^{n-1}} \frac{a_i}{X-x_i}dX \in \Omega$ with $a_i \in \F_p^\times$.
Then $\Phi(\omega)=\sum_{i=1}^{\lambda p^{n-1}} \frac{a_i}{X-x_i^p}dX$ is clearly logarithmic.
Moreover, it has a unique zero at $\infty$ because this condition is equivalent to the one given by equation (\ref{eq:zeroatinfty}).
Since $\F_p$-linearly independent forms are sent to linearly independent forms, we have that $\Phi(\Omega)$ is a space $L_{\lambda p^{n-1},n}$.

To show the bijectivity note that, since $k$ is algebraically closed, the Frobenius is an automorphism of $k$.
Its inverse induces the inverse of $\Phi$, which restricts naturally to the set of spaces $L_{\lambda p^{n-1},n}$ for the reasons above.
\end{proof}

Concretely, the relative Frobenius acts on the points of $\PP^1_k$ by raising them to the $p$-th power.
If $\Omega$ is a space $L_{\lambda p^{n-1}, n}$, then one gets the poles of $\Phi(\Omega)$ by raising to the $p$-th power the poles of $\Omega$.
This condition determines uniquely the space $\Phi(\Omega)$ thanks to Corollary \ref{coro:poles}.\\

The second construction exploits the properties of finite \'etale covers of the affine line in characteristic $p>0$.
More precisely, we fix $d>0$ and we recall that in this setting a finite \'etale morphism $\AA^1_k \to \AA^1_k$ of degree $dp$ is induced by a map $k[X]\to k[X]$ sending $X$ to a polynomial of the form $\gamma X + T(X^p)$ with $\gamma \in k^\times$ and $T\in k[X]$ a polynomial of degree $d$.
This is equivalent to ask that $X$ is sent to a polynomial whose derivative is a non-zero constant.
Such a morphism extends uniquely to a degree $dp$ cover $\phi: \PP^1_k \to \PP^1_k$ branched only over $\infty$, and we can consider the pullback map $\phi^\star: \Omega(k(X)) \to \Omega(k(X))$.

\begin{lemma}\label{lem:pullbacketale}
Let $\phi: \PP^1_k \to \PP^1_k$ be the compactification of a finite \'etale morphism $\AA^1_k \to \AA^1_k$ of degree $dp$.
Let $\Omega=\langle \omega_1, \dots, \omega_n\rangle_{\F_p}$ be a space $L_{\lambda p^{n-1}, n}$.
Then, the $\F_p$-vector space generated by the differential forms $\phi^{\star}(\omega_1), \dots, \phi^{\star}(\omega_n)$ is a space $L_{d \lambda p^{n}, n}$, called the \emph{\'etale pullback} of $\Omega$ via $\phi$.
\end{lemma}
\begin{proof}
The restriction of $\phi$ to $\AA^1_k$ is induced by a polynomial $S(X)$ such that $S'(X)=\gamma \in k^\times$.
Set $Z=S(X)$ and $\omega_i= \frac{dX}{P_i(X)}=\frac{F_i'(X)}{F_i(X)}dX$. 
Then $\phi ^\star (\omega_i)=\frac{[F_i(S(X))]'}{F_i(S(X))}dX$ and hence is logarithmic. 
Moreover
\[ \phi^{\star}(\omega_i)= \frac{dZ}{P_i(Z)}=\frac{\gamma dX}{P_i(S(X))}.\]
and hence it has a unique zero of order $d \lambda p^{n}-2$ at infinity.
Finally, $\phi^{\star}$ is a linear operator, hence we have $\sum_i a_i \phi^{\star}(\omega_i)=\phi^{\star}\left(\sum_i a_i\omega_i\right)$. 
It follows that any $\F_p$-linear combination of $\phi^{\star}(\omega_i)$ is also a logarithmic differential form with a unique zero of order $d \lambda p^{n}-2$ at infinity.
\end{proof}

\subsection{Known results on spaces $L_{\lambda, 1}$}

It is easy to verify that every $\ell$-dimensional subspace of a space $L_{\lambda p^{n-1}, n}$ is a space $L_{\lambda p^{n-1}, \ell}$.
It is therefore useful to have results in the case $n=1$, as these can provide significant information for studying the higher dimensions too.
In this short section, we recall the known results in dimension one that will be used in the rest of the paper.

\begin{proposition}\label{prop:Llambda1} 
Let $\omega \in \Omega(k(X))$ be a differential form on the projective line $\mathbb{P}^1_k$.
The following conditions are equivalent:
\begin{enumerate}
\item The $\F_p$-vector space $\langle \omega \rangle_{\F_p}$ is a space $L_{\lambda, 1}$.
\item The differential form $\omega$ has precisely $\lambda$ distinct simple poles $x_1, \dots, x_\lambda$ and corresponding residues $a_1, \dots, a_\lambda$ in $\F_p^\times$ (which is equivalent to being logarithmic).
Moreover, these poles and residues satisfy the polynomial equations:
\begin{equation}\label{eq:zeroatinfty}
\displaystyle \sum_{i=1}^{\lambda} a_i x_i^k = 0\;\; \text{for} \; \; 0 \leq k \leq \lambda-2.
\end{equation}
\item
We can write $\omega=\frac{1}{P(X)} dX$ with $P(X)\in k[X]$ of degree $\lambda$ in such a way that the coefficient of $X^{p-1}$ in the polynomial $P^{p-1}$ is $1$ and that the coefficient of $X^{\mu p-1}$ in the polynomial $P^{p-1}$ vanishes for all $2 \leq \mu \leq \lambda + \floor{\frac{1-\lambda}{p}}$.
\end{enumerate}
\end{proposition}
\begin{proof}
Let us assume (i). 
Then $\omega$ is logarithmic and has a unique zero of order $\lambda-2$ at infinity.
We write $\omega=\sum_{i=1}^{\lambda} \frac{a_i}{X-x_i}dX$ and we introduce a change of variable $Z=\frac{1}{X}$.
Then,
\[\omega =  \sum_{i=1}^{\lambda} \frac{- a_i}{Z(1-x_i Z)}dZ = \sum_{i=1}^{\lambda} \left( \frac{- a_i}{Z} + \frac{-a_ix_i}{1-x_i Z} \right) dZ =\sum_{i=1}^{\lambda} \frac{- a_i x_i}{(1-x_i Z)}dZ,\]
the first equality resulting from partial fraction expansion, and the second arising from the fact that $\omega$ has no poles at infinity and then $\sum_{i=1}^\lambda a_i = 0$.  
The further condition that the zero at infinity of $\omega$ is of order $\lambda-2$ results in the equality
\[\sum_{i=1}^{\lambda} \frac{a_ix_i}{1-x_i Z} dZ = \frac{u Z^{\lambda-2}}{\prod_{i=1}^{\lambda} (1-x_i Z)}dZ,\]
for some $u\in k^\times$.
By developing the denominators in this equality in formal power series in the variable $Z$, one obtains the following equations: 
\begin{equation}\label{eq:zeroatinfty2}
\begin{cases}
\displaystyle \sum_{i=1}^{\lambda } a_i x_i^k = 0\;\; \text{for} \; \; 0 \leq k \leq \lambda-2,\\
\displaystyle \sum_{i=1}^{\lambda} a_i x_i^{\lambda-1} = u,   \\
\displaystyle \sum_{i=1}^{\lambda} a_i x_i^{\lambda+k-1} = u\cdot c_{k}(x_1, \dots, x_{\lambda})  \;\; \text{for} \; \; k \geq 1,
\end{cases}
\end{equation}
where $c_k$ is the $k$-th complete homogeneous symmetric polynomial.
The first line show that the system (\ref{eq:zeroatinfty}) is satisfied, while the second line ensures that the poles are distinct.

Conversely, assume $(ii)$.
Then, we have that $\omega$ is logarithmic thanks to the condition on the poles and residues.
Moreover, from the development in formal power series it follows that the equations in (\ref{eq:zeroatinfty}) imply that the zero at infinity is of order $\geq \lambda-2$.
Since there are precisely $\lambda$ distinct simple poles, this is enough to conclude that there are no other zeroes and that the order at infinity is precisely $\lambda-2$.\\

The equivalence between $(iii)$ and $(i)$ follows from Corollary \ref{cor:1/P}.
\end{proof}

\begin{remark}
The element $u \in k^\times$ appearing in Equation \eqref{eq:zeroatinfty2} can be expressed in terms of the polynomial $P$ appearing in condition $(iii)$ as $u=\frac{1}{\alpha}$, where $\alpha$ is the leading coefficient of $P$.
\end{remark}

It is easy to construct logarithmic differential forms $\omega$ satisfying the conditions of Proposition \ref{prop:Llambda1}, as the following examples show:

\begin{example}\label{ex:Llambda1}
Let $\lambda\in \Z$ with $\lambda>1$ and $p \nmid  (\lambda -1)$.
Consider $f(X):=\frac{X^{\lambda-1} -1}{X^{\lambda-1}}$ and the associated logarithmic differential form $\omega:=\frac{df}{f} =(\lambda-1)\frac{dX}{X^\lambda-X}$.
Then, we have that $\omega = \frac{dX}{P(X)}$ for $P(X)= \frac{X^\lambda-X}{\lambda-1}$.
There are $\lambda$ simple poles and no zeroes outside $\infty$, then the unique zero at $\infty$ is of order $\lambda-2$
and $\langle \omega \rangle_{\Fp}$ is a space $L_{\lambda,1}$
\end{example}

It follows from Example \ref{ex:Llambda1} that spaces $L_{\lambda, 1}$ exist for all $\lambda$ satisfying $p \nmid (\lambda-1)$.
The converse is also true: if $p \mid (\lambda-1)$ then by Proposition \ref{prop:Llambda1} (iii) the leading term of $P^{p-1}$ vanishes, but this implies that $\deg(P)< \lambda$, which is not possible.
The paper \cite{GreenMatignon99} contains more results on spaces $L_{\lambda, 1}$, such as a description of all possible spaces $L_{\lambda, 1}$ in the case $\lambda<p+1$.
A simple but fundamental example that fits in this case is the following:
\begin{example}
Let $p\geq 3$ and $f(X):= \prod_{i=1}^{p-1} (X-i1_k)^i \in k[X]$, where $1_k \in k$ is the unity of the field $k$. 
Then, $\Omega:=\langle \frac{df}{f} \rangle_{\Fp}$ is a space $L_{p-1,1}$. 
In fact, by construction the non-zero forms in $\Omega$ are logarithmic and their set of (simple) poles is $\{1_k, 2 \cdot 1_k, \dots, (p-1)\cdot 1_k\}$. At the pole $i \cdot 1_k$, the residue of $\omega$ is equal to $i \cdot 1_k$.
We can then verify that the equations (\ref{eq:zeroatinfty}) are satisfied: 
\[\begin{cases}
\sum_{i=1}^{p-1} i^{k} \equiv 0 \; \; \mod p & \text{ for } 1 \leq k \leq p-2\\
\sum_{i=1}^{p-1} i^{p-1} \equiv \sum_{i=1}^{p-1} 1 \equiv -1 & \mod p
\end{cases}\]
By Proposition \ref{prop:Llambda1}, $\Omega$ is a space $L_{p-1,1}$.
We can obtain the same result by rewriting the differential form as $\omega=\frac{df}{f}= \frac{dX}{1-X^{p-1}}$, from which the computation of residues also follows.
\end{example}

\begin{remark}
Without the assumption $m<p$, a deeper overview of the possible $m+1$-tuples of residues $\underline{a}$ (called \emph{Hurwitz data}) is contained in Henrio's Ph.D. thesis \cite{HenrioThese}.
It is worth noting that several questions about Hurwitz data remain unanswered (see also \cite[\S 1.1]{Pagot02}).
\end{remark}

\begin{remark}
Let $q=p^t$ with $t\geq 1$.
In \cite[Definition 3.2.]{FresnelMatignon23} a generalization of spaces $L_{m+1, n}$ to the setting of $\F_q$-vector spaces is introduced.
Namely, a space $L_{m+1, n}^q$ is defined as a $\F_q$-vector space of differential forms on $\PP^1_k$ whose nonzero elements have simple poles, a unique zero of order $m-1$ at $\infty$ and residues in $\F_q$.
We remark that all the results proved in this section for spaces $L_{m+1, n}$ generalize to spaces $L_{m+1, n}^q$.
In particular, the key Proposition \ref{prop:Cartier} has a natural analogue, and if one introduces the operator $\nabla$ acting on the field $k(X)$ by
\[ \nabla: f(X) \mapsto \left(f(X)^{(p-1)} \right)^{1/p}, \] 
the Jacobson-Cartier condition (Corollary \ref{cor:1/P}) is generalized over $\F_q$ by saying that a differential form $\omega=\frac{dX}{P(X)}$ has simple poles and residues in $\F_q$ if, and only if, $\nabla^t(P(X)^{q-1})=(-1)^t$.
As in Proposition \ref{prop:Llambda1}, it is therefore possible to express the fact that $\omega$ generates a space $L_{m+1, 1}^q$ as a condition on the coefficients of $P(X)^{q-1}$.
More precisely, the condition states that the coefficient of $X^{q-1}$ is $(-1)^t$ and that the coefficients of $X^{\mu q -1}$ vanish for every $\mu\geq 2$.

More generally, it is not difficult to adapt the majority of the results in the rest of the papers to spaces $L_{m+1, n}^q$.
The main idea is to change the key definition of the Moore determinant of an $n$-tuple $\underline{a}:=(a_1, \dots, a_n) \in k^n$ into
\[\Moore{n}{a}:= \begin{vmatrix}
a_1 & a_2 & \dots & a_n \\
a_1^q & a_2^q & \dots & a_n^q \\
\vdots & \vdots & \dots & \vdots \\
a_1^{q^{n-1}} & a_2^{q^{n-1}} & \dots & a_n^{q^{n-1}} \\
\end{vmatrix}\]
as already considered in \cite{FresnelMatignon23}.
The vanishing of this Moore determinant is a necessary and sufficient condition for $\F_q$-linear dependence of the elements of $\underline{a}$ and key Theorem \ref{thm:Pagotn} restated in terms of $q$-Moore determinants and spaces $L_{m+1, n}^q$ holds with a very similar proof.

However, we have decided not to state our results in this full generality and the case $q=p$ remains the focus of our study: in fact, only in this setting the elements of a space $L_{m+1, n}^q$ are logarithmic, while we do not have a similar interpretation otherwise.
This property is crucial to solve a concrete question about lifting local actions of elementary abelian groups (see \cite[Th\'eor\`eme 11]{Pagot02}), while it remains unknown what geometric interpretation or other use might have the existence of a space $L_{m+1, n}^q$ for $q$ not prime.
\end{remark}

\section{An obstruction to the existence of spaces $L_{\lambda p, 2}$}

We now consider the case $n=2$, where we recall results established by Pagot and show that, for $p>3\lambda$ there are no spaces $L_{\lambda p, 2}$.

\subsection{Known results on spaces $L_{\lambda p, 2}$}
We recall the following fundamental result by Pagot (\cite[Proposition 7]{Pagot02}). A proof is included, which contains elements that are crucial for the main result of this section, as well as for the generalization that we propose in Section \ref{sec:Pagotn}.
\begin{proposition}\label{prop:Pagot}
Let $\omega_1, \omega_2 \in \Omega(k(X))$. Then the $\F_p$-vector space $\Omega$ generated by $\omega_1$ and $\omega_2$ is a space $L_{\lambda p,2}$ if and only if there exist two polynomials $Q_1,Q_2 \in k[X]$ satisfying the following conditions:
\begin{enumerate}
\item For every $[i:j] \in \PP^1(\F_p)$ we have $\deg(iQ_1 +jQ_2) = \lambda$
\item $\displaystyle \omega_1=\frac{Q_2 dX}{Q_1Q_2^p-Q_1^pQ_2}$ and $\displaystyle \omega_2=\frac{-Q_1 dX}{Q_1Q_2^p-Q_1^pQ_2}$
\item The $p-1$-th derivative $\left(\frac{Q_2^{p^2-1}-Q_1^{p^2-1}}{Q_2^{p-1}-Q_1^{p-1}}\right)^{(p-1)}$ of the polynomial $\frac{Q_2^{p^2-1}-Q_1^{p^2-1}}{Q_2^{p-1}-Q_1^{p-1}}$ is equal to $-1$.
\end{enumerate}
\end{proposition}
\begin{proof}
Suppose that $\Omega$ is a space $L_{\lambda p,2}$.
Let us fix a basis $(\omega_1, \omega_2)$ of $\Omega$, and remark that the set of poles of an element $i\omega_1 + j\omega_2 \in \Omega$ depends only on the corresponding $[i:j] \in \PP^1(\F_p)$.
Let us then denote by $X_{[i:j]}$ the set of poles of differential forms in $\Omega$ that are not poles of $i\omega_1 + j\omega_2$.
By the results of section \ref{ssec:combpoles}, every $X_{[i:j]}$ consists of $\lambda$ elements and the set consisting of the $X_{[i:j]}$ for all $[i:j] \in \PP^1(\F_p)$ is a partition of the set of $\lambda (p+1)$ poles of $\Omega$.
Let us consider the polynomials
\[
P_{[i:j]}(X)=\prod_{x\in X_{[i:j]}} (X-x) \;\; \mbox{and} \;\; P(X)= \prod_{[i:j]\in \PP^1(\F_p)} P_{[i:j]}(X).
\]
Since $i\omega_1 + j\omega_2$ has a unique zero at infinity and poles outside $X_{[i:j]}$, we need that
\[ i\omega_1 + j\omega_2 = \dfrac{c_{ij}P_{[i:j]}(X)}{P(X)} dX\] 
for nonzero constants $c_{ij} \in k$ that satisfy the condition
\[ c_{ij}P_{[i:j]}(X)= i c_{10}P_{[1:0]}(X) + jc_{01}P_{[0:1]}(X).\]
Since all the $P_{[i:j]}$ are monic polynomials, this means in particular that $c_{ij} = ic_{10}+j c_{01}$.
Let us set $a=\frac{c_{10}}{c_{01}}$ and note that $a \notin \F_p$, otherwise we would have $c_{(p-1)a}=0$, which is not possible, as $\omega_1$ and $\omega_2$ are assumed $\F_p$-linearly independent.
As a result, since $k$ is algebraically closed, there is an element $c \in k^\times$ satisfying $c^p = -\frac{1}{c_{01}\left(a^p-a\right)}$.
Let us set 
\[Q_1 := - cP_{[0:1]} \;\; \text{and} \;\; Q_2 := ac P_{[1:0]} .\]
Then, $iQ_2-jQ_1=c(iaP_{[1:0]}+jP_{[0:1]})=c\frac{c_{ij}}{c_{01}}P_{[i:j]}$, which is a polynomial of degree $\lambda$.\\
Moreover, we have $Q_1Q_2^p-Q_1^pQ_2=-c^{p+1}(a^p-a) P = \frac{c}{c_{01}}P$ and hence 
\[ \omega_1=\frac{Q_2 dX}{Q_1Q_2^p-Q_1^pQ_2} \;\; \mbox{and} \;\; \omega_2=\frac{-Q_1 dX}{Q_1Q_2^p-Q_1^pQ_2},\] as required by condition $(ii)$.
Among other things, this proves that $\omega_1$ is of the form $\frac{1}{P(X)}dX$, and hence it satisfies the hypotheses of Corollary \ref{cor:1/P} (in fact, the multiplication by the element $c\in k^\times$ performed earlier on is crucial for having $\omega_1$ in such a form).
Applying the corollary, we then have that $\left((Q_1Q_2^{p-1}-Q_1^p)^{p-1}\right)^{(p-1)}=-1$.
In order to prove $(iii)$, we only need to show that
\begin{equation}\label{eq:JCinvariant}
    \left((Q_1Q_2^{p-1}-Q_1^p)^{p-1}\right)^{(p-1)} =  \left(\frac{Q_2^{p^2-1}-Q_1^{p^2-1}}{Q_2^{p-1}-Q_1^{p-1}}\right)^{(p-1)}.
\end{equation}
To do this, we simply note that 
\[Q_2^{p(p-1)}+(Q_1Q_2^{p-1}-Q_1^p)^{p-1}=\sum_{i=0}^{p} Q_1^{i(p-1)}Q_2^{(p-i)(p-1)}=\frac{Q_2^{p^2-1}-Q_1^{p^2-1}}{Q_2^{p-1}-Q_1^{p-1}},\]
and therefore $((Q_1Q_2^{p-1}-Q_1^p)^{p-1})'=\left(\frac{Q_2^{p^2-1}-Q_1^{p^2-1}}{Q_2^{p-1}-Q_1^{p-1}}\right)'$, from which it follows \emph{a fortiori} that 
$$\left((Q_1Q_2^{p-1}-Q_1^p)^{p-1}\right)^{(p-1)} =  \left(\frac{Q_2^{p^2-1}-Q_1^{p^2-1}}{Q_2^{p-1}-Q_1^{p-1}}\right)^{(p-1)}=-1.$$
\smallskip

Conversely, let us start with $Q_1$ and $Q_2$ satisfying the three conditions of the proposition and show that they give rise to a space $L_{\lambda p, 2}$.
By $(ii)$, we write $\omega_1=\frac{Q_2 dX}{Q_1Q_2^p-Q_1^pQ_2}$ and $\omega_2=\frac{-Q_1 dX}{Q_1Q_2^p-Q_1^pQ_2}$.
By $(i)$, there are poles of $\omega_1$ that are not poles of $\omega_2$, so that $\omega_2$ can not be a multiple of $\omega_1$ and $\Omega=\langle \omega_1, \omega_2 \rangle_{\Fp}$ has dimension 2.
We then note that any non-zero element $\omega \in \Omega$ can be written as $\omega=\frac{R_2 dX}{R_1R_2^p-R_1^pR_2}$ for some pair of polynomials $(R_1,R_2) = (Q_1,Q_2)M$ obtained multiplying $(Q_1,Q_2)$ by a suitable matrix $M \in GL_2(\F_p)$.
It follows that $\omega$ has a unique zero at $\infty$.
We leave the reader with the task to check that $\frac{Q_2^{p^2-1}-Q_1^{p^2-1}}{Q_2^{p-1}-Q_1^{p-1}}=\frac{R_2^{p^2-1}-R_1^{p^2-1}}{R_2^{p-1}-R_1^{p-1}}$. Combining this with condition $(iii)$ we get that $\left(\frac{R_2^{p^2-1}-R_1^{p^2-1}}{R_2^{p-1}-R_1^{p-1}}\right)^{(p-1)}=-1$.
By applying equation (\ref{eq:JCinvariant}) we find that $\left((R_1R_2^{p-1}-R_1^p)^{p-1}\right)^{(p-1)}=-1$, which by Corollary \ref{cor:1/P} implies that $\omega$ is logarithmic for every $\omega \in \Omega$.
\end{proof}
\begin{remark}\label{rmk:simplepolesn=2}
    It follows from the proposition that, given a pair $Q_1, Q_2$ satisfying conditions $(i), (ii)$ and $ (iii)$, then $Q_1$ and $Q_2$ are necessarily coprime, and have simple roots.
    In fact, the poles of $\omega_1$ being simple, we have that $\prod_{i=0}^{p-1} (Q_1+iQ_2)=Q_1^p-Q_1Q_2^{p-1}$ has simple roots. In particular, $Q_1$ has simple roots and no roots in common with $Q_1+Q_2$, hence no roots in common with $Q_2$.
    Similarly, from the fact that $\omega_2$ has simple poles one derives that $Q_2$ has simple roots.
\end{remark}

\begin{remark}\label{rmk:Moore}
The polynomial $(Q_1Q_2^p-Q_1^pQ_2)$ appearing in the denominators of $\omega_1, \omega_2$ is called the \emph{Moore determinant} and denoted by $\Delta_2(Q_1,Q_2)$.
This is the determinant of the Moore matrix
$\begin{pmatrix}
Q_1 & Q_2 \\
Q_1^p & Q_2^p
\end{pmatrix}$.
Furthermore, the polynomial $\frac{Q_2^{p^2-1}-Q_1^{p^2-1}}{Q_2^{p-1}-Q_1^{p-1}}$ is the \emph{Dickson invariant} $c_{2,1}$ evaluated in $Q_1, Q_2$.
The appearance of Moore determinants and Dickson invariants is far from a coincidence: as we will see in Section \ref{sec:Pagotn}, these have a fundamental role in the generalizations of Proposition \ref{prop:Pagot} for spaces $L_{\lambda p^{n-1},n}$.
Moreover, Moore determinants will be helpful to simplify some proofs, even in the case of dimension 2.
For this reason, we have collected the results we need on Moore determinants and Dickson invariants in Appendix \ref{app:Moore}.
\end{remark}

In light of the result of the Proposition \ref{prop:Pagot}, we introduce the following definition, which will be extensively used in the classification of spaces $L_{12,2}$ and $L_{15,2}$:
\begin{definition}\label{defn:givesriseto}
Let $Q_1, Q_2 \in k[X]$ be polynomials of degree $\lambda$ such that $\deg(iQ_1 +jQ_2) = \lambda$, for every $[i:j] \in \PP^1(\F_p)$.
Define the associated differential forms 
\[ \omega_1:= \frac{dX}{\left (Q_1Q_2^{p-1}-Q_1^p \right)} \; \mbox{and} \; \omega_2:= \frac{dX}{\left(Q_2Q_1^{p-1}-Q_2^p\right)}.\]
We say that the pair $(Q_1,Q_2)$ \emph{gives rise to} the pair $(\omega_1, \omega_2)$.
Moreover, if there exists $c \in k^\times$ such that the pair $(c Q_1, c Q_2)$ gives rise to $(\frac{\omega_1}{c^p}, \frac{\omega_2}{c^p})$ a basis of $\Omega$ a space $L_{\lambda p, 2}$, then we say that the pair $(Q_1,Q_2)$ is a \emph{prompt} for the space $\Omega$.
\end{definition}

\begin{remark}\label{rmk:givesriseto}
By Proposition \ref{prop:Pagot}, a pair $(Q_1,Q_2)$ of polynomials of degree $\lambda$ is a prompt for a space $L_{\lambda p, 2}$ if, and only if, $Q_1$ and $Q_2$ have leading coefficients that are $\F_p$-independent, and are such that $\left((Q_1^p-Q_1Q_2^{p-1})^{p-1}\right)^{(p-1)}$ is a non-zero constant $d\in k^\times$.
The number $c$ then needs to satisfy $d(c^{p})^{p-1}=-1$, and hence it is uniquely determined up to multiplication by a $p-1$-th root of unity.
\end{remark}

\begin{conv}\label{conv:Ri}
If we have polynomials $Q_1$ and $Q_2$ satisfying the conditions $(i)$ and $(ii)$ of Proposition \ref{prop:Pagot}, we can write $(Q_1^p-Q_1Q_2^{p-1})^{p-1}(X)=\sum r_i X^i$, and apply Proposition \ref{prop:Llambda1} to deduce that condition $(iii)$ is equivalent to the equalities $r_{p-1}= 1$ and $r_{k p -1}=0$ for $k=2, \dots,\lambda (p-1)$.
The $r_i$'s are polynomials in the coefficients of $Q_1$ and $Q_2$, and the equalities above will be used in Sections \ref{sec:class} and \ref{sec:L4p,2} to classify certain spaces $L_{\lambda p, 2}$.
In order to simplify a frequently used notation, we set $R_{k}:=r_{k p -1}$ for $k=1, \dots,\lambda (p-1)$.
\end{conv}

\begin{lemma}\label{lem:s1=t1}
Assume $(\lambda,p) \neq (1,2)$ and let $Q_1,Q_2\in k[X]$ be polynomials of degree $\lambda$ with
\begin{align*}
Q_1(X)= a\left(X^\lambda+\sum_{i=1}^{\lambda} (-1)^i s_i X^{\lambda-i} \right) \\
Q_2(X)= b \left(X^\lambda+\sum_{i=1}^{\lambda} (-1)^i t_i X^{\lambda-i} \right).
\end{align*} 
If the pair $(Q_1,Q_2)$ is a prompt for a space $L_{\lambda p, 2}$, then we have $s_1=t_1$.
\end{lemma}
\begin{proof}
We have that $Q_1$ and $Q_2$ satisfy condition $(iii)$ of Proposition \ref{prop:Pagot} and hence the polynomials $R_k$ of Convention \ref{conv:Ri} vanish for $k>1$.
In particular, this is true for $R_{\lambda(p-1)}$, the coefficient of degree $\lambda(p-1)p-1$ of the polynomial $(Q_1^p-Q_1Q_2^{p-1})^{p-1}(X)$.
To compute $R_{\lambda(p-1)}$, let us write $(Q_1^p-Q_1Q_2^{p-1})^{p-1}=X^{\lambda(p-1)p}\left[ \left( \frac{Q_1}{X^\lambda} \right)^{p-1} \frac{Q_2}{X^\lambda}- \left(\frac{Q_2}{X^\lambda}\right)^p \right]^{p-1}$, introduce the variable $Z=\frac{1}{X}$ and compute the coefficient of $Z$ in the expression in brackets above. 
We have:
\begin{align*}
\left( \frac{Q_1}{X^\lambda} \right)^{p-1} \frac{Q_2}{X^\lambda}- \left(\frac{Q_2}{X^\lambda}\right)^p & \equiv a^{p-1}b(1-s_1Z)^{p-1}(1-t_1Z) - b^p(1-t_1^pZ^p)  \mod \; Z^2 k[Z]\\
& \equiv b\left (a^{p-1}(1-(t_1-s_1)Z-b^{p-1} \right ) \mod \; Z^2 k[Z].
\end{align*}
From this, we deduce that
\begin{align*}
\scriptstyle{\left( \left( \frac{Q_1}{X^\lambda} \right)^{p-1} \frac{Q_2}{X^\lambda}- \left(\frac{Q_2}{X^\lambda}\right)^p \right)^{p-1}} & \equiv  b^{p-1}\left (a^{p-1}(1-(t_1-s_1)Z)-b^{p-1} \right )^{p-1} \mod Z^2k[Z]\\
& \equiv b^{p-1}\left( (a^{p-1}-b^{p-1})^{p-1} + a^{p-1}(a^{p-1}-b^{p-1})^{p-2}(t_1-s_1) Z \right ) \mod Z^2k[Z]. &
\end{align*}
We see then that $R_{\lambda(p-1)}=b^{p-1}a^{p-1}(a^{p-1}-b^{p-1})^{p-2}(t_1-s_1)=0$. Since $a,b$ are nonzero and $\F_p$-linearly independent, then we have that $s_1=t_1$.
\end{proof}

\subsection{A new generic obstruction to the existence of spaces $L_{\lambda p, 2}$}

It is a result of Pagot (cf. \cite[Theor\`eme 1 and Theor\`eme 2]{Pagot02}) that spaces $L_{p,2}$ and $L_{3p,2}$ exist only for $p=2$ and spaces $L_{2p,2}$ exist only for $p=2,3$.
In this section, we show that there are no spaces $L_{\lambda p,2}$ if $p$ is large enough with respect to $\lambda$, vastly improving on the previously known situation.
The genericity in the title of the section refers then to the fact that our result holds for all but finitely many primes once $\lambda$ is fixed, which allows for a direct computation of the remaining cases.
For example, for $\lambda=4$ we only need to check the primes $p \in \{2,3,5,7,11\}$.
For $p=2$ the result is classical (see Remark \ref{rmk:L_2lambda,2}), the case $p=3$ is achieved in Section \ref{sec:L12,2} and the case $p\in \{5,7,11\}$ is achieved in Section \ref{sec:L4p,2}, completing the classification of spaces $L_{4p,2}$.

Let us now state our result. 
To simplify the demonstration, we exclude from the statement the case $\lambda=1$, which has a known short proof (see \cite[Th\'eor\`eme 2, Part 1]{Pagot02}).
By contrast, the proof in the other known cases $\lambda=2,3$ consists of several pages and the argument below consistently simplifies it.

\begin{theorem}\label{thm:generic}
Let $p>3\lambda$. Then there are no spaces $L_{\lambda p,2}$.
\end{theorem}

To prove the theorem, we need to recall some notation and establish two fundamental lemmas.
If we have a space $L_{\lambda p,2}$ generated by a basis$(\omega_1, \omega_2)$ and we consider the polynomials $Q_1$ and $Q_2$ giving rise to $(\omega_1, \omega_2)$,
we recall from the proof of Proposition \ref{prop:Pagot} that, for every $[i:j] \in \PP^1(\F_p)$, $P_{[i:j]}$ denotes the monic polynomial whose zeroes are those of $iQ_2-jQ_1$, and that $a$ denotes the quotient of the leading terms of $Q_1$ and $Q_2$, which satisfies $a \notin \F_p$.
We are now ready to establish our lemmas:

\begin{lemma}
For every $t\in k - \{ -a\}$, let $P_t:= \frac{aP_{[1:0]}+tP_{[0:1]}}{a+t}$ and denote by $\Disc(P_t)$ its discriminant.
Then there exists a polynomial $R(X)\in k[X]$ such that:
\begin{enumerate}[ref=\thelemma~(\roman*)]
\item We have $\Disc(P_t)=\frac{R(t)}{(a+t)^{2\lambda-3}}$ and $\deg(R(X)) \leq 2\lambda - 3$.\label{lem:inter1_i}
\item Let $p>3\lambda$. Then the element $-a^p \in k$ is a root of $R(X)$ of order $\geq \lambda +3$.\label{lem:inter1_ii}
\end{enumerate}
\end{lemma}

\begin{proof}
We will first prove item $(i)$, and then use it as one of the ingredients for the proof of $(ii)$.
\smallskip

\emph{Proof of (i)}:
Let $a_i$ be the coefficient of degree $i$ in the polynomial $P_t$.
The discriminant $\Disc(P_t)$ is the determinant of the Sylvester matrix
\[
\begin{pmatrix}
a_\lambda & a_{\lambda-1} & \cdots & a_2& a_1 & a_0 & \cdots & 0 \\
0 & a_\lambda &  \cdots & a_3 & a_2 & a_1 & \cdots & 0 \\
\vdots & \vdots & \ddots & \vdots & \vdots & \vdots & & \vdots \\
0 & 0 & \cdots & a_\lambda &  a_{\lambda-1} & a_{\lambda -2} & \cdots & a_0 \\
\lambda a_\lambda & (\lambda-1)a_{\lambda-1} & \cdots & 2a_2& a_1 & 0 & \cdots & 0 \\
0 & \lambda a_\lambda & \cdots & 3 a_3 & 2a_2 & a_1 & \cdots & 0 \\
\vdots & \vdots & \ddots & \vdots & \vdots & \vdots &\ddots & \vdots \\
0 & 0 & \cdots & \lambda a_\lambda & (\lambda-1)a_{\lambda-1} & (\lambda-2)a_{\lambda-2} &\cdots & a_1
\end{pmatrix}.
\]
Since both $a_\lambda$ and $a_{\lambda-1}$ are independent of $t$ (the former being 1 and the latter as a result of Lemma \ref{lem:s1=t1}), the first two columns of this matrix are always the same for every $t$.
As a result, we are left with at most $2\lambda -3$ rows that contain quantities of the form $\frac{p_{ij}(t)}{a+t}$, where $p_{ij}(t)$ is either zero or a polynomial of degree 1.
Applying Leibniz formula for the determinant, we then obtain that $\Disc(P_t)=\frac{R(t)}{(a+t)^{2\lambda-3}}$ for some $R(t)$ of degree at most $2\lambda -3$ as desired.

\smallskip
\emph{Proof of (ii)}: 
We consider the differential form $\omega_2=\frac{-Q_1 dX}{Q_1Q_2^p-Q_1^pQ_2}$ and remark that the poles of $\omega_2$ are the zeroes of the polynomial $Q_2^p-Q_1^{p-1}Q_2 = \prod_{j=0}^{p-1} \left( Q_2 - jQ_1 \right)$.
Hence, for every pole $x$ of $\omega_2$, there exists a number $j \in \{0,\dots, p-1\}$ such that $x$ is a zero of $Q_2 - j Q_1$.
The residue of $\omega_2$ at $x$ can then be computed as follows:
\[\res_{\omega_2}(x)=\frac{-1}{\left(Q_2^p-Q_1^{p-1}Q_2\right)^{'}(x)}=\frac{-1}{(Q_2-jQ_1)'(x)\prod_{i\neq j}(Q_2-iQ_1) (x)}=\frac{1}{(Q_2-jQ_1)'(x) (Q_1^{p-1}(x))}.\]
For every $j=0, \dots, p-1$, we recall from the proof of Proposition \ref{prop:Pagot} that $Q_2-jQ_1=c(a+j)P_{[1:j]}$, and we observe that $P_{[1:j]}$ coincides with the polynomial $P_j$ defined in the statement of the Lemma.
We can then rewrite the identity above as 
\[ \res_{\omega_2}(x)=\frac{1}{c(a+j){P_{j}'}(x)Q_1(x)^{p-1}}.\]

Considering the product $H_j$ of the residues at all the poles that are roots of $P_{j}$, we have that
\[H_j:= \prod_{x \in Z(P_{j})} \res_{\omega_2}(x) = \frac{1}{c^\lambda (a+j)^\lambda \prod_{x \in Z(P_{j})} P_{j}^{'}(x) \prod_{x \in Z(P_{j})} Q_1^{p-1}(x)}.\]

If we denote by $\Res ( \bullet, \bullet)$ the resultant of two polynomials, we can rewrite the above as 

\[H_j= \frac{1}{(-1)^{\frac{\lambda(\lambda-1)}{2}}\Disc(P_{j}) (c (a+j))^{-\lambda(p-2)}  \Res(Q_2 - jQ_1,Q_1)^{p-1}} = (-1)^\frac{\lambda(\lambda-1)}{2}  \frac{(c (a+j))^{\lambda(p-2)}} {\Disc(P_{j}) \Res(Q_2,Q_1)^{p-1}}.\]
 
By Lemma \ref{lem:inter1_i}, we can express the discriminant $\Disc(P_{j})$ in terms of the polynomial $R(X)$ to obtain that $H_j = \delta \frac{(a+j)^{\lambda p - 3}}{R(j)}$,
where we have set $\delta:= (-1)^{\frac{\lambda(\lambda-1)}{2}}  \frac{c^{\lambda(p-2)}}{\Res(Q_2, Q_1)^{p-1}}$ for ease of notation, since this is independent of $j$.

Since the differential form $\omega_2$ is logarithmic, we have that $H_j\in \F_p^\times$ and in particular that $H_j^{p-1}=1$. 
The following equations then hold for every $j \in \{0, \dots, p-1\}$:

\begin{align*}
\delta^{p-1} (a+j)^{(\lambda p -3)(p-1)}& =  R(j)^{p-1} \\
\delta^{p-1} (a+j)^{(\lambda p^2-(\lambda+3) p +3)} R(j)& =  R(j)^{p} \\
\delta^{p-1} (a+j)^{\lambda p^2} (a+j)^3 R(j)& =  R(j)^{p}(a+j)^{(\lambda+3)p} \\
\delta^{p-1} (a^{p^2}+j)^{\lambda}(a+j)^3 R(j)& =  R_p(j)(a^p+j)^{(\lambda+3)},\\
\end{align*}
where $R_p(X)$ denotes the polynomial obtained from $R(X)$ by raising its coefficients to the $p$-th power.
We thus have obtained the equation
\[ \delta^{p-1} (a^{p^2}+j)^{\lambda}(a+j)^3 R(j) -  R_p(j)(a^p+j)^{(\lambda+3)} = 0,\]
which is a polynomial equation of degree at most $3\lambda$ in $j$ that is satisfied for every $j=0,\dots,p-1$.
Since we have that $p>3\lambda$, this is actually an equality of univariate polynomials, namely we have that

\begin{equation*}
\delta^{p-1} (a^{p^2}+X)^{\lambda}(a+X)^3 R(X) =  R_p(X)(a^p+X)^{(\lambda+3)}
\end{equation*}
holds in the ring $k[X]$.
The right hand side of the equation admits $-a^p$ as root of order at least $\lambda+3$.
Since $a^p \neq a$ and $a^p \neq a^{p^2}$ we conclude that $R(X)$ has $-a^p$ as root of order at least $\lambda+3$, as well.
\end{proof}

\begin{remark}
Despite its fairly elementary proof, Lemma \ref{lem:inter1_i} is quite powerful. Combining the lower bound on the order of $-a^p$ as a zero of $R(X)$ given in $(ii)$ and the upper bound on the degree of $R(X)$ given by $(i)$ one gets that $2\lambda-3 \geq \lambda+3$, which gives $\lambda \geq 6$. 
It follows that no further argument is needed to prove Theorem \ref{thm:generic} when $\lambda \leq 5$.
In order to prove it in full generality, we need to go a bit further.
\end{remark}

In order to prove the auxiliary lemma below, we introduce the following notation: for polynomials $A,B \in k[X]$, we denote by $A \wedge B \in k[X]$ the monic gcd of $A$ and $B$ and by $\rdeg(A)$ the degree of the reduction of $A$, that is, the number of distinct roots of the polynomial $A$.
Finally, if $A = \prod_{i=1}^{\rdeg(A)} (X-x_i)^{m_i}$, we denote by $Z_A = \{x_1, \dots, x_{\rdeg{A}}\}$ the set of zeroes of $A$, and we define $\hat{\prod}_{x\in Z_A} B(x):= \prod_{i=1}^{\rdeg{A}} B(x_i)^{m_i}$.

\begin{lemma}\label{lem:inter2}
Let $k$ be an algebraically closed field of characteristic $p>0$.
Let $P,Q$ be coprime polynomials in $k[X]$ such that $0 \leq \deg(Q) < \deg(P) < p$ and let $D(Z) \in k[Z]$ be the discriminant in degree $\deg(P)$ of the polynomial $P(X) - Z Q(X) \in k[X]$.

Then, the derivative $\left( \frac{P}{Q} \right)'$ of the rational function $\frac{P}{Q} \in k(X)$ can be written as
\[ \left( \frac{P(X)}{Q(X)} \right)' = \frac{A(X)}{B(X)Q(X)} \]
such that $A \wedge Q = 1$ and $A(X) = \prod_{i=1}^d (X - x_i)$ with $x_1, \dots, x_d \in k$ and $d=\deg(P)+\rdeg(Q)-1$.
Moreover, the $x_i$'s are such that
\[ D(Z)= c \prod_{i=1}^d\left(Z-\frac{P(x_i)}{Q(x_i)}\right) \; \mbox{for some} \; c \in k^\times.\]
In particular, $\deg(D) = \deg(P)+\rdeg(Q)-1$.
\end{lemma}

\begin{proof}
Let us write $N=P'Q-PQ'$, so that $\left(\frac{P}{Q}\right)'=\frac{N}{Q^2}$. 
Clearly $(P\wedge P')(Q\wedge Q')$ divides $N$ and local computations at the zeroes of $P Q$, relying on the fact that both $\deg{(P)}$ and $ \deg{(Q)}$ are prime to $p$, show that the quotient $N_0 :=\frac{N}{(P\wedge P')(Q\wedge Q')}$ has no common roots neither with $P$ nor with $Q$.
We can then set 
\[A=a(P\wedge P')N_0 \; \; \mbox{and} \; \; B=\frac{Q}{Q\wedge Q'}\] for a suitable $a\in k^\times$ that makes $A$ monic.
By construction, $A \wedge Q =1$ and $A \wedge B =1$.
Hence, a zero of $\left(\frac{P}{Q}\right)'$ is either a zero of $P\wedge P'$ or of $N_0$. 

Since $p$ does not divide $\deg(P)-\deg(Q)$, the term of degree $\deg P+\deg Q-1$ in $N$ is nonzero and hence $\deg(N)=\deg P+\deg Q-1$.
Moreover, since $p$ does not divide $\deg(Q)$, one has $\deg(Q\wedge Q') =\deg Q-\rdeg Q$. 
It follows that 
\begin{align*}
\deg(A)&=\deg ((P\wedge P')N_0) = \deg(N)-\deg(Q\wedge Q')=\deg P+\deg Q-1-(\deg Q-\rdeg Q)\\ &=\deg P+\rdeg Q-1.
\end{align*}

Let us now prove the second part of the statement.
For $t\in k$, we denote by $Z_t$ the set of zeroes of the polynomial $P- tQ \in k[X]$, and by $Z_N$ the set of zeroes of the polynomial $N$. 
We then compute $D(t)$ using properties of resultants: we have
\begin{eqnarray*}
D(t)&=&\Res(P-tQ,P'-tQ')\\
&=& c_1\hat\prod_{x\in Z_t} (P'-tQ')(x)
\end{eqnarray*}
where $c_1$ is independent of $t$, as it just involves the leading term of $P-tQ$ and we have by assumption that $\deg(Q)<\deg(P)$.
As a result, we can write
\[D(t)=c_1 \frac{\hat{\prod}_{x\in Z_t}\left(P'(x)Q(x)-tQ'(x)Q(x)\right)}{ \hat{\prod}_{x\in Z_t} Q(x)} .\]

Note that this is possible as $P\wedge Q=1$ and then for every $x\in Z_t$ we have $Q(x)\neq 0$. 
Now, up to multiplication by a constant that does not depend on $t$, the numerator of the fraction above is equal to $\Res(P'Q-tQ'Q,P-tQ)$, while the denominator is equal to $\Res(P-tQ, Q)$.
We can then write 
  \begin{eqnarray*}
  D(t)&=&c_2\frac{\Res(P'Q-tQ'Q,P-tQ)}{\Res(P-tQ,Q)}\\
  &=&c_2\frac{\Res(P'Q-PQ',P-tQ)}{\Res(P,Q)}\\
  &=& c_3\hat{\prod_{x\in Z_N}}(P-tQ)(x)
  \end{eqnarray*}
  where the new constant $c_3$ results from dividing $c_2$ by $\Res(P,Q)$ and multiplying by a suitable power of the leading term of $N$.
  Using this, we have
   \begin{eqnarray*}
  D(t) &=& c_3\hat{\prod_{\substack{x\in Z_N\\Q(x)=0}}}(P-tQ)(x)\hat{\prod_{\substack{x\in Z_N\\Q(x)\neq0}}}(P-tQ)(x)\\
  &=&c_3\hat{\prod_{\substack{x\in Z_N\\Q(x)=0}}}P(x)\hat{\prod_{\substack{x\in Z_N\\Q(x)\neq0}}}(P-tQ)(x)\\
  &=&c_4\hat{\prod_{\substack{x\in Z_N\\Q(x)\neq0}}}(P-tQ)(x)\\
  &=&c_4\hat{\prod_{\substack{x\in Z_N\\Q(x)\neq0}}}Q(x)\hat{\prod_{\substack{x\in Z_N\\Q(x)\neq0}}}\left(\frac{P}{Q}-t\right)(x)\\
  &=&c_5\hat{\prod_{\substack{x\in Z_N\\Q(x)\neq0}}}\left(\frac{P(x)}{Q(x)}-t\right)\\
  &=&c_5\hat{\prod_{x\in Z_A}}\left(\frac{P(x)}{Q(x)}-t\right)\\
  \end{eqnarray*}
where the final equation is justified by the facts that $N= a^{-1}A(Q\wedge Q')$ and that $A \wedge Q=1$.
\end{proof}

\begin{proof}[Proof of Theorem \ref{thm:generic}]

Assume by contradiction that there is a space $L_{\lambda p, 2}$ with $p> 3\lambda$.
We set $P:= aP_{[1:0]}-a^p P_{[0:1]}$ and $Q:=a(P_{[1:0]}-P_{[0:1]})$ and we remark that these satisfy the conditions of Lemma \ref{lem:inter2}.
As a result, the zeroes of $D(Z)=\Disc(P-ZQ )\in k[Z]$ are all of the form $\frac{P(x)}{Q(x)}$ for $x$ a zero of $\left( \frac{P}{Q} \right)'$.
In particular, $z=0$ is a zero of $D$ of order at most $\deg(P')=\lambda -1$, as it corresponds to a zero $x$ of $\left( \frac{P}{Q} \right)'$ that also satisfies $P(x)=0$ (and therefore also $P'(x)=0$).

Let us now give a lower bound to the order of $0$ as a zero of $D$ and see that it is incompatible with the one above. 
To do this, we evaluate the polynomial $D$ at an element $z = \frac{t + a^p}{t+a}$.
In this way, we have
\begin{align*}
D(z) & = \Disc\left(P - \frac{(t + a^p)}{(t+a)} Q\right) = \Disc \left( \frac{(t+a)P-(t+a^p)Q}{(t+a)} \right) \\
& = \Disc \left( \frac{(a-a^p)(aP_{[1:0]}+t P_{[0:1]})}{(t+a)} \right) = \Disc \big((a-a^p) P_t \big) = (a-a^p)^{2 \lambda -1} \Disc(P_t).
\end{align*} 

By Lemma \ref{lem:inter1_i}, this last expression can be written in terms of $R(t)$, giving
\[D(z) = (a-a^p)^{2 \lambda -1}\frac{R(t)}{(a+t)^{2\lambda - 3}}.\]
By Lemma \ref{lem:inter1_ii}, $-a^p$ is a zero of order at least $\lambda+3$ of the polynomial function $t \mapsto R(t)$ and, since the expression of $z$ in $t$ is a linear fractional transformation, we have that $0$ is a zero of the same order of the polynomial $D(Z)$.
This gives the desired contradiction and concludes the proof of the theorem.
\end{proof}

\begin{remark}\label{rmk:genericobs}
It is natural to ask whether the obstruction of Theorem~\ref{thm:generic}
extends to the spaces $L_{\lambda p^{n-1},n}$.
However, any $2$-dimensional subspace of a space $L_{\lambda p^{n-1},n}$ would have to
satisfy the inequality $p > 3\lambda p^{n-2}$ in order for
Theorem~\ref{thm:generic} to apply and this condition is meaningful only when $n=2$.
Consequently, no higher-dimensional obstructions arise directly from
Theorem~\ref{thm:generic}, and a more refined approach is required.
We establish preliminary results in this direction—most notably
Theorem~\ref{thm:lambda=1} that settles the case $\lambda=1$ for every $n$—but it remains open whether a generic
condition yielding obstructions to the existence of spaces $L_{\lambda p^{n-1},n}$ for fixed $\lambda, n$, and almost every $p$ exists when $n>2$.
\end{remark}

\section{Conditions for the existence of spaces $L_{\lambda p^{n-1}, n}$}\label{sec:Pagotn}

In this section, we prove a generalization of Proposition \ref{prop:Pagot} that applies to spaces $L_{\lambda p^{n-1}, n}$ for any $n \geq 2$ and discuss some of its consequences.
As anticipated in \ref{rmk:Moore}, our strategy makes a crucial use of Moore determinants.
Definition and results about Moore determinants that we use in this section are recalled in Appendix \ref{app:Moore}.
For every $n$-tuple of the form $\ul{X}:=(X_1, \dots, X_n)$, we denote by $\Delta_n(\ul{X})$ the associated Moore determinant.
Moreover, we denote by $\ul{\hat{X}_i}$ the $n-1$-tuple obtained from $\ul{X}$ by removing $X_i$ and by $\Delta_{n-1}(\ul{\hat{X}_i})$ the associated Moore determinant.



\begin{lemma}\label{lem:MooreP}
Let $Q_1, \dots, Q_n \in k[X]$ be polynomials of common degree $\lambda \geq 1$.
Let $P:= \Delta_n(\ul{Q})$ and for every $i\in \{1, \dots, n\}$ let $P_i:= (-1)^{i-1} \Delta_{n-1}(\ul{\hat{Q}_i})$.
Then, we have
\begin{align*}
  \Moore{n}{P}=P^{1+p+\cdots+p^{n-2}},
\end{align*}
\end{lemma}

\begin{proof}
This is a direct corollary of Theorem \ref{thm:FM4.1.}.
\end{proof}

\begin{proposition}\label{prop:Pe|P}
Let $Q_1, \dots, Q_n \in k[X]$ be polynomials of common degree $\lambda \geq 1$.
Let $P:= \Delta_n(\ul{Q})$ and let us define for every $i\in \{1, \dots, n\}$ the polynomials $P_i:= (-1)^{i-1} \Delta_{n-1}(\ul{\hat{Q}_i})$ and for every $\underline{\epsilon}\in \F_p^n-\{\underline 0\}$ the polynomials $P_{\underline{\epsilon}}:= \sum_{i} \epsilon_i P_i$.
Denote by $q_i$ be the leading coefficient of $Q_i$ and assume that $\Delta_n(\underline q)\neq 0$.
Then, the following conditions are met
\begin{enumerate}
    \item We have the equality $\deg P=(1+p+p^2+\dots+p^{n-1})\lambda$
    \item For every $\underline{\epsilon}\in \F_p^n-\{\underline 0\}$, $\deg P_{\underline\epsilon}=(1+p+p^2+\dots+p^{n-2})\lambda$
    \item For every $\underline{\epsilon}\in \F_p^n-\{\underline 0\}$ $P_{\underline\epsilon}|P$.
\end{enumerate}

\end{proposition}
\begin{proof}
Let us prove the three items separately:
\begin{enumerate}
    \item The leading coefficient of $P$ is $\Delta_n(\underline q)$, which is non-zero by hypothesis, hence computing the Moore determinant gives $\deg P=(1+p+p^2+\dots+p^{n-1})\lambda$.
    \item The leading coefficient of $P_{\underline\epsilon}$ is $\sum_i \epsilon_i\Delta_{n-1}(\underline{\hat{q}_i})$, which is non-zero by \cite[Corollary 2.1]{FresnelMatignon23}, hence $\deg P_{\underline\epsilon}=\deg{P_1}=(1+p+p^2+\dots+p^{n-2})\lambda$.
    \item We first note that we have 
\[P_{\underline\epsilon}=\sum_{i=1}^n(-1)^{i-1}\epsilon_i \Moore{n-1}{\hat{Q_i}}= \begin{vmatrix}
\epsilon_1 & \epsilon_2 & \dots & \epsilon_n\\
Q_1 & Q_2 & \dots &Q_n\\
Q_1^p & Q_2^p & \dots &Q_n^p\\
\vdots & \vdots & \ddots & \vdots \\
Q_1^{p^{n-2}} & Q_2^{p^{n-2}} & \dots &Q_n^{p^{n-2}}
\end{vmatrix},
\]
a determinant that we denote by $\delta_{\ul{\epsilon}}(\underline{Q})$ as in Appendix \ref{app:Moore}.

Let $W:= \langle Q_1, \dots, Q_n\rangle_{\Fp}$. 
Since $\Delta_n(\underline q)\neq 0$, then the $Q_i$'s are $\F_p$-linearly independent, hence $\dim W = n$.
Let $\{Q_1^\star, \dots, Q_n^\star\}$ be the basis of $W^\star$ which is dual to $\{Q_1, \dots, Q_n\}$ and denote by $\varphi_{\ul{\epsilon}} \in W^\star$ the $\F_p$-linear form $\sum_{i=1}^n (-1)^{i-1}\epsilon_i Q_i^\star$.
Then by Formula (\ref{eq:deltae}) we have
\[\prod_{Q \in \ker \varphi_{\ul{\epsilon}} - \{\ul{0}\}} Q = (-1)^{n-1}\delta_{\underline{\epsilon}}(\underline{Q})^{p-1}=(-1)^{n-1}P_{\underline\epsilon}^{p-1}.\]
We choose the following system of representatives of the projectivization $\PP(W)$ of $W$:
\[S(W):= \bigcup_{i=1}^n \left( Q_i+\F_p Q_{i-1}+\dots + \F_p Q_1 \right).\]
and for every subspace $V \subset W$ we denote by $S(V)$ the intersection $V \cap S(W)$.
It is a system of representatives of $\Proj(V)$.
A counting argument shows that $\prod_{Q\in V - \{\ul{0}\}} Q = (-1)^{\dim V} (\prod_{Q\in S(V)} Q)^{p-1}$, which combined with the above gives the identity $P_{\underline\epsilon}^{p-1}=\left( \prod_{Q \in S(\ker \varphi_\epsilon)} Q\right)^{p-1}$.
We have then that there exists $\mu \in \F_p^\times$ such that $P_{\underline\epsilon}=\mu \prod_{Q \in S(\ker \varphi_{\ul{\epsilon}})} Q$.
By Equation (\ref{eq:Mooreprod}), we have that $P=\Moore{n}{Q}=\prod_{Q \in S(W)} Q$. Since $S(\ker(\varphi_{\ul{\epsilon}})) \subset S(W)$ it follows that $P_{\underline\epsilon}|P$.
\end{enumerate}
\end{proof}

\begin{remark}\label{rmk:PPi}
Proposition \ref{prop:Pe|P} shows that the expression $\frac{P}{P_i}$ is a polynomial.
We observe that it is an additive polynomial in the variable $Q_i$. 
Let us prove this for $i=n$, from which the other cases follow.
First of all, from \eqref{eq:Mooreprod} we get that \[ P=\prod_{i=1}^n \prod_{\epsilon_{i-1}\in
\F_p} \dots \prod_{\epsilon_{1}\in \F_p}
(Q_i+\epsilon_{i-1}Q_{i-1}+\dots+\epsilon_1Q_1)
  \] and that  \[ P_n=(-1)^{n-1}\prod_{i=1}^{n-1} \prod_{\epsilon_{i-1}\in \F_p}
\dots \prod_{\epsilon_{1}\in \F_p}
(Q_i+\epsilon_{i-1}Q_{i-1}+\dots+\epsilon_1Q_1).
  \]
  Putting these two equations together results in the formula
  \[ \frac{P}{P_n}=(-1)^{n-1}\prod_{\epsilon_{n-1}\in \F_p} \dots
\prod_{\epsilon_{1}\in \F_p}
(Q_n+\epsilon_{n-1}Q_{n-1}+\dots+\epsilon_1Q_1).
  \] 
 If we denote by ${\mathcal Q}_{n-1}$ the $\F_p$-vector space
$\langle Q_1,Q_2,\dots, Q_{n-1}\rangle_{\Fp}$ the formula above is rewritten as 
  \[ \frac{P}{P_n}= (-1)^{n-1} P_{\calQ_{n-1}}(Q_n), \]
where $P_{\calQ_{n-1}}$ is the structural polynomial of ${\mathcal Q}_{n-1}$ of Definition \ref{defn:strucpoly}, which is additive in the variable $Q_n$.
\end{remark}

\subsection{Constructing spaces $L_{\lambda p^{n-1},n}$ from polynomials}\label{subsec:Pagotn}

We now have all the tools to prove the two general results of this section (Theorems \ref{thm:Pagotn0} and \ref{thm:Pagotn}).
Let us first establish some general results on the relationship between the $Q_i$'s and some spaces of differential forms that we can build from them:

\begin{definition}\label{defn:Qi}
Let $\ul{Q}:=(Q_1, \dots, Q_n) \in k[X]^n$ be a $n$-tuple of polynomials of degree $\lambda \geq 1$ with leading coefficients $q_i$ satisfying $\Delta_n(\underline q)\neq 0$.
We write $P:= \Delta_n(\ul{Q})$ and $P_i:= (-1)^{i-1} \Delta_{n-1}(\ul{\hat{Q}_i})$.
We define differential forms $\omega_i:=\frac{P_i}{P}dX$ and the space $\Omega:=\langle \omega_1, \dots, \omega_n \rangle_{\Fp}$.
We say that the $n$-tuple $\ul{Q}$ \emph{gives rise} to the basis $(\omega_1, \dots, \omega_n)$.
\end{definition}

By Proposition \ref{prop:Pe|P} we have that for all $\underline{\epsilon}\in \F_p^n-\{\underline 0\}$ the polynomial $P_{\underline\epsilon}$ divides $P$ and from Definition \ref{defn:Qi} we see that

\[ \frac{P_{\underline\epsilon}}{P}dX = \epsilon_1\omega_1+\dots+\epsilon_n\omega_n. \]

This entails that all the nonzero differential forms in $\Omega$ have a unique zero of order $\lambda p^{n-1}-2$ at infinity.
However, they are not in general logarithmic.
In the following, we begin an investigation of conditions for $\Omega$ to be a space $L_{\lambda p^{n-1},n}$ that culminates in Theorem \ref{thm:Pagotn0}. 

For every $M \in GL_n(\F_p)$ we denote by $(\underline{Q} M)_1, \dots, (\underline{Q} M)_n$ the components of the vector $\underline{Q} M$ obtained by applying the matrix $M$ to $\underline{Q}$.
We note that all the entries $(\underline{Q} M)_i$'s are polynomials of degree $\lambda$ with leading coefficients that are $\Fp$-independent.
We associate with this $n$-tuple the differential forms $\omega_i^M:=(-1)^{i-1}\frac{\Delta_{n-1}\big( \widehat{(\underline{Q} M)_i} \big)}{\Delta_n(\underline{Q} M)}dX$ and the space $\Omega':=~\langle \omega_1^M, \dots, \omega_n^M \rangle_{\Fp}$.\\

\begin{proposition}\label{prop:common}
Assume the notation of Definition \ref{defn:Qi}.
\begin{enumerate}[label=\ref{prop:common}(\roman*)]
\item For every $M \in GL_n(\F_p)$ we have that
\[ (\omega_1^M, \dots, \omega_n^M) = (\omega_1, \dots, \omega_n) (M^{-1})^t,\]
where $(M^{-1})^t \in GL_n(\Fp)$ is the transpose of the inverse of $M$.
In particular, $\Omega'=\Omega$. \label{prop:QiMi}
\item  Let $\underline{Q}:=(Q_1, \dots, Q_n)$ and $\underline{T}:=(T_1, \dots, T_n)$ be $n$-tuples of polynomials in $k[X]$ giving rise to the same basis $(\omega_1, \dots, \omega_n)$. Then $\underline{Q} =\underline{T}$. \label{prop:QiMii}
\item Let $\underline{Q}:=(Q_1, \dots, Q_n)$ and $\underline{T}:=(T_1, \dots, T_n)$ be $n$-tuples of polynomials in $k[X]$ giving rise to bases of the same space $\Omega$.
Then, there exists a matrix $M \in GL_n(\F_p)$
such that $\underline{T}
 =  \underline{Q} M.$ \label{prop:QiMiii}
\end{enumerate}
\end{proposition}

\begin{proof}
\begin{enumerate}
\item For every $i,j \in \{1, \dots, n\}$ we denote by $M_{i,j}$ the $(i,j)$-th minor of the matrix $M$.
We then have that 
$\Delta_{n-1}(\widehat{(\underline{Q} M)}_j) = \sum_{i=1}^n M_{i,j} \Moore{n-1}{\hat{Q}_i}$
using $\F_p$-multilinearity and the alternating property of $\Delta_{n-1}$.
Then, using the fact that $\Delta_{n}(\underline{Q} M)=\Moore{n}{Q} \det(M)$ and that $M$ is invertible, we get

\begin{align*}
\omega_j^M := & \frac{(-1)^{j-1}\Delta_{n-1}(\widehat{(\underline{Q} M)}_j)}{\Delta_n(\underline{Q} M)} dX = \frac{(-1)^{j-1}\sum_{i=1}^n M_{i,j} \Moore{n-1}{\hat{Q}_i}}{\Delta_n(\underline{Q} M)} dX = \frac{\sum_{i=1}^n {(-1)^{i+j}M_{i,j} P_i}}{\Delta_n(\underline{Q} M)} dX \\
=& \frac{1}{\det(M)} \sum_{i=1}^n (-1)^{i+j} \frac{M_{i,j}P_{i}}{P} dX  = \sum_{i=1}^n \frac{(-1)^{i+j}}{\det(M)} M_{i,j} \omega_i.
\end{align*}

In other words, $(\omega_1^M, \dots, \omega_n^M) = (\omega_1, \dots, \omega_n) (M^{-1})^t$.
We have then that $\omega_1^M, \dots, \omega_n^M$ is a basis of $\Omega$ for every invertible matrix $M \in GL_n(\F_p)$.

\item[(ii) and (iii)] Let $\omega_{i,Q}:=\frac{(-1)^{i-1}\Delta_{n-1}(\hat{Q_i})}{\Moore{n}{Q}} dX$ and $\omega_{i,T}:=\frac{(-1)^{i-1}\Delta_{n-1}(\hat{T_i})}{\Moore{n}{T}} dX$.
Since $\underline{T}$ and $\underline{Q}$ both arise from the space $\Omega$ we have that there exists a matrix $N \in GL_n(\F_p)$ such that
\[ (\omega_{1,Q}, \dots, \omega_{n,Q} ) N =(\omega_{1,T}, \dots, \omega_{n,T} ) \]
Then, by part (i), one has that the $n$-tuple $\underline{Q} (N^{-1})^t$ gives rise to the basis $(\omega_{1,T}, \dots, \omega_{n,T})$.
Let $M:= (N^{-1})^t$ and let us show that $\underline{Q} M = \underline{T}$: this will both prove $(ii)$ (in which case $N= M= \mathbb{I}$) and $(iii)$.
From the fact that $\underline{Q} M$ and $ \underline{T}$ give rise to the same basis we get
\[ \left( \frac{\Delta_{n-1}\left(\widehat{(\underline{Q}M)}_1\right)}{\Delta_n(\underline{Q} M)}  , \dots, \frac{(-1)^{n-1}\Delta_{n-1}\left(\widehat{(\underline{Q}M)}_n\right)}{\Delta_n(\underline{Q} M)} \right) =  \left( \frac{\Delta_{n-1}(\hat{T}_1)}{\Delta_n(\underline{T})}, \dots, \frac{(-1)^{n-1}\Delta_{n-1}(\hat{T_n})}{\Delta_n(\underline{T})} \right). \]
We can then apply the Moore determinant to the terms of this equality and use Theorem \ref{thm:FM4.1.} to get that
\[ \Delta_{n}(\underline{Q}M)^{p^{n-1}}=\Moore{n}{T}^{p^{n-1}}, \]
which implies that $\Delta_{n}(\underline{Q}M)=\Moore{n}{T}$.
Hence, we have that 
\[(-1)^{i-1}\Delta_{n-1}(\widehat{(\ul{Q}M)_i})=(-1)^{i-1}\Delta_{n-1}(\widehat{T_i}) \;\; \mbox{ for every} \; \; i=1, \dots, n\] which we know by Proposition \ref{prop:FM5.1.} to be equivalent to the fact that $\underline{Q} M = \theta \underline{T} $ for some $\theta \in k(X)^{alg}$ with $\theta^{1+p+\dots+p^{n-2}}=1$.
But since $\Delta_{n}(\underline{Q}M)=\Moore{n}{T}$, we have that $\theta^{1+\dots+p^{n-1}}=1$ and hence $\theta^{p^{n-1}}=1$, that is, $\theta=1$.
\end{enumerate}
\end{proof}

For a space $L_{\lambda p^{n-1}, n}$, we will show in Theorem \ref{thm:Pagotn} that we can always associate polynomials $Q_1, \dots, Q_n$ giving rise to a basis as in Definition \ref{defn:Qi}.
In this context, Proposition \ref{prop:QiMiii} says that two choices of such a $n$-tuple are necessarily related by multiplication of an invertible matrix with entries in $\F_p$.

Let us now prove another useful proposition, first recalling from Definition \ref{defn:strucpoly}, that the \emph{structural polynomial} of a $\Fp$-vector space $V$ is defined as $P_V(X):= \prod_{v \in V} (X-v) \in k[X]$.

\begin{proposition}\label{prop:subspacestruc}
Let $\Omega = \langle \omega_1, \dots, \omega_n \rangle_{\Fp}$ be as in definition \ref{defn:Qi}.
For every $1\leq t\leq n$, let $\Omega_t\subset \Omega$ be the $\Fp$-subspace of $\Omega$ generated by $\{\omega_{n-t+1}, \dots, \omega_n \}$ and let $\calQ_{n-t} = \langle Q_1, \dots, Q_{n-t} \rangle_{\Fp}$.
Then, the $t$-tuple of polynomials $\big(P_{\calQ_{n-t}}(Q_{n-t+1}), \dots, P_{\calQ_{n-t}}(Q_n)\big)$ gives rise to the basis $(-1)^{n-t}(\omega_{n-t+1},\dots,\omega_n\big)$ of $\Omega_t$.
\end{proposition}
\begin{proof}
If we specialize Corollary \ref{cor:structuralMoore} to the case $X_i = Q_{i}$, we get, for every $n-t+1 \leq i \leq n$, that
\[ \frac{\Delta_{n-1}(\widehat{Q_i})}{\Moore{n}{Q}} = \frac{\Delta_{t-1}(P_{\calQ_{n-t}}(Q_{n-t+1}), \dots, \widehat{P_{\calQ_{n-t}}(Q_{i})}, \dots, P_{\calQ_{n-t}}(Q_{n}))}{\Delta_t(P_{\calQ_{n-t}}(Q_{n-t+1}), \dots, P_{\calQ_{n-t}}(Q_{n}))}. \]
As a result, the $t$-tuple $(P_{\calQ_{n-t}}(Q_{n-t+1}), \dots, P_{\calQ_{n-t}}(Q_n))$ gives rise to the basis of $\Omega_t$ given by the $t$-tuple $\big((-1)^{n-t}\omega_{n-t+1}, \dots, (-1)^{n-t} \omega_{n} \big)$.
\end{proof}

\begin{theorem}\label{thm:Pagotn0}
Let $\Omega$ be a space of differential forms constructed as in Definition \ref{defn:Qi}. 
If there exists a non-zero $\omega\in \Omega$ that is a logarithmic differential form, then $\Omega$ is a space $L_{\lambda p^{n-1},n}$.
\end{theorem}

\begin{proof}
Lemma \ref{lem:MooreP} ensures that $\Moore{n}{P} \neq 0$, hence the $P_i$'s are $\F_p$-linearly independent and then $\Omega=\langle \omega_1, \dots, \omega_n \rangle_{\Fp}$ is a vector space of dimension $n$ of differential forms that have a unique zero at $\infty$ (recall that by Proposition \ref{prop:Pe|P} we have that $P_{\underline{\epsilon}} | P$).

Up to a change of basis of $\Omega$, we can assume that $\omega=\omega_n$.
We then need to show that, if $\omega_n$ is a logarithmic differential form then all the forms in $\Omega$ are logarithmic.
We start by claiming the following: if $\omega_n$ is logarithmic, then $\langle \omega_{n-1}, \omega_n \rangle_{\Fp}$ is a space $L_{\lambda p^{n-1},2}$. 
For this, we recall from the proof of Proposition \ref{prop:Pagot} that it is sufficient to find polynomials $R_1$ and $R_2$ such that 
$\displaystyle \omega_n =\frac{R_1 dX}{R_1^pR_2-R_1R_2^p}$ and $\displaystyle \omega_{n-1}=\frac{R_2 dX}{R_1R_2^p-R_1^pR_2}$.
This is done by applying Proposition \ref{prop:subspacestruc} to the case $t=2$, and setting $R_1=P_{\calQ_{n-2}}(Q_{n-1})$ and $R_2=P_{\calQ_{n-2}}(Q_{n})$.

It follows that $\omega_{n-1}$ is also a logarithmic differential form. 
Using the same argument, we can show that $\omega_i$ is logarithmic for every $i=1, \dots, n-1$ and then that every form in $\Omega$ is logarithmic, which entails that $\Omega$ is a space $L_{\lambda p^{n-1},n}$.
\end{proof}

\begin{remark}\label{rmk:simplepolesanyn}
    As a consequence of the theorem, we can prove that $P=\Delta_n(\ul{Q})$ has simple roots, and hence that the $Q_i$'s are pairwise coprime and have simple roots,
    generalising Remark \ref{rmk:simplepolesn=2} to any $n$.
    To see that $P=\Delta_n(\ul{Q})$ has simple roots, we note that its degree is precisely $\lambda \frac{p^n-1}{p-1}$, and that by construction its zeroes are the elements of $\calP(\Omega)$, which we know by Lemma \ref{lem:Pagotpoles} to be a set of cardinality $\lambda \frac{p^n-1}{p-1}$, as $\Omega$ is a space $L_{\lambda p^{n-1}, n}$.
    The roots of $P$ then need to be simple.
\end{remark}

In analogy with the case $n=2$, Theorem~\ref{thm:Pagotn0} provides a directly computable method to verify that a space constructed as in Definition~\ref{defn:Qi} is a space $L_{\lambda p^{n-1},n}$.
Namely, one can apply Proposition~\ref{prop:Llambda1}~(iii) and conclude that it suffices to solve the following system of polynomial equations:

\begin{equation}\label{eq:coeff}
\begin{cases}
\coeff \left( \big( \frac{P}{P_n} \big)^{p-1} , X^{p-1}\right) = 1 \\
\coeff \left( \big( \frac{P}{P_n} \big)^{p-1} , X^{\mu p -1}\right) = 0, \;\; 2 \leq \mu \leq \lambda p^{n-2}(p-1).
\end{cases}
\end{equation}

Note that, in adapting the proposition to the case of an $n$-dimensional space, the number of poles under consideration becomes $\lambda p^{n-1}$, and the number of equations must be computed accordingly.

The reciprocal of Theorem \ref{thm:Pagotn0} also holds, completing the generalization of Proposition \ref{prop:Pagot}:
\begin{theorem}\label{thm:Pagotn}
Let $\Omega$ be a space $L_{\lambda p^{n-1},n}$ with $n \geq 2$.
Then, there exist polynomials~$Q_1,\cdots,Q_n \in~k[X]$ of degree $\lambda$ such that, writing $P:=\Moore{n}{Q}$ and $P_i:=(-1)^{i-1}\Delta_{n-1}(\ul{\hat Q_i})$, we have that $P_i | P$ and $\Omega=\langle \omega_1, \dots, \omega_n \rangle_{\F_p}$ with
$\omega_i=\frac{P_i}{P}dX$.
\end{theorem}

\begin{proof}
The case $n=2$ is provided by Proposition \ref{prop:Pagot}, so we can proceed by induction: we fix $n\geq 3$ and assume that we have the result of the theorem in dimension up to $n-1$.\\
Let $(\omega_1, \dots, \omega_n)$ be a basis of $\Omega$, and let $\displaystyle P(X) = \prod_{x \in \calP(\Omega)}(X-x)$.
Then we can find $P_1, \dots, P_n \in k[X]$ such that $\omega_i=\frac{P_i}{P} dX$.
By Lemma \ref{lem:Pagotpoles} the degree of $P$ is $\lambda(1+p+\dots+p^{n-1})$ and then $\deg(P_i)=\lambda(1+p+\dots+p^{n-2})$.
By \cite[Proposition 4.1]{FresnelMatignon23} specialized at the case $q=p$ (cf. also the remark in \cite[p. 68]{Pagot02}) we have that $\Delta_n(P_1, \dots, P_n)=\gamma P^{1+p+\dots + p^{n-2}}$
for some $\gamma \in k^\times$.
By possibly multiplying $P$ and the $P_i$'s by $\mu \in k^\times$ satisfying $\mu^{p^{n-1}}\gamma=1$, we can assume that
\begin{equation}\label{eq:DeltaP}
\Delta_n(P_1, \dots, P_n)=P^{1+p+\dots + p^{n-2}}.
\end{equation}
We now remark that the space $\langle \omega_1, \dots, \omega_{n-1}\rangle_{\Fp}$ is a space $L_{\lambda p^{n-1}, n-1}$ and then by inductive hypothesis there exist $S_1, \dots, S_{n-1} \in k[X]$ of degree $\lambda p$ such that 
$\displaystyle \frac{P_i}{P} = (-1)^{i-1} \frac{\Moore{n-2}{\hat{S_i}}}{\Moore{n-1}{S}}$, which means that $P_i=(-1)^{i-1} \Moore{n-2}{\hat{S_i}} \frac{P}{\Moore{n-1}{S}}$ for every $1 \leq i \leq n-1$.
We then have that
\begin{equation}\label{eq:PiSi}
\Delta_{n-1}(P_1, \dots, P_{n-1})=\Delta_{n-1}((-1)^{i-1} \Moore{n-2}{\hat{S_i}})\frac{P^{1+p+\dots+p^{n-2}}}{\Moore{n-1}{S}^{1+p+\dots+p^{n-2}}}=\frac{P^{1+p+\dots+p^{n-2}}}{\Moore{n-1}{S}^{p^{n-2}}},
\end{equation}
where the last equality is obtained by applying Theorem \ref{thm:FM4.1.}.

To conclude, we need to show that there exist $Q_1,Q_2,\dots,Q_n \in k[X]$ of degree $\lambda$ such that $P=\Moore{n}{Q}$ and $P_i=(-1)^{i-1}\Delta_{n-1}(\ul{\hat Q_i})$.
Let $\varphi: k^n \longrightarrow k^n$ be the map defined by $\left(\varphi(\underline{a})\right)_i  =  (-1)^{i-1}\Moore{n-1}{\hat{a_i}}$.
Then by Proposition \ref{prop:FM5.1.} (since $\Moore{n}{P} \neq 0$) there exist $n$-tuples $\underline{Q}, \underline{R}$ of elements of $k(X)^{\mathrm{alg}}$ satisfying $\varphi(\underline{R})=\underline{Q}$, $\varphi(\underline{Q})=\underline{P}$, and 
\begin{equation}\label{eq:PiRi}
P_i = (-1)^{n-1} \Moore{n}{R}^{1+p+\dots+p^{n-3}} R_i^{p^{n-2}}.
\end{equation}
We ought to show that the entries of $\underline{Q}$ are polynomials with coefficients in $k$.
From equation (\ref{eq:PiRi}) we can deduce the following identities
\begin{align}
\Delta_{n-1}(P_1, \dots, P_{n-1}) = (-1)^{n-1}(\Moore{n}{R})^{(1+p+\dots+p^{n-3})(1+p+\dots+p^{n-2})}\Delta_{n-1}(R_1^{p^{n-2}}, R_2^{p^{n-2}}, \dots, R_{n-1}^{p^{n-2}}) \label{eq:n-1}\\
\Moore{n}{P} = (-1)^{n(n-1)}\Moore{n}{R}^{(1+p+\dots+p^{n-3})(1+p+\dots+p^{n-1})}\Moore{n}{R}^{p^{n-2}} = \Moore{n}{R}^{(1+p+\dots+p^{n-2})^2}.\label{eq:n}
\end{align} 
Combining (\ref{eq:n}) with (\ref{eq:DeltaP}) gives 
\begin{equation}\label{eq:theta}
\Moore{n}{R}^{1+p+\dots+p^{n-2}}= \theta P
\end{equation}
for some $\theta \in k$ such that $\theta^{1+p+\dots+p^{n-2}}=1$.

Moreover, we have that $Q_n= (-1)^{n-1}\Delta_{n-1}(R_1, \dots, R_{n-1})$, and then
\begin{align*}
Q_n^{p^{n-2}} & = (-1)^{n-1} \Delta_{n-1}(R_1^{p^{n-2}}, \dots, R_{n-1}^{p^{n-2}})=\frac{\Delta_{n-1}(P_1, \dots, P_{n-1})}{\Moore{n}{R}^{(1+p+\dots+p^{n-2})(1+p+\dots+p^{n-3})}}= \\
& = \frac{\Delta_{n-1}(P_1, \dots, P_{n-1})}{(\theta P)^{1+p+\dots+p^{n-3}}} =\frac{P^{1+p+\dots+p^{n-2}}}{(\theta P)^{1+p+\dots+p^{n-3}} \Moore{n-1}{S}^{p^{n-2}}}=\frac{P^{p^{n-2}}}{\theta^{1+p+\dots+p^{n-3}} \Moore{n-1}{S}^{p^{n-2}}},
\end{align*}
where the equalities are obtained by applying equations (\ref{eq:n-1}), (\ref{eq:theta}) and (\ref{eq:PiSi}).
Finally, using the fact that $\theta^{1+p+\dots+p^{n-2}}=1$, we get that
\[Q_n = \frac{\theta P}{\Moore{n-1}{S}},\]
which is a polynomial of degree $\lambda$ thanks to the fact that the zeroes of $\Moore{n-1}{S}$ are simple and correspond to the set of poles of the space $\langle \omega_1, \dots, \omega_{n-1}\rangle_{\Fp}$ (see Theorem \ref{thm:Pagotn0}).
Moreover, by Lemma \ref{lem:hyperplanes}, we have that $\deg(Q_n)=\lambda$.
In a completely analogous way, we can show that the $Q_i$'s are polynomials of degree $\lambda$ also for $1\leq i \leq n-1$.
Finally, we need to multiply the $Q_i's$ by a constant in order for them to satisfy the Theorem.
In fact, we have \[  \Moore{n}{P}= \Delta_n(\Delta_{n-1}(\underline{\hat{Q}_1}),\dots,(-1)^{i-1}\Delta_{n-1}(\underline{\hat{Q}_i}),\dots,(-1)^{n-1}\Delta_{n-1}(\underline{\hat{Q}_n})),\] which combined with (\ref{eq:DeltaP}) and Theorem \ref{thm:FM4.1.} gives 
\[P^{1+p+\cdots+p^{n-2}}=\Delta_n(\underline Q)^{1+p+\cdots+p^{n-2}}.\]

Then, there exists $\theta'\in k$ with $\theta'^{1+p+\cdots+p^{n-2}}=1$ and $\Delta_n(\underline Q)=\theta' P$. 

Let $\theta''\in k$ with $(\theta'')^{p^{n-1}}\theta'=1$. This satisfies $(\theta'')^{1+p+\cdots+p^{n-2}}=1$ and then by Proposition \ref{prop:FM5.1.} we have $\varphi(\theta''\underline Q)= \underline P$ and 
\[\Delta_n(\theta''\underline Q)=(\theta'')^{1+p+\cdots+p^{n-1}}\Delta_n(\underline Q)=(\theta'')^{p^{n-1}}\theta' P= P.\]
\end{proof}

Finally, we conclude the section with two results that relate the poles of a space $L_{\lambda p^{n-1}, n}$ and the zeroes of linear combinations of the polynomial $Q_i$'s.

\begin{corollary}\label{cor:subspacepoles} 
Let $\Omega=\langle \omega_1, \dots, \omega_n \rangle_{\Fp}$ be a space $L_{\lambda p^{n-1},n}$ and let $Q_1, \dots, Q_n$ be the $n$-tuple of polynomials arising from Theorem \ref{thm:Pagotn}. 
Then, for every $1\leq t\leq n$ the subspace $\Omega_t= \langle \omega_{n-t+1}, \dots, \omega_n \rangle_{\Fp}$ is such that  
\[ \calP(\Omega_{t}) =  \calP(\Omega) - \bigcup_{\underline{\epsilon}\in \F_p^{n-t} - \{0\}} Z\left(\sum_{i=1}^{n-t} \epsilon_i Q_{i} \right).\]
\end{corollary}
\begin{proof}
By Proposition \ref{prop:subspacestruc}, we have that $\calP(\Omega_t)=Z(S_t)$ where 
\[ S_t = \Delta_t(P_{\calQ_{n-t}}(Q_{n-t+1}), \dots, P_{\calQ_{n-t}}(Q_n)),\]
where $P_{\calQ_{n-t}}$ is the structural polynomial of $\langle Q_1, \dots, Q_{n-t} \rangle_{\Fp}$.
Since $P_{\calQ_{n-t}}$ is an additive polynomial, we have that 
\[S_t=\prod_{i=n-t+1}^n \prod_{\epsilon_{i-1}\in \F_p} \dots \prod_{\epsilon_{n-t+1}\in \F_p} (P_{\calQ_{n-t}}(Q_{i}+\epsilon_{i-1}Q_{i-1}+\dots+\epsilon_{n-t+1}Q_{n-t+1}))=\]  \[\prod_{i=n-t+1}^n \prod_{\epsilon_{i-1}\in \F_p} \dots \prod_{\epsilon_{n-t+1}\in \F_p} \prod_{Q\in \calQ_{n-t}}(Q+Q_{i}+\epsilon_{i-1}Q_{i-1}+\dots+\epsilon_{n-t+1}Q_{n-t+1}), \]
and the fact that the zeroes of $\Moore{n}{Q}$ are simple (see Theorem \ref{thm:Pagotn0}) ensures that the zeroes of $S_t$ are precisely those zeroes of $\Moore{n}{Q}$ that are not zeroes of any $Q\in \calQ_{n-t} - \{0\}$.
This is equivalent to say that
\[ \calP(\Omega_{t}) =  \calP(\Omega) - \bigcup_{\underline{\epsilon}\in \F_p^{n-t} - \{0\}} Z\left(\sum_{i=1}^{n-t} \epsilon_i Q_{i} \right).\]
\end{proof}

\begin{corollary}
Let $\Omega=\langle \omega_1, \dots, \omega_n \rangle_{\Fp}$ be a space $L_{\lambda p^{n-1},n}$ and let $Q_1, \dots, Q_n$ be the $n$-tuple of polynomials arising from Theorem \ref{thm:Pagotn} for this basis. 

Denote by $q_i$ the leading coefficient of $Q_i$. 
Then, for every $1\leq t \leq n$, we have the equality of Moore determinants
\[ \Delta_n(Q_1, \dots, Q_n)=\alpha \Delta_t(P_{\calQ_{n-t}}(Q_{n-t+1}), \dots, P_{\calQ_{n-t}}(Q_{n}))\Delta_{n-t}(Q_{1}, \dots, Q_{n-t}),\]
where 
\[\alpha = \frac{\Delta_n(q_1, \dots, q_n)}{\Delta_t(P_{\ul{q}}(q_{n-t+1}), \dots,P_{\ul{q}}(q_n))\Delta_{n-t}(q_{1}, \dots, q_{n-t})} \in k^\times,\] and $P_{\ul{q}_t}$ is the structural polynomial of the vector space $\langle q_1, \dots, q_{n-t}\rangle_{\Fp}$.
\end{corollary}
\begin{proof}
We know that the zeroes of $\Delta_n(\ul{Q})$ are simple and consist of the set $\calP(\Omega)$, and by Proposition \ref{prop:subspacestruc} the zeroes of $\Delta_t(P_{\calQ_{n-t}}(Q_{n-t+1}), \dots, P_{\calQ_{n-t}}(Q_{n}))$ are simple and consist of the set $\calP(\Omega_t)$ with $\Omega_t=~\langle \omega_{n-t+1}, \dots, \omega_n \rangle_{\Fp}$.
We may then apply Corollary \ref{cor:subspacepoles} to see that the set of zeroes of the polynomials on both sides of the equation are equal, and that these zeroes are all simple.
The corollary then follows from a comparison of the leading coefficients of these polynomials.
\end{proof}

We conclude this part with a result that describes the $n$-tuples of polynomials giving rise to a basis of an étale pullback of a space $L_{\lambda p^{n-1},n}$.

\begin{proposition}\label{prop:pullbackQi}
Let $\Omega=\langle \omega_1, \dots, \omega_n \rangle_{\Fp}$ be a space $L_{\lambda p^{n-1},n}$ and let $Q_1, \dots, Q_n$ be the $n$-tuple of polynomials arising from Theorem \ref{thm:Pagotn}. Let $S(X)\in k[X]$ with $S'(X)\in k^\times$ and let $\sigma^\star(\Omega)$ be the pullback of $\Omega$ with respect to the morphism $\PP^1_k \overset{\sigma}{\rightarrow}  \PP^1_k $ induced by $X \mapsto S(X)$ (cf. Lemma \ref{lem:pullbacketale}).
Then the polynomials arising from Theorem \ref{thm:Pagotn} for $\sigma^\star(\Omega)$ are $\big(\eta Q_1(S), \dots, \eta Q_n(S)\big)$ with $\eta^{p^{n-1}}=\frac{1}{S'(X)}$.
In particular, we have that $\calP(\sigma^\star(\Omega))=\{a\in k\ |\ S(a)\in \calP(\Omega)\}$.
\end{proposition}
\begin{proof}
From the equation $\omega_i=\frac{(-1)^{i-1} \Delta_{n-1}(\ul{\hat{Q}_i})}{\Delta_{n}(\ul{Q})}dX$ it follows that 
\[\sigma ^\star(\omega_i)=(-1)^{i-1}\frac{ \Delta_{n-1}(\ul{\widehat{Q_i(S)}})}{\Delta_{n}(\ul{Q(S)})}S'(X)dX=(-1)^{i-1}\frac{\Delta_{n-1}(\ul{\widehat{\eta Q_i(S)}})}{\Delta_{n}(\ul{\eta Q(S)})}dX
\]
where $\eta^{p^{n-1}}=\frac{1}{S'(X)}$.
\end{proof}

\subsection{Dickson invariants and spaces $L_{\lambda p^{n-1},n}$}\label{subsec:applyingPagotn}

In the previous section, we proved several results using the structural polynomial of the vector space generated by the $ Q_i $'s. 
Dickson invariants (see Appendix~\ref{app:Moore}) provide an elegant framework for describing this polynomial and can be effectively applied to the problem of determining the existence of spaces $ L_{\lambda p^{n-1}, n} $. 
First of all, Dickson invariants can be used to give a shorter and potentially more insightful proof of Theorem~\ref{thm:Pagotn0}.
To show this, let us adopt the same notation as in that theorem and recall that the non-trivial part is to show that if the form $\omega_n$ is logarithmic, then the forms $\omega_i$ are logarithmic for every $i = 1, \dots, n$.
Equivalently, we need to prove that if $\left(\frac{P}{P_n}\right)^{p-1}$ satisfies the system of equations~(\ref{eq:coeff}), then so does $\left(\frac{P}{P_i}\right)^{p-1}$ for every $i$.

From~(\ref{eq:Dickson.i}), it follows that
\[
\left(\frac{P}{P_n}\right)^{p-1}=\left( \frac{\Moore{n}{Q}}{\Delta_{n-1}(\ul{\hat{Q}_n})} \right)^{p-1} = c_{n,n-1} - c_{n-1,n-2}^p \in \F_p[Q_1, \dots, Q_n].
\]
We then consider the element $\tau_{(n,i)} \in \GL_n(\F_p)$, acting on $\F_p[Q_1, \dots, Q_n]$ by exchanging $Q_n$ with $Q_i$ and leaving all other variables unchanged.
Applying the action of $\tau_{(n,i)}$ to $\left(\frac{P}{P_n}\right)^{p-1}$, we obtain:
\[
\left( \frac{\Moore{n}{Q}}{\Delta_{n-1}(\ul{\hat{Q}_i})} \right)^{p-1}
= \left(  \frac{\tau_{(n,i)} \star\Moore{n}{Q}}{\tau_{(n,i)} \star\Delta_{n-1}(\ul{\hat{Q}_n})} \right)^{p-1}
= -(\tau_{(n,i)} \star c_{n,n-1}) + \left( \tau_{(n,i)} \star c_{n-1,n-2} \right)^p.
\]
Since $c_{n,n-1}$ is a Dickson invariant, we have $\tau_{(n,i)} \star c_{n,n-1} = c_{n,n-1}$, and it follows that
\[
\left(\frac{P}{P_i}\right)^{p-1} - \left(\frac{P}{P_n}\right)^{p-1} \in \left( \F_p[Q_1, \dots, Q_n] \right)^p \subset \left( \F_p[X] \right)^p.
\]
In particular, $\left(\frac{P}{P_n}\right)^{p-1}$ satisfies~(\ref{eq:coeff}) if and only if $\left(\frac{P}{P_i}\right)^{p-1}$ does.

More crucially, the identities in Proposition~\ref{prop:Dickson} allow us to achieve two new key results: (1) a construction of a large class of previously unknown spaces $ L_{\lambda 2^{n-1}, n} $ in characteristic 2 and (2) a proof that spaces $ L_{p^{n-1}, n} $ (i.e., with $ \lambda = 1 $) do not exist when $ p > 2 $.
The key step required for these two applications is a proposition that distills Theorems~\ref{thm:Pagotn0} and~\ref{thm:Pagotn} into a new necessary and sufficient condition for the existence of a space $L_{\lambda p^{n-1},n}$ using the property of Dickson invariants.

\begin{proposition}\label{prop:det}
     Let $(Q_1,\cdots,Q_n)\in k[X]^n$ be a $n$-tuple of polynomials of degree $\lambda \geq 1$ with leading coefficients $q_i$ satisfying $\Moore{n}{q}\neq 0$.
     Then, the $n$-tuple $\ul{Q}$ gives rise to a basis of a space $L_{\lambda p^{n-1},n}$ if and only if
\begin{equation*}
         \left(\det(\underline{Q}',\underline{Q},\underline{Q}^p,\cdots,\underline{Q}^{p^{n-2}})\cdot \det(\underline{Q},\underline{Q^p},\cdots,\underline{Q}^{p^{n-2}},\underline{Q}^{p^{n-1}})^{p-2}\right)^{(p-2)}=1, 
\end{equation*}
where the exponent $(p-2)$ denotes the $p-2$-th derivative of a polynomial in $k[X]$.
 \end{proposition} 

 \begin{proof}

 By \ref{eq:Dickson.ii}, we have that $c_{n,n-1}=\frac{\widetilde{\Delta}_n(\ul{Q})}{\Moore{n}{Q}}$, where
   $$ \widetilde{\Delta}_n(\ul{Q}):=\det(\underline{Q},\underline{Q}^p,\cdots,\underline{Q}^{p^{n-2}},\underline{Q}^{p^n}).$$
     
Using this, we can compute the derivative $c_{n,n-1}'$:
    $$c_{n,n-1}'=\left(\frac{\widetilde{\Delta}_n(\ul{Q})}{\Moore{n}{Q}}\right)'
    =\frac{\widetilde{\Delta}_n(\ul{Q})'\Moore{n}{Q}-\widetilde{\Delta}_n(\ul{Q})\Moore{n}{Q}'}{\Moore{n}{Q}^2}.$$

Since  $\Moore{n}{Q}\neq 0$, the set of vectors $\{\underline{Q},\underline{Q}^p,\cdots,\underline{Q}^{p^{n-2}},\underline{Q}^{p^{n-1}} \}$ is a basis of the $k(X)$-vector space $k(X)^n$ and therefore there exist $a_0,\cdots a_{n-1}, b_0,\cdots,b_{n-1}\in k(X)$ such that 
   $$\underline{Q}'=\sum_{i=0}^{n-1}a_i\underline{Q}^{p^{i}} \;\;\;\mathrm{and}\;\;\;\underline{Q}^{p^n}=\sum_{i=0}^{n-1}b_i\underline{Q}^{p^{i}}.$$
   We first compute
   $$ \Moore{n}{Q}^p=\det(\underline{Q}^p,\underline{Q}^{p^2},\cdots,\underline{Q}^{p^{n-1}},\underline{Q}^{p^{n}})=\det(\underline{Q}^p,\underline{Q}^{p^2},\cdots,\underline{Q}^{p^{n-1}},\sum_{i=0}^{n-1}b_i\underline{Q}^{p^{i}})=(-1)^{n-1}b_0\Moore{n}{Q},$$
   and find that the coefficient $b_0$ is equal to $(-1)^{n-1}\Moore{n}{Q}^{p-1}$. 
Then, by multilinearity of the determinant we find that
   $$ \widetilde{\Delta}_n(\ul{Q})'=\det(\underline{Q}',\underline{Q}^p,\cdots,\underline{Q}^{p^{n-2}},\underline{Q}^{p^n})=\det\left(\sum_{i=0}^{n-1}a_i\underline{Q}^{p^{i}},\underline{Q}^p,\cdots,\underline{Q}^{p^{n-2}},\sum_{j=0}^{n-1}b_j\underline{Q}^{p^{j}}\right)=(a_0b_{n-1}-a_{n-1}b_0)\Moore{n}{Q}.$$
   With a similar computation, we get that 
   $$ \widetilde{\Delta}_n(\ul{Q})=\det\left(\underline{Q},\underline{Q}^p,\cdots,\underline{Q}^{p^{n-2}},\sum_{j=0}^{n-1}b_j\underline{Q}^{p^{j}}\right)=b_{n-1}\Moore{n}{Q}$$
   and
   $$ \Moore{n}{Q}'=\det(\underline{Q}',\underline{Q}^p,\cdots,\underline{Q}^{p^{n-2}},\underline{Q}^{p^{n-1}})=\det\left(\sum_{i=0}^{n-1}a_i\underline{Q}^{p^{i}},\underline{Q}^p,\cdots,\underline{Q}^{p^{n-2}},\underline{Q}^{p^{n-1}}\right)=a_0\Moore{n}{Q}.$$
By putting everything together, we conclude that $c_{n,n-1}'=-a_{n-1}b_0=-a_{n-1}(-1)^{n-1}\Moore{n}{Q}^{p-1}$. At the same time, we have 
   $$\det(\underline{Q}',\underline{Q},\underline{Q}^p,\cdots,\underline{Q}^{p^{n-2}})= \det\left(\sum_{i=0}^{n-1}a_i\underline{Q}^{p^{i}} ,\underline{Q},\underline{Q}^p,\cdots,\underline{Q}^{p^{n-2}}\right)=(-1)^{n-1}a_{n-1}\Moore{n}{Q}.$$
We then have
   $$ c_{n,n-1}'=-\det(\underline{Q}',\underline{Q},\underline{Q}^p,\cdots,\underline{Q}^{p^{n-2}})\cdot \Moore{n}{Q}^{p-2}.$$
   By Theorems \ref{thm:Pagotn} and \ref{thm:Pagotn0} the $n$-tuple $\underline{Q}$ gives rise to a basis of a space $L_{\lambda p^{n-1},n}$ if and only if $(P_{n-1}(Q_n)^{p-1})^{(p-1)}=-1$.
   By \ref{eq:Dickson.i}, this is equivalent to $c_{n,n-1}^{(p-1)}=-1$ and the computation of the first derivative above allows us to rewrite the condition as  
   $$ \left[\det(\underline{Q}',\underline{Q},\underline{Q}^p,\cdots,\underline{Q}^{p^{n-2}})\cdot \Moore{n}{Q}^{p-2}\right]^{(p-2)}=1,$$
from which the result follows.
   \end{proof}

\subsubsection{Many new examples of spaces $L_{\lambda 2^{n-1}, n}$}\label{sec:p=2}
As a first application, we assume $p=2$ and give a construction for all values of $n\geq 2$ of new large classes of examples of spaces $L_{\lambda 2^{n-1}, n}$ in characteristic $2$.
In this case, Proposition \ref{prop:det} tells us that the necessary and sufficient condition for the existence of a space $L_{\lambda 2^{n-1}, n}$ is the existence of a $n$-tuple of polynomials $\ul{Q}$ of same degree satisfying
$$\det(\underline{Q}',\underline{Q},\underline{Q}^p,\cdots,\underline{Q}^{2^{n-2}}) = 1$$

\begin{remark}\label{rmk:p=2;n=2}
    Note that when $n=2$ this easily leads to a complete classification of spaces $L_{2\lambda,2}$.
In fact, there are unique polynomials $U_1, U_2, V_1, V_2 
\in k[X]$ such that $Q_1 = U_1^2 +XV_1^2$ and $Q_2 = U_2^2 + XV_2^2$.
With these notations, we have $Q_1' = V_1^2$ and $Q_2' = V_2^2$, and hence
\[ \det \begin{pmatrix}
Q_1' & Q_1 \\
Q_2' & Q_2
\end{pmatrix} = \det \begin{pmatrix}
V_1^2 & U_1^2 +XV_1^2 \\
V_2^2 & U_2^2 + XV_2^2
\end{pmatrix} = \det \begin{pmatrix}
V_1^2 & U_1^2 \\
V_2^2 & U_2^2
\end{pmatrix} = \left( V_1U_2 + U_1V_2 \right)^2.\]

The condition imposed by Proposition \ref{prop:det} then is equivalent to
\[ V_1U_2 + U_1V_2 = 1.\]
If $\lambda$ is odd, then the polynomials $V_1, V_2$ have degree $\frac{\lambda-1}{2}$. By B\'ezout's theorem, if $V_1$ and $V_2$ are coprime, then there exists a unique pair $(U_1, U_2)$ with $\deg(U_i)< \frac{\lambda-1}{2}$ satisfying the above condition.
If we pick $(V_1, V_2)$ such that $V_1+V_2$ has also degree $\frac{\lambda-1}{2}$, and we let $(U_1, U_2)$ be the pair given by B\'ezout's theorem, then by Proposition \ref{prop:det} the polynomials $U_1^2 +XV_1^2$ and $U_2^2 +XV_2^2$ give rise to a basis of a space $L_{2\lambda, 2}$.
If we relax the condition $\deg(U_i) <\frac{\lambda-1}{2}$, then we have other pairs that satisfy Bezout's theorem: these can be obtained from the minimal one $(U_1, U_2)$ as $(U_1+aV_1, U_2+aV_2)$ for $a \in k$.\\
If $\lambda$ is even, then the argument above works by choosing coprime $U_1, U_2 \in k[X]$ of degree $\frac{\lambda}{2}$ with $U_1+U_2$ of degree $\frac{\lambda}{2}$ and applying B\'ezout's theorem to find $V_1, V_2$.
This shows that spaces $L_{2\lambda, 2}$ exist for every $\lambda$ and in large abundance.
While this is already known by \cite[Th\'eor\`eme 8]{Pagot02}, which relies on a study of the properties of the set of zeroes of $Q_1^2Q_2-Q_1Q_2^2$, the approach of Proposition \ref{prop:det} is of a significantly different nature.
\end{remark}

For $n>2$, the approach of Remark \ref{rmk:p=2;n=2} is not enough to give a complete classification, but we can extract sufficient conditions for the existence of spaces $L_{2^{n-1}\lambda, n}$ that lead to the discovery of new large classes of examples for every $n$.

\begin{proposition}\label{prop:UVdet}
Let $p=2$ and $n \geq 3$.
Let $\ul{Q} \in k[X]^n$ be a $n$-tuple giving rise to a basis of a space $\Omega$ as in Definition \ref{defn:Qi}.
For every $i=1, \dots, n$ let $U_i, V_i \in k[X]$ be such that
$Q_i = U_i^2 + X  V_i^2$ and assume that they satisfy the system of equations
\begin{equation}\label{eq:suffcond}
\begin{cases}  \det(\ul{U}, \ul{V}, \ul{U}^2, \dots, \ul{U}^{2^{n-2}}) = 1 &  \\
\det \left(\ul{U}, \ul{V}, ((1+\epsilon_3)\ul{U}+\epsilon_3 \ul{V})^2, \dots, ((1+\epsilon_n)\ul{U}+\epsilon_n \ul{V})^{2^{n-2}}\right) = 0,& 
\end{cases}
\end{equation}
where $(\epsilon_3, \dots, \epsilon_n)$ runs over all elements of $\F_2^{n-2} - \{0\}$. Then $\Omega$ is a space $L_{\lambda 2^{n-1},n}$.
\end{proposition}
\begin{proof}
We need to show that, if $\ul{U}$ and $\ul{V}$ satisfy (\ref{eq:suffcond}), then $\ul{Q}$ satisfy the condition of Proposition \ref{prop:det}.
This latter is equivalent to $\det \left( \ul{V}^2, \ul{U}^2, \ul{Q^2}, \dots, \ul{Q}^{2^{n-2}}\right) = 1$, which is in turn equivalent to
\[ \det \left( \ul{V}, \ul{U}, \ul{Q}, \dots, \ul{Q}^{2^{n-3}}\right) = \det \left( \ul{V}, \ul{U}, \ul{U}^2 + X  \ul{V}^2, \dots, \ul{U}^{2^{n-2}} + X  \ul{V}^{2^{n-2}}\right) =1. \]
By linearity of determinants, the above condition can be expressed as a polynomial condition in $X$, namely
\[ \sum_{(\epsilon_3, \dots, \epsilon_n) \in \F_2^{n-2}} \det \left(\ul{U}, \ul{V}, ((1+\epsilon_3)\ul{U}+\epsilon_3 \ul{V})^2, \dots, ((1+\epsilon_n)\ul{U}+\epsilon_n \ul{V})^{2^{n-2}}\right) \cdot X^{(\sum_i \epsilon_i 2^{i-3})}=1.\]
If the system (\ref{eq:suffcond}) is satisfied, then the non-constant coefficients of the polynomial above vanish and the constant term is equal to 1. 
Hence we can apply Proposition \ref{prop:det} to get that $\Omega$ is a space $L_{\lambda 2^{n-1},n}$.
\end{proof}

\begin{proposition}\label{prop:UV}
Let $n \geq 3$.
Let $U_1, \dots, U_n$ and $V_1, \dots, V_n$ be polynomials in $k[X]$ and  $\alpha, \beta \in k(X)$ such that:
\begin{enumerate}
\item $\ul{U} = \alpha\ul{V} + \beta \ul{V}^2$
\item $\beta^{2^{n-1}-1}\cdot \Moore{n}{V} =1$.
\end{enumerate}
Then $\ul{U}$ and $\ul{V}$ satisfy the system of equations (\ref{eq:suffcond}).
\end{proposition}
\begin{proof}
The proof consists of two steps:
\begin{itemize}
\item We first check that $\ul{U}$ and $\ul{V}$ satisfy the first line of (\ref{eq:suffcond}): by repeatedly using condition $(i)$ we find that
\begin{align*}
\det(\ul{U}, \ul{V}, \ul{U}^2, \ul{U}^4 \dots, \ul{U}^{2^{n-2}}) & = \det(\beta \ul{V}^2, \ul{V}, \ul{U^2}, \ul{U^4}, \dots, \ul{U}^{2^{n-2}}) \\
& = \beta \det( \ul{V}^2, \ul{V}, \beta^2\ul{V}^4, \ul{U^4}, \dots, \ul{U}^{2^{n-2}}) \\
& = \beta^3 \det( \ul{V}^2, \ul{V}, \ul{V}^4, \beta^4\ul{V}^8,  \dots, \ul{U}^{2^{n-2}}) \\
& = \beta^7 \det( \ul{V}^2, \ul{V},  \ul{V}^4, \ul{V}^8,  \dots, \ul{U}^{2^{n-2}})\\
& = \dots \\
& = \beta^{2^{n-1}-1}\det(\ul{V}^2, \ul{V}, \ul{V}^4, \ul{V}^8, \dots, \ul{V}^{2^{n-2}}) = \beta^{2^{n-1}-1}\cdot \Moore{n}{V} =1.
\end{align*}
\item We then check the equations in the second line of (\ref{eq:suffcond}): for each $(\epsilon_3, \dots, \epsilon_n) \in \F_2^{n-2} - \{0\}$ we let $k\in \{3,\dots, n\}$ be the smallest number such that $\epsilon_k \neq 0$.
Since we need every $\epsilon_i$ with $i<k$ to be equal to 0, the vectors 
$\left(\ul{U}, \ul{V}, ((1+\epsilon_3)\ul{U}+\epsilon_3 \ul{V})^2, \dots, ((1+\epsilon_k)\ul{U}+\epsilon_k \ul{V})^{2^{k-2}}\right)$ forming the first $k$ columns inside the determinant can be rewritten as $\left(\ul{U}, \ul{V}, \ul{U}^2, \dots, \ul{U}^{2^{k-3}},\ul{V}^{2^{k-2}}\right)$.
By condition $(i)$, we have that all these column vectors belong to the $k-1$ dimensional space generated by $\ul{V}, \dots, \ul{V}^{2^{k-2}}$ and hence 
\[\det \left(\ul{U}, \ul{V}, \ul{U}^2, \dots, \ul{U}^{2^{k-3}},\ul{V}^{2^{k-2}}, \dots, ((1+\epsilon_n)\ul{U}+\epsilon_n \ul{V})^{2^{n-2}}\right) = 0.\]
As this is true for every $(\epsilon_3, \dots, \epsilon_n) \in \F_2^{n-2} - \{0\}$, we have that all the equations in the second line of (\ref{eq:suffcond}) are satisfied.
\end{itemize}
\end{proof}

\begin{theorem}\label{thm:p=2generaln}
Let $k$ be an algebraically closed field of characteristic 2 and $n \geq 3$.\\
Let $\ul{W}:=~(W_1, \dots, W_n) \subset k[X]^n$ be a $n$-tuple of polynomials such that all the nonzero elements of $\langle W_1, \dots, W_n \rangle_{\F_2}$ are pairwise coprime and of the same degree, denoted by $d$ (in particular, we have that $\Moore{n}{W} \neq 0$).
If we set $V_i := \Moore{n-1}{\hat{W_i}}$ for every $i=1, \dots, n$, then there exists a rational function $\alpha \in \frac{1}{\Moore{n}{W}}k[X]$ satisfying the following properties:
\begin{enumerate}
\item The rational function \[Q_i := \frac{V_i^4}{\Moore{n}{W}^2} + (X+\alpha^2) V_i^2 \] is a polynomial for every $i=1, \dots, n$.
\item The $n$-tuple $\ul{Q}=(Q_1, \dots, Q_n) \in k[X]^n$ gives rise to a basis of a space $L_{2^{n-1}\lambda, n}$.
\end{enumerate}
Moreover, given such an $\alpha$, the set of all rational functions satisfying $(i)$ and $(ii)$ is $\{ \alpha + R |  R \in k[X] \}$.
\end{theorem}
\begin{proof}
We set $\beta = \frac{1}{\Moore{n}{W}}$ and remark that the expression of $Q_i$ given at the point $(i)$ is equal to $U_i^2 + X V_i^2$ where $U_i:=\alpha V_i + \beta V_i^2$.
As a result, to prove the theorem we need to define $\alpha \in k(X)$ such that both the conditions of Proposition \ref{prop:UV} are met. 
In fact, by virtue of that Proposition, it follows that $\ul{Q}$ satisfies Proposition \ref{prop:UVdet} and then gives rise to a basis of a space $L_{2^{n-1}\lambda, n}$.

We begin by observing that the hypothesis that $V_i= \Moore{n-1}{\hat{W_i}}$ for every $i$, combined with Theorem \ref{thm:FM4.1.} implies that
\[ \Moore{n}{V} = \Moore{n}{W}^{2^{n-1}-1}.\]
Hence condition $(ii)$ of Proposition \ref{prop:UV} is met.
It then remains to find an appropriate $\alpha \in k(X)$ such that $\alpha V_i + \beta V_i^2$ is a polynomial for every $i=1, \dots, n$.

We denote by $\frakW$ the space $\langle W_1, \dots, W_n \rangle_{\F_2}$ and for every subspace $\frakW' \subset \frakW$ we consider its structural polynomial $P_{\frakW'}(Y):=\prod_{W\in {\frakW'}} (Y-W) \in k(X)[Y]$ (cf. Definition \ref{defn:strucpoly}).
We then consider the spaces 
\[\frakW_i := \langle W_1, \dots, \widehat{W_i}, \dots, W_n \rangle_{\F_2} \;\; \mbox{for every} \; i=1, \dots, n,\] and note that every polynomial in $W_i + \frakW_i$ divides $P_{\frakW_i}(W_i)$ by definition.
Conversely, if $W \in \frakW$ divides $P_{\frakW_i}(W_i)$ then $W \in W_i + \frakW_i$.
In fact, it is clear that $W$ has a factor in common with a polynomial $W' \in W_i + \frakW_i$ and since any two distinct elements in $\frakW - \{0\}$ are coprime, then $W=W'$.

We then apply Lemma \ref{lem:FM2.2.} to get that $\beta V_i = \frac{\Moore{n-1}{\hat{W_i}}}{\Moore{n}{W}} = \frac{1}{P_{\frakW_i}(W_i)}$.
If we write $\alpha=\beta\gamma$ with $\gamma \in k[X]$, we then have that
\[U_i= \alpha V_i + \beta V_i^2 = \beta V_i (\gamma + V_i) = \frac{\gamma + V_i}{P_{\frakW_i}(W_i)}.\]

We now find the desired $\gamma$ as a solution of a system of congruences in $k[X]$ with coprime moduli:
for every $W \in \langle W_1, \dots, W_n \rangle - \{0\}$, we let $k_W:= \min \{ k\in \N: W \in \langle W_1, \dots, W_{k} \rangle \}$ and we consider the set of congruences
\begin{equation}\label{eq:cong}
 \left\{ \gamma \equiv V_{k_W}  \; \; \mod W \;| \;W \in \frakW- \{0\} \right\}.
 \end{equation}
We claim that this system is equivalent to the condition that $P_{\frakW_i}(W_i)$ divides $\gamma + V_i$ for every $i \in \{1, \dots, n\}$.
We start by proving the following statement:

\begin{align*}
(*)\; \text{If } W \in \langle W_1, \dots, W_n \rangle - \{0\} \text{ divides both } P_{\frakW_i}(W_i) \text{ and } P_{\frakW_j}(W_j), \text{ then } W \text{ divides } V_i + V_j.
\end{align*}

To prove $(*)$, let $i \neq j$ and set $\frakW_{ij} := \frakW_{i} \cap \frakW_{j}$.
If $W \in \langle W_1, \dots, W_n \rangle - \{0\}$ is a common divisor of $P_{\frakW_i}(W_i)$ and $P_{\frakW_j}(W_j)$ for $i \neq j$, then by the observation above $W \in (W_i + \frakW_i) \cap (W_j + \frakW_j)= W_i+W_j + \frakW_{ij}$.
Therefore, $W$ divides also $P_{\frakW_{ij}}(W_i+W_j)$.
Combining the additivity of $P_{\frakW_{ij}}$ with Lemma \ref{lem:FM2.2.}, we have 
\[ P_{\frakW_{ij}}(W_i+W_j) = P_{\frakW_{ij}}(W_i)+P_{\frakW_{ij}}(W_j)= \frac{V_i + V_j}{\Delta_{n-2}(W_1, \dots, \widehat{W_i}, \dots, \widehat{W_j}, \dots, W_n)},\]
which implies that $W$ divides $V_i+V_j$.

We choose now $\gamma$ verifying (\ref{eq:cong}) and $W$ dividing $P_{\frakW_j}(W_j)$. 
We have $W=W_j +Z_j$ with $Z_j\in \frakW_j$ and therefore $k_W\geq j$.
If $k_W=j$ then we have that $\gamma \equiv V_j \; \mod W$ by (\ref{eq:cong}).
If $k_W>j$ then set $i=k_W$. We have that $W=W_{i}+Z_{i}$ with $Z\in \frakW_{i}$ and therefore $W$ divides $P_{\frakW_{i}}(W_{i})$. Using $(*)$, we have that $W$ divides $V_i+V_j$ and by (\ref{eq:cong}) we conclude that 
$\gamma \cong V_i \mod W$  and therefore $\gamma \cong V_j \mod W$.
This shows that every solution $\gamma$ to (\ref{eq:cong}) satisfies 
\[\gamma \equiv V_j \; \mod W \mbox{ for all pairs } (j, W) \mbox{ with } W | P_{\frakW_j}(W_j).\]
As a result, $\gamma \equiv V_i \; \mod P_{\frakW_i}(W_i)$ for every $i \in \{1, \dots, n\}$.

Since the elements of $\frakW - \{0\}$ are pairwise coprime, we can apply the Chinese reminder theorem to find a unique solution $\gamma \in k[X]$  to (\ref{eq:cong}) such that $\deg(\gamma) < d (2^{n}-1)$.
Moreover, all the solutions to (\ref{eq:cong}) are of the form $\gamma + R \Moore{n}{W}$, for $R \in k[X]$. 
As a consequence, $\alpha$ satisfies $(i)$ and $(ii)$ if, and only if, $\alpha + R$ satisfies $(i)$ and $(ii)$.

In summary, for every choice of the polynomials $W_1, \dots, W_n$, and of a solution of (\ref{eq:cong}), this construction gives rise to a unique $n$-tuple of polynomials $U_1, \dots, U_n$ satisfying Proposition \ref{prop:UV}. 
Since the nonzero elements of $\frakW$ are all of the same degree, and $R$ is fixed, then also the resulting $Q_i=U_i^2 + XV_i^2$ are such that all the nonzero polynomials in $\langle Q_1, \dots, Q_n \rangle_{\F_2}$ are all of the same degree, denoted by $\lambda$.
The $n$-tuple $\ul{Q}$ satisfies the conditions of Proposition \ref{prop:UVdet} and then gives rise to a basis of a space $L_{\lambda 2^{n-1}, n}$.
\end{proof}

We want to investigate the parameter space of spaces $L_{\lambda 2^{n-1},n}$ arising from the construction of Theorem \ref{thm:p=2generaln}. For this, we use the notion of equivalence between spaces $L_{\lambda 2^{n-1},n}$ introduced in Definition \ref{defn:equiv}.
For every $d\geq0$, let $\calW_d \subset (k[X])^n$ be the quasi-affine variety consisting of elements $(W_1, \dots, W_n)$ of the same degree $d$ such that all the non-zero elements of $\calW:= \langle W_1, \dots, W_n \rangle_{\F_2}$ are of degree $d$ and pairwise coprime.
This is defined inside the space of coefficients $k^{(d+1)n}$ by the inequation 
\[\Delta_n(w_1, \dots, w_n) \prod_{W, W' \in \calW - \{0\}} \Res(W, W') \neq 0,\]
where, for every $i=1, \dots, n$, $w_i$ is the leading term of $W_i$.
Then the construction of Theorem \ref{thm:p=2generaln} associates with every element of $(W_1, \dots, W_n, R) \in \calW_d \times k[X]$ a $n$-tuple of polynomials $(Q_1, \dots, Q_n)$ giving rise to a basis of a space $L_{\lambda 2^{n-1}, n}$.
More precisely, the construction yields $Q_i=U_i^2 + X V_i^2$ with $V_i=\Delta_{n-1}(\hat{\ul{W_i}})$ and $U_i= (\alpha+R) V_i + \frac{V_i^2}{\Moore{n}{W}}$, where $\alpha$ denotes the unique proper rational function (i.e. of the form $\frac{\gamma}{\Moore{n}{W}}$ with $\gamma \in k[X]$ such that $\deg(\gamma)< \deg(\Moore{n}{W})$) satisfying conditions $(i)$ and $(ii)$ in Theorem \ref{thm:p=2generaln}.

\begin{corollary}\label{cor:WRp=2}
Let $n\geq 3$, $d\geq 0$ and let $\Omega$ and $\Omega'$ be the spaces $L_{\lambda 2^{n-1}, n}$ arising respectively from elements $(W_1, \dots, W_n, R)$ and $(W_1', \dots, W_n', R') \in\calW_d \times k[X]$ as in Theorem \ref{thm:p=2generaln}.
Then the following hold:
\begin{enumerate}
\item We have $\Omega=\Omega'$ if, and only if, $R=R'$ and $\langle W_1', \dots, W_n' \rangle_{\F_2} = \langle W_1, \dots, W_n \rangle_{\F_2}$.
\item We have that $\Omega$ is equivalent to $\Omega'$ if, and only if, there exists $b \in k$ such that
\[\langle W_1'(X), \dots, W_n'(X) \rangle_{\F_2} = \langle W_1(X+b^2), \dots, W_n(X+b^2) \rangle_{\F_2} \; \mbox{and} \; R'(X) = R(X) + b.\] 
\item If $\lambda \equiv 1 \mod (2^n-2)$, then there exist infinitely many equivalence classes of spaces $L_{\lambda 2^{n-1}, n}$. 
They arise from elements $(W_1, \dots, W_n, R) \in \calW_d \times k[X]$ such that $R$ is constant.
\item If $\lambda$ is even, then there exist infinitely many equivalence classes of spaces $L_{\lambda 2^{n-1}, n}$. 
They arise from elements $(W_1, \dots, W_n, R) \in \calW_d \times k[X]$ such that $\deg(R) \geq 1$.
\end{enumerate}
\end{corollary}
\begin{proof}
Let $\ul{Q}=(Q_1,\dots, Q_n)$ be the $n$-tuple arising from $(W_1, \dots, W_n, R)$ and $\ul{Q}'=(Q_1',\dots, Q_n')$ be the $n$-tuple arising from $(W_1', \dots, W_n', R')$. We recall that the writings
\begin{equation}\label{eq:QUV}
Q_i = U_i^2 + X V_i^2 \; \mbox{and} \; Q_i' = U_i'^2 + X V_i'^2
\end{equation}
are unique, and call $\ul{U}, \ul{U}', \ul{V}, \ul{V}'$ the $n$-tuples arising from these.
We now prove separately the statements of the corollary:
\begin{enumerate}
\item 
We first show that to have $\Omega=\Omega'$ it is necessary and sufficient to find a matrix $M \in GL_n(\F_2)$ such that $\ul{V}M = \ul{V}'$ and $\ul{U}M = \ul{U}'$.
In fact, by Proposition \ref{prop:QiMi} and \ref{prop:QiMiii} $\Omega=\Omega'$ if, and only if, there exists $M\in GL_n(\F_2)$ such that $\ul{Q}'=\ul{Q} M$.
By uniqueness of \ref{eq:QUV}, this is equivalent to have $(\ul{U}')^2=(\ul{U}^2)M$ and $(\ul{V}')^2=(\ul V^2)M$, and, since the entries of $M$ are elements of $\F_2$, it is equivalent to have that $\ul U'=\ul U M$ and $\ul{V}'=\ul{V} M$.

We then prove the two implications stated above:
\begin{itemize}
\item Let $R=R'$ and $\langle W_1', \dots, W_n' \rangle_{\F_2} = \langle W_1, \dots, W_n \rangle_{\F_2}$, so that there is $M \in GL_2(\F_2)$ such that $\ul{W}'=\ul{W} M$.
Then by Lemma \ref{lem:Anew}, we have that $\ul{V'} = \ul{V} M^c$.
Since $R=R'$, this implies that we also have that $\ul{U'} = \ul{U} M^c$ and hence $\Omega=\Omega'$.
\item Conversely, suppose that $\ul U'=\ul U M$ and $\ul{V}'=\ul{V} M$ for some $M \in GL_n(\F_2)$.
Then, applying $M^{-1}$ on both sides of the equation
\[U' = (\alpha' + R') \ul{V}' + \frac{(\ul{V}')^2}{\Moore{n}{W'}}\]
results in $\ul{U}=(\alpha'+R') \ul{V} + \frac{\ul{V}^2}{\Moore{n}{W'}}.$
We also have that $\ul{U}= (\alpha+R) \ul{V} + \frac{\ul{V}^2}{\Moore{n}{W}}$,
so that
\[(\alpha +R)+ \frac{\ul V}{\Moore{n}{W}}=(\alpha' +R')+ \frac{\ul V}{\Moore{n}{W'}}.\]

By rearranging the terms, we get that
 \[\left( \frac{1}{\Moore{n}{W}} - \frac{1}{\Moore{n}{W'}} \right) \cdot V_i = (\alpha'+R')-(\alpha+R) \;\; \forall \;\; i=1, \dots, n\]
and since $V_i \neq V_j$ if $i \neq j$, and $n\geq 3$, we deduce that $\Moore{n}{W}=\Moore{n}{W'}$ and $\alpha+R=\alpha'+R'$.

Note that $\ul{V} = \varphi(\ul{W})$ and $\ul{V'} = \varphi(\ul{W'})$, where $\varphi$ is the map defined in Proposition \ref{prop:FM5.1.}. 
Then, by applying this proposition, we have that
\[ \varphi(\ul{V}) = \Moore{n}{W}^{2^{n-2}-1}\ul{W}^{2^{n-2}} \; \mbox{and} \;  \varphi(\ul{V'}) = \Moore{n}{W'}^{2^{n-2}-1}\ul{W'}^{2^{n-2}}.\]
Since $\ul{V}'=\ul{V} M$, then we can apply Lemma \ref{lem:Anew} and get $\varphi(\ul{V}') = \varphi(\ul{V}) M^c$, which, combined with the equality $\Moore{n}{W}=\Moore{n}{W'}$ results in 
\[ \ul{W'}^{2^{n-2}} = \ul{W}^{2^{n-2}}M^c,\]
proving that $\langle W_1, \dots, W_n \rangle_{\F_2} = \langle W_1', \dots, W_n' \rangle_{\F_2}$.
Finally, from $\alpha+R=\alpha'+R'$ and the fact that there is a unique proper rational function in the set $\{ \alpha + R | R \in k[X]\}$, we obtain that $\alpha=\alpha'$ and $R=R'$.
\end{itemize}

\item If $\ul{W'}$ and $R'$ are as in the statement, then it is easy to see that they give rise to $\Omega'$ equivalent to $\Omega$, as we can see by applying the construction that $Q_i'(X)=Q_i(X+b^2)$ for all $i=1, \dots, n$.

Conversely, assume that there exist $a\in k^\times$ and $b \in k$ such that $Q_i(aX+b^2) = Q_i'(X)$.
We can see that
\begin{align*}
Q_i(aX+b^2)&=U_i(aX+b^2)^2+b^2V_i(aX+b^2)^2+aXV_i(aX+b^2)^2\\
 &=[U_i(aX+b^2)+bV_i(aX+b^2)]^2+X[\sqrt{a}V_i(aX+b^2)]^2,
\end{align*}
resulting in the relations
\[\begin{cases}
U_i'(X) = U_i(aX+b^2) + bV_i(aX+b^2) \\
V_i'(X)=\sqrt{a}V_i(aX+b^2).
\end{cases}\]
From the latter of these, combined with Proposition \ref{prop:FM5.1.}, we get that \[W_i'(X)= \theta a^{\frac{1}{2(2^{n-1}-1)}} W_i(aX+b^2)\] 
for $\theta \in k$ such that $\theta^{2^{n-1}-1}=1$.
In particular, we have that
\[\Moore{n}{W'} =\theta^{2^n - 1} a^{\frac{2^n-1}{2(2^{n-1}-1)}} \Moore{n}{W(aX+b^2)}.\]
We now compute $U_i'$ in two different ways.
On the one hand, we have
\begin{align*}
U_i' &= U_i(aX+b^2) + bV_i(aX+b^2) = (\alpha + R + b)V_i(aX +b^2) + \frac{V_i(aX+b^2)^2}{\Moore{n}{W(aX+b^2)}} \\
 &= \frac{(\alpha + R + b)}{\sqrt{a}} V_i'(X) + \frac{a^{-1} V_i'(X)^2}{\theta^{1-2^n} a^{\frac{1-2^n}{2(2^{n-1}-1)}} \Moore{n}{W'}} = \frac{(\alpha + R + b)}{\sqrt{a}} V_i'(X) + \theta^{2^n -1} a^{\frac{1}{2^n-2}} \frac{V_i'(X)^2}{\Moore{n}{W'}}.
\end{align*}
On the other hand, we have
\[  U_i' = (\alpha' + R')V_i'(X) + \frac{V_i'(X)^2}{\Moore{n}{W'}}\]
and since this both computations are true for all $i$, an argument analogue to the one used to prove $(i)$ shows that $\theta^{2^n -1} a^{\frac{1}{2^n-2}}= \theta a^{\frac{1}{2^n-2}}=1$.
But we have that $\theta^{2^n-2}=\theta^{2(2^{n-1}-1)}=1$ and hence $a=1$.
As a result, $\alpha + R + b = \alpha' + R'$ and hence $\alpha=\alpha'$ and $R' = R+b$.

We have shown so far that applying the construction of Theorem \ref{thm:p=2generaln} to the $n+1$-tuple $\left( W_1(X+b^2), \dots, W_n(X+b^2), R+b \right)$ produces $(Q_1', \dots, Q_n')$ giving rise to a basis of $\Omega'$ equivalent to $\Omega$.
By applying part $(i)$ we conclude that the same $\Omega'$ can only arise from $(W_1', \dots, W_n', R')$ with 
\[\langle W_1', \dots, W_n' \rangle_{\F_2} = \langle W_1(X+b^2), \dots, W_n(X+b^2) \rangle_{\F_2} \; \mbox{and} \; R' = R + b.\]
\item If $R$ is constant, we have
\[ \lambda=1+2 \deg V_i=1+2d(2^{n-1}-1) = 1 + d(2^n-2).\]
For every $\lambda \equiv 1 \mod 2^n-2$ we can then construct infinite equivalence classes of spaces $L_{\lambda 2^{n-1}, n}$ by picking $d=\frac{\lambda-1}{2^n-2}$ and all possible $n$-tuples $(W_1, \dots, W_n) \in \calW_d$.
\item If $R$ is not constant, then $\deg(U_i) > \deg(V_i)$ and $\lambda$ is even. 
In fact, any even value of $\lambda$ can be achieved in this way, simply by choosing $W_1, \dots, W_n$ to be constant and $R$ to be of degree $\frac{\lambda}{2}$.
\end{enumerate}
\end{proof}

\begin{remark}\label{rmk:L_2lambda,2}
Even though in Theorem \ref{thm:p=2generaln} we assumed $n \geq 3$, the construction of the $n$-tuple $\ul{Q}$ in its proof makes sense also for $n=2$.
Namely, for every pair $(W_1, W_2) \in \calW_d$ one can set $V_1=W_2$, $V_2=W_1$ and fix a solution $\gamma$ to the congruences
\[\begin{cases}
\gamma \equiv V_1\mod V_2 \\
\gamma \equiv V_1 \mod V_1+ V_2\\
\gamma \equiv V_2 \mod V_1.
\end{cases}\]
By considering 
\[U_1 = \frac{(\gamma + V_1)}{V_2(V_1+V_2)} \; \; \mbox{and} \; \; U_2 = \frac{(\gamma + V_2)}{V_1(V_1+V_2)},\]
one produces a pair $(U_1^2+X V_1^2, U_2^2+XV_2^2)$ giving rise to a space $L_{2\lambda, 2}$.
We can show that this construction recovers all the spaces $L_{2\lambda, 2}$.
In fact, starting from pairs $\ul{U}, \ul{V}$ satisfying Bezout's identity $U_1V_2+V_1U_2=1$ (cf. the discussion following Proposition \ref{prop:det} to see that these define all possible spaces $L_{2\lambda, 2}$), we have that 
\[\begin{cases}
U_1V_2 \equiv 1 \mod V_1 \\
U_2V_1 \equiv 1 \mod V_2.
\end{cases}\]
Using these, we verify that the polynomial 
\[\gamma=V_1^2U_2+V_2^2U_1= V_1+V_2U_1(V_1+V_2)=V_2+U_2V_1(V_1+V_2)\]
solves the congruences (\ref{eq:cong}) and yields
\[U_1 = \frac{(\gamma + V_1)}{V_2(V_1+V_2)} \; \; \mbox{and} \; \; U_2 = \frac{(\gamma + V_2)}{V_1(V_1+V_2)}\]
as above. For every choice of $\ul{U}$ and $\ul{V}$ we can find such a $\gamma$: hence we can recover in this way all the spaces $L_{2\lambda, 2}$.
\end{remark}

\begin{remark}\label{rmk:p=2pullpack}
We remark that, if $\Omega$ a space $L_{2^{n-1}\lambda, n}$ is arising from an element $(W_1, \dots, W_n, R) \in \calW_d \times k[X]$ as in Theorem \ref{thm:p=2generaln}, then every \'etale pullback of $\Omega$ also arise from this construction.
More precisely, let $S(X)\in k[X]$ with $S'(X)=1$ and let $\sigma^\star(\Omega)$ be the pullback of $\Omega$ with respect to the morphism $\sigma \in \mathrm{End}(\PP^1_k)$ induced by $X \mapsto S(X)$ (cf. Lemma \ref{lem:pullbacketale}).
Then, we know by Proposition \ref{prop:pullbackQi} that $(Q_1(S(X)), \dots, Q_n(S(X)))$ is the $n$-tuple arising from Theorem \ref{thm:Pagotn} for $\sigma^\star(\Omega)$.
One verifies that $Q_i(S(X))= (U_i(S(X))+RV_i(S(X))^2 + X V_i(S(X))^2$, and as a result one proves that $\sigma^\star(\Omega)$ is the space arising from $( W_1\circ S, \dots, W_n \circ S, R+S-X)$.
\end{remark}

\subsubsection{Non-existence of spaces $L_{p^{n-1},n}$ in odd characteristic}
Another application of Proposition \ref{prop:det} is the following theorem, which settles the case $\lambda=1$ for every $n$ and every $p$.

\begin{theorem}\label{thm:lambda=1}
Let $p>2$. Then, there exist no space $L_{p^{n-1},n}$.
\end{theorem}
\begin{proof}
Let us prove this by contradiction. Assuming that such a space exist and applying Proposition \ref{prop:det} there exist a $n$-tuple $Q_1, \dots, Q_n$ of polynomials of degree 1 satisfying 
\begin{equation}\label{eq:det}
    \left( \det(\underline{Q}',\underline{Q},\underline{Q}^p,\cdots,\underline{Q}^{p^{n-2}})\cdot \Moore{n}{Q}^{p-2} \right)^{(p-2)}=1.
\end{equation}
We have $Q_i=u_iX+v_i$ with $\Moore{n}{u}\neq 0$ and so clearly $\underline{Q}''=0$, from which one sees that $\det(\underline{Q}',\underline{Q},\underline{Q}^p,\cdots,\underline{Q}^{p^{n-2}})'=0$ and that $\Moore{n}{Q}''=0$. Equation (\ref{eq:det}) then becomes
$$\det(\underline{Q}',\underline{Q},\underline{Q}^p,\cdots,\underline{Q}^{p^{n-2}})\cdot \left( \Moore{n}{Q}^{p-2} \right)^{(p-2)}=1. $$
       Using the fact that  $\Moore{n}{Q}''=0$, we can compute by induction on $j$ that  
       $$\left(\Moore{n}{Q}^{p-2}\right)^{(j)}=(-1)^j(j+1)!(\Moore{n}{Q}')^j\Moore{n}{Q}^{p-2-j} \;\; \mbox{for every } \;\; 1 \leq j \leq p-2.$$
       In particular, we have
       $$ \left(\Moore{n}{Q}^{p-2}\right)^{(p-2)}=(\Moore{n}{Q}')^{p-2}$$
       and Equation (\ref{eq:det}) becomes
       $$ \det(\underline{Q}',\underline{Q},\underline{Q}^p,\cdots,\underline{Q}^{p^{n-2}})\cdot(\Moore{n}{Q}')^{p-2} =1.$$
In particular, the polynomial $(\Moore{n}{Q}')^{p-2}$ is of degree zero. It follows that $\Moore{n}{Q}'$ is also of degree zero since we are assuming that $p\geq 3$.
However, when we express this constant polynomial in terms of $\ul{u}$ and $\ul{v}$ we get
   $$ \Moore{n}{Q}'=\det(\underline{u},\underline{u}^pX^p+\underline{v}^p,\cdots,\underline{u}^{p^{n-1}}X^{p^{n-1}}+\underline{v}^{p^{n-1}})$$
   and by multilinearity of determinant, the term of higher degree of $\Moore{n}{Q}'$ is $\Moore{n}{u}X^{p+p^2+\cdots+p^{n-1}}.$
   Since $\Moore{n}{u}\neq 0$ by hypothesis, we get a contradiction and the result follows.
\end{proof}

\begin{remark}
When $n=2$, Theorem \ref{thm:lambda=1} corresponds to the first part of \cite[Th\'eor\`eme 9]{Pagot02} which says that, for $p>2$, there are no spaces $L_{p,2}$.
This is enough to give an obstruction to lifts of certain $(\Z/p\Z)^2$ extensions, thanks to \cite[Theorem III.3.1]{GreenMatignon99} which implies that any lift of a $(\Z/p\Z)^2$--extension $k[\![z]\!]/k[\![t]\!]$ with a single ramification jump of conductor $m+1=p$ has equidistant branch locus.
Following this strategy, one achieves \cite[Th\'eor\`eme 13]{Pagot02}, which states that for $p>2$ lifts of $(\Z/p\Z)^2$--extensions with a single ramification jump of conductor $m+1=p$ do not exist.

Unfortunately, the generalization of \cite[Theorem III.3.1]{GreenMatignon99} to $m+1>p$ is not true, and one needs to consider also lifts with branch loci having more complicated geometries.
Work in progress by Pagot addresses this issue by considering generalisations of spaces $L_{m+1,n}$ that correspond to these geometries, achieving a proof that lifts of $(\Z/p\Z)^n$--extensions with a single ramification jump of conductor $m+1=p^{n-1}$ do not exist.
\end{remark}



\section{Standard $L_{(p-1) p^{n-1},n}$ spaces and their subspaces}

In this section, we analyze certain spaces $L_{\lambda p^{n-1},n}$ for $\lambda=p-1$ that appear implicitly in \cite{Matignon99}.
In that paper, they are instrumental to lift to characteristic zero certain local actions of elementary abelian groups and they have since been made explicit in \cite{Pagot02} as well as in \cite{FresnelMatignon23}, in a generalized form.
Due to their relevance, and for the fact that they were among the first examples of spaces $L_{\lambda p^{n-1},n}$ to be discovered, we choose to give these spaces a name and we call them \emph{standard $L_{(p-1) p^{n-1},n}$ spaces}.
In the first subsection, we explain how to construct these spaces using Theorem \ref{thm:Pagotn}.
Namely, we find sufficient conditions for a set of polynomial $Q_1, \dots, Q_n $ as in the theorem to produce a standard space.
In the second subsection, we show that we can construct all the subspaces of a given standard space using étale pullbacks of other standard spaces.

\subsection{Standard spaces}\label{sec:standard}

We begin by giving the definition of a standard $L_{p^{n-1}(p-1),n}$ space.
We recall from Corollary \ref{coro:poles} that the set of poles of a space $L_{\lambda p^{n-1},n}$ characterizes the space itself.

\begin{definition}\label{defn:standard}
A \emph{standard $L_{p^{n-1}(p-1),n}$ space} $\Omega$ is a space $L_{p^{n-1}(p-1),n}$ such that there exists a $n$-tuple $\ul{a}:=(a_1,a_2,\cdots, a_n)\in k^n$ with $\Moore{n}{a} \neq 0$ for which $\calP(\Omega)=\langle a_1,a_2,\cdots, a_n \rangle_{\Fp} - \{0\}$.
\end{definition}

For $i \in \{1, \dots, n\}$ and $\ul{a}$ as above, let
\begin{equation}\label{eq:standard} \omega_i:= \sum_{(\epsilon_1, \dots, \epsilon_n) \in {\Fp}^n} \frac{\epsilon_i}{X - \sum_{j=1}^n \epsilon_j a_j} dX.
\end{equation}
Then $\langle \omega_1, \dots, \omega_n\rangle_{\Fp}$ is a standard space by \cite[Proposition 3.1]{FresnelMatignon23}.
This shows that, for every $\ul{a}\in k^n$ whose entries are $\Fp$-linearly independent, there exists a standard space whose set of poles is $\langle a_1,a_2,\cdots, a_n \rangle_{\Fp} - \{0\}$.

We can recover a description of standard spaces using the polynomials $Q_i$'s of Theorem \ref{thm:Pagotn}, as we establish in the following proposition

\begin{proposition}\label{prop:standard}
Let $\omega_1, \dots, \omega_n$ be defined by the identities \eqref{eq:standard}, and let \[Q_i:= \mu \frac{\Delta_{2}(a_i, X)}{X} =\mu(a_iX^{p-1}-a_i^p) \] with $\mu^{p^{n-1}}=\frac{-1}{\Delta_n(\underline a)^{p-1}}$.
Then we have $\omega_i=\frac{P_i}{P}dX$ where $P:=\Delta_n(\underline{Q})$ and $P_i:=~(-1)^{i-1}\Delta_{n-1}(\underline {\hat Q_i})$.
\end{proposition}

\begin{proof}
First, we compute 
\[\frac{P_i}{P}=(-1)^{i-1}\left(\frac{X}{\mu}\right)^{p^{n-1}}\frac{\Delta_{n-1}(\ul{\widehat{\Delta_2(a_i, X)}})}{\Delta_{n}(\Delta_2(a_1, X), \dots, \Delta_2(a_n, X))},\] which combined with Lemma
\ref{lem:Moore} (for $m=1$) gives 
\[\frac{P_i}{P}=
\frac{1}{\mu^{p^{n-1}}}\frac{(-1)^{i-1}\Delta_{n}(\ul{\hat{a}}_i, X)}{\Delta_{n+1}(\ul{a}, X)}.\]

Using \cite[Proposition 3.1]{FresnelMatignon23}, we can write
\[\omega_i = -\Moore{n}{a}^{p-1} \frac{(-1)^{i-1}\Delta_{n}(\ul{\hat{a}}_i, X)}{\Delta_{n+1}(\ul{a}, X)}dX \;\; \forall \; \; i=1, \dots, n,\]
and this proves the proposition.
\end{proof}

\begin{example}\label{ex:p(p-1)}
Let $n=2$ and choose $a_1=1$, $a_2 \in \F_{p^2} - \F_p$.

Then we have 
\[\begin{cases} 
\omega_1 &= \displaystyle \sum_{i=0}^{p-1} \left(\frac{1}{X-(1+i a_2)}+ \frac{2}{X-(2+i a_2)} +\dots+\frac{p-1}{X-(p-1+i a_2)} \right) dX \\
\omega_2 &= \displaystyle \sum_{i=0}^{p-1} \left(\frac{1}{X-(i+ a_2)}+ \frac{2}{X-(i+2 a_2)} +\dots+ \frac{p-1}{X-(i+(p-1) a_2)} \right) dX.
\end{cases}\]
Each $\omega_i$ has $p(p-1)$ poles. 
The poles in $\calP(\Omega)$ are all the elements of the multiplicative group $\F_{p^2}^\times$ and those in common between $\omega_1$ and $\omega_2$ are those elements of $\F_{p^2}$ that neither belong to the one-dimensional $\F_p$-vector space generated by $1$ nor to the one generated by $a_2$.
\end{example}

\begin{remark}
Let us fix a $n$-tuple $\ul{a} \in k^n$ such that $\Moore{n}{a} \neq 0$ and let us denote by $A:=\langle a_1, \dots, a_n \rangle_{\Fp}$ the $\Fp$-vector space generated by $\ul{a}$ and by $A^\star$ its dual.
We can make several remarks:
\begin{enumerate}
\item 
If $\Omega$ is the standard space with $\calP(\Omega) = A - \{0\}$, there is a natural pairing
\[
\begin{tikzcd}[row sep=0.05cm]
\psi: \Omega \times A  \arrow[r] & \Fp \\
(\omega, x)   \arrow[r, mapsto] & 
\begin{cases}
        \res_x(\omega)  \; \; \text{if} \;\; x \neq 0, \\
        0  \; \; \text{if} \;\; x = 0.
    \end{cases}
\end{tikzcd}\]

We remark that this is a perfect pairing.

In fact, writing $x=\epsilon_1 a_1 + \dots + \epsilon_n a_n$, we deduce from equation \eqref{eq:standard} that
\[ \res_x \left(\sum_i \alpha_i \omega_i \right)= \sum_i \alpha_i \epsilon_i,\]

which shows that $\psi$ is bilinear and that $\psi(\omega_i,a_j)=\delta_{ij}$, the Kronecker symbol.
Hence $\psi$ is perfect and, as a result, the homomorphism $\omega$ to $x \mapsto \psi(\omega,x)$ defines an isomorphism $\iota: \Omega \cong A^\star$.
We can explicitly describe the inverse of $\iota$ as follows:
for all $\varphi\in A^\star-\{0\}$ define the differential form $\omega_\varphi:=\sum_{a\in A}\frac{\varphi(a)dX}{X-a}$.
Then the isomorphism
 \[
\begin{tikzcd}[row sep=0.05cm]
\iota': A^\star  \arrow[r] & \Omega \\
\varphi  \arrow[r, mapsto] & 
\begin{cases}
        \omega_\varphi  \; \; \text{if} \;\; \varphi \neq 0, \\
        0  \; \; \text{if} \;\; \varphi = 0.
    \end{cases}
\end{tikzcd}\]
satisfies $\iota\circ \iota' = \id$ and $\iota' \circ \iota = \id$.
\item Using Definition \ref{defn:standard} it is easy to show that the Frobenius twist of the standard $L_{\lambda p^{n-1},n}$ space associated to the vector space $A=\langle a_1, \dots, a_n \rangle_{\Fp}$ is the standard space associated to $\Phi(A)=\langle a_1^p, \dots, a_n^p \rangle_{\Fp}$.
\item We can apply a translation to a standard space $\Omega$ to get a space $L_{p^{n-1}(p-1),n}$ \emph{equivalent} to $\Omega$, in the sense of the definition at the beginning of Section \ref{sec:class}.
This will not be in general a standard space.
\end{enumerate}
\end{remark}

\subsection{Subspaces of standard spaces via  \'etale pullbacks}

Let $A$ be a $n$-dimensional $\Fp$-vector subspace of $k$.
The construction of Section \ref{sec:standard} results in a map
\[ A \to \Omega(A) \]
associating with it the standard $L_{p^{n-1}(p-1),n}$ space whose set of poles is $A - \{0\}$.
It is easy to see that every subspace of $\Omega(A)$ of dimension $t<n$ is a space $L_{p^{n-1}(p-1),t}$, and that it can not be equivalent to a standard space.

However, we have seen in Lemma \ref{lem:pullbacketale} how to find spaces $L_{\lambda d p^{n}, n}$ starting from spaces $L_{\lambda p^{n-1}, n}$ via the pullback of the differential forms under suitable morphisms of degree $dp$.
In this section, we characterize the subspaces of the standard spaces constructed in \S \ref{sec:standard} as \'etale pullbacks of standard spaces of lower dimension.
More precisely, let $n>1$, $1\leq t\leq n$ and $A :=\langle a_{1},\cdots, a_n \rangle_{\Fp}\subset k$ with $\underline {a}:=(a_1,a_2,\cdots, a_n)$ such that $\Delta_n(\ul{a})\neq 0$. 
Let $A_{n-t}= \langle a_1,\dots,a_{n-t} \rangle_{\F_p}$, $A_{n-t}^s=\langle a_{n-t+1},\dots,a_n \rangle_{\F_p}$ and let $P_{A_{n-t}}$ be the structural polynomial of $A_{n-t}$ (cf. Definition \ref{defn:strucpoly}).
For every $n-t+1\leq i\leq n$, we set $\widetilde {a}_i:=P_{A_{n-t}}(a_i)$.
Then, by Lemma \ref{lem:structuralMoore} we have that $\Delta_t(\ul{\tilde{a}})\neq 0$ and hence the $\F_p$-vector space $\widetilde{A_{n-t}^s}: = \langle \tilde a_{n-t+1},\dots,\tilde a_{n} \rangle_{\F_p}$ is $t$-dimensional.
Let $\Omega(A)$ be the standard space associated with $A$ and let \[\omega_i= \sum_{(\epsilon_1, \dots, \epsilon_n) \in {\F_p^n}} \frac{\epsilon_i}{X - \sum_{j=1}^n \epsilon_j a_j} dX,\ 1\leq i\leq n\] be the elements of its usual basis.
Let $Q_1, \dots, Q_n$ be the $n$-tuple of polynomials arising from Theorem \ref{thm:Pagotn}, which we know by proposition \ref{prop:standard} to satisfy $Q_i=\mu \frac{\Delta_{2}(a_i, X)}{X}$ with $\mu^{p^{n-1}}=\frac{-1}{\Delta_n(\underline a)^{p-1}}$.
Similarly, we let 
\[\tilde{\omega}_i= \sum_{(\epsilon_{n-t+1}, \dots, \epsilon_n) \in {\F_p^t}} \frac{\epsilon_i}{X - \sum_{j=n-t+1}^n \epsilon_j \tilde{a}_j} dX,\ n-t+1\leq i\leq n\] be the elements of the usual basis of the standard $L_{p^{t-1}(p-1),t}$-space $\Omega(\widetilde{A_{n-t}^s})$.
Finally, we let $\widetilde{Q_{n-t+1}}, \dots, \widetilde{Q_{n}}$ be the $t$-tuple of polynomials arising from Theorem \ref{thm:Pagotn}, which satisfy $\widetilde{Q_i}=\tilde\mu \frac{\Delta_{2}(\tilde{a_i}, X)}{X}$ with $\tilde\mu^{p^{t-1}}=\frac{-1}{\Delta_t(\tilde a_{n-t+1},\dots,\tilde a_{n} )^{p-1}}$.

\begin{proposition}\label{prop:standardsubsp} 
Consider the subspace $\Omega^s= \langle \omega_{n-t+1}, \dots, \omega_n \rangle_{\Fp} \subset\Omega(A)$. 

\begin{enumerate}
\item We have that $\Omega^s=\sigma^\star(\Omega(\widetilde{A_{n-t}^s})).$ More precisely, for every $n-t+1\leq i\leq n$, $\omega_i$ is equal to $\sigma^\star(\tilde\omega_{i})$, the pullback of $\tilde\omega_{i}$ with respect to the morphism $\PP^1_k \overset{\sigma}{\rightarrow} \PP^1_k $ induced by $X \mapsto P_{A_{n-t}}(X)$.

\item The $t$-tuple of polynomials arising from Theorem \ref{thm:Pagotn} for the basis $(\sigma^\star\tilde\omega_{n-t+1},\dots,\sigma^\star\tilde\omega_n)$ of $\Omega^s=\sigma^\star(\Omega(\widetilde{A_{n-t}^s}))$ is given by
\[\eta\widetilde{ Q_{n-t+1}}(P_{A_{n-t}}), \dots, \eta\widetilde{ Q_{n}}(P_{A_{n-t}}),\]
with $\eta$ satisfying $\eta^{p^{t-1}}=\frac{1}{P_{A_{n-t}}'}=\frac{(-1)^{n-t}}{\Delta_{n-t}(a_1, \dots, a_{n-t})^{p-1}}.$

These polynomials satisfy the condition 
\[(-1)^{n-t}\eta\tilde{Q_{i}}(P_{A_{n-t}} )=P_{\calQ_{n-t}}(Q_{i}) \;\; \forall \; n-t+1\leq i\leq n,\]
where $P_{\calQ_{n-t}}$ is the structural polynomial of the $\F_p$-vector space $\calQ_{n-t}:=\langle Q_{1},\dots,Q_{n-t}\rangle_{\F_p}$.
\end{enumerate}
\end{proposition}

\begin{proof}
\begin{enumerate}
    \item By Lemma \ref{lem:FM2.2.}, we have that $P_{A_{n-t}}'(X)= (-1)^{n-t}\Delta_{n-t}(a_1, \dots, a_{n-t})^{p-1}\in k^\times$, so that the morphism $\sigma$ gives rise to an etale pullback as in Section \ref{sec:Frobetale}.
 Since the degree of $P_{A_{n-t}}$ is $p^{n-t}$, it follows that $\sigma^\star(\Omega(\widetilde{A_{n-t}^s}))$ is a space $L_{(p-1) p^{n-1},t}$. 
For every $n-t+1\leq i\leq n$, we remark that $\tilde\omega_i=\frac{dF_i}{F_i}$ with 
\[ F_i(X) =\prod_{(\epsilon_{n-t+1}, \dots, \epsilon_n) \in {\F_p^t}}(X - \sum_{j=n-t+1}^n \epsilon_j \tilde{a}_j)^{\epsilon_i} \\
 =\prod_{(\epsilon_{n-t+1}, \dots, \epsilon_n) \in {\F_p^t}}(X - \sum_{j=n-t+1}^n \epsilon_j P_{A_{n-t}}(a_j))^{\epsilon_i}.\]
By using additivity of $P_{A_{n-t}}$, we have that
\begin{align*}
    F_i(P_{A_{n-t}}(X))& =\prod_{(\epsilon_{n-t+1}, \dots, \epsilon_n) \in {\F_p^t}} \left( P_{A_{n-t}}(X-\sum_{j=n-t+1}^n \epsilon_j a_j )\right)^{\epsilon_i} \\
    & =\prod_{a\in A_{n-t}}\prod_{(\epsilon_{n-t+1}, \dots, \epsilon_n) \in {\F_p^t}}(X-(a+\sum_{j=n-t+1}^n \epsilon_j a_j))^{\epsilon_i}\\
    &=\prod_{(\epsilon_1,\dots,\epsilon_n)\in\F_p^n}(X-\sum_{j=1}^n\epsilon_ja_j)^{\epsilon_i}.
\end{align*}
As a result, $\sigma^\star(\tilde\omega_i)=\frac{dF_i(P_{A_{n-t}})}{F_i(P_{A_{n-t}})}$ has the same set of poles as $\omega_i$, and each pole has the same residue for both forms.
By applying Lemma \ref{lem:poles1}, we conclude that $\sigma^\star(\tilde\omega_i)= \omega_i$.
\item By Proposition \ref{prop:pullbackQi}, the $t$-tuple of polynomials arising from Theorem \ref{thm:Pagotn} for the basis $(\sigma^\star\tilde\omega_{n-t+1},\dots,\sigma^\star\tilde\omega_n)$ of $\sigma^\star(\Omega(\widetilde{A^s_{n-t}}))$ is 
$\eta\widetilde{ Q_{n-t+1}}(P_{A_{n-t}}), \dots, \eta\widetilde{ Q_{n}}(P_{A_{n-t}})$,
with $\eta^{p^{t-1}}=\frac{1}{P_{A_{n-t}}'}=\frac{(-1)^{n-t}}{\Delta_{n-t}(a_1, \dots, a_{n-t})^{p-1}}$, as stated.\\
From Proposition \ref{prop:subspacestruc}, we have that the $t$-tuple of polynomials arising from Theorem \ref{thm:Pagotn} for the basis $\big((-1)^{n-t}\omega_{n-t+1},\dots,(-1)^{n-t}\omega_n\big)$ of $\Omega^s$ is $\big(P_{\calQ_{n-t}}(Q_{n-t+1}), \dots, P_{\calQ_{n-t}}(Q_{n})\big)$.
We can apply Proposition \ref{prop:QiMi} to find that the $t$-tuple of polynomials arising from Theorem \ref{thm:Pagotn} for the basis $\big(\omega_{n-t+1},\dots,\omega_n\big)$ of $\Omega^s$ is $\big((-1)^{n-t}P_{\calQ_{n-t}}(Q_{n-t+1}), \dots, (-1)^{n-t}P_{\calQ_{n-t}}(Q_{n})\big)$.
Since by $(i)$ we know that $\sigma^\star(\tilde\omega_i)= \omega_i$, and Proposition \ref{prop:QiMii} ensures the uniqueness of the $t$-tuples of polynomials arising from Theorem \ref{thm:Pagotn} for the same basis, 
we have that
\[(-1)^{n-t}\eta\tilde{Q_{i}}(P_{A_{n-t}} )=P_{\calQ_{n-t}}(Q_{i}) \;\; \forall \; n-t+1\leq i\leq n.\]
\end{enumerate}

\end{proof}

\begin{example}
Let $n=3$, and $t=2$.
Choose $A':=~\langle \alpha, \beta \rangle_{\Fp}~\subset~k$ of dimension 2 and consider the corresponding standard $L_{p(p-1),2}$ space $\Omega(A')$. Let $a_1 \in k^\times$, and define $S(X):=X^p-a_1^{p-1} X$.

Since $k$ is algebraically closed, we can choose a solution $a_2$ to the equation
$S(X)= \alpha$, 
as well as a solution $a_3$ to the equation
 $S(X) = \beta$.
We remark that 
\[S(f_1 a_1 +f_2 a_2 +f_3 a_3) = f_2 S(a_2) +f_3 S(a_3) = f_2 \alpha+ f_3 \beta \;\;\; \forall \; f_1,f_2,f_3 \in \F_p,\]
and this vanishes only when $f_2= f_3 = 0$.
In particular, if $(f_2,f_3) \neq (0,0)$ then $f_1 a_1 +f_2 a_2 +f_3 a_3 \neq 0$.
We then have that $\Delta_3(a_1,a_2,a_3) \neq 0$, so we can consider the standard space $\Omega(A)$ with $A:=\langle a_1, a_2, a_3 \rangle_{\Fp}$ and Proposition \ref{prop:standardsubsp} tells us that the pullback of $\Omega(A')$ with respect to the morphism induced by $X \mapsto S(X)$ is a two dimensional subspace of $\Omega(A)$.

Note that, by making different choices of $a_2, a_3$ we still end up with the same vector space $A$.
Hence, the datum of the space $A'$ together with the element $a_1 \in k^\times$ is enough to determine the space $A$.
\end{example}

Proposition \ref{prop:standardsubsp} proves that subspaces of standard spaces can be realized as \'etale pullbacks of standard spaces.
While a complete characterization of the spaces that can be obtained as \'etale pullbacks of standard spaces is outside the scope of this paper, we include here a lemma that gives a necessary condition for a space $L_{\lambda p^{n-1}, n}$ to be equivalent to such a pullback.

\begin{lemma}\label{lem:etale}
Let $\Omega$ be a space $L_{\lambda p^{n-1}, n}$ obtained as in Lemma \ref{lem:pullbacketale} via an \'etale pullback of $\Omega_1$ a space $L_{\lambda_1 p^{n-1}, n}$. If $\Omega$ can also be obtained via an \'etale pullback of $\Omega_2$, a $L_{\lambda_2 p^{n-1}, n}$ space  with $\lambda_1 < \lambda_2$ then $p |\lambda_2$. 
\end{lemma}
\begin{proof}
For $i=1,2$ let $S_i$ be the polynomial defining the \'etale pullback of $\Omega_i$.
We write
\[ S_1(X)= a_{d_1}X^{d_1p}+ a_{d_1-1}X^{(d_1-1)p}+ \dots + \gamma_1 X + a_0 \]
\[ S_2(X)= b_{d_2}X^{d_2p}+ b_{d_2-1}X^{(d_2-1)p}+ \dots + \gamma_2 X + b_0 \]
with non-zero leading coefficients.
Let $Q_{11}, \dots, Q_{n1}$ be a $n$-tuple of polynomials giving rise to a basis of $\Omega_1$.
Then by Proposition \ref{prop:pullbackQi}, up to multiplying by a constant, $Q_{11}(S_1(X)), \dots, Q_{n1}(S_1(X))$ give rise to a basis of $\Omega$.
Similarly, if $Q_{12}, \dots, Q_{n2}$ give rise to a basis of $\Omega_2$ then (up to multiplying by a constant) $Q_{12}(S_2(X)), \dots, Q_{n2}(S_2(X))$ give rise to a basis of $\Omega$, as well.
Up to changing the basis of $\Omega_2$, we can assume that the bases of $\Omega$ produced in the two cases are the same, and then apply Proposition \ref{prop:QiMii} to get that $Q_{j1}(S_1(X))=\eta Q_{j2}(S_2(X))$ for every $j=1, \dots, n$ and $\eta \in k^\times$.
By definition we have $\deg(Q_{j1})=\lambda_1$ and $\deg(Q_{j2})=\lambda_2$, and the condition $\lambda_1 < \lambda_2$ implies that $d_1 > d_2$, since we ought to have $\lambda_1 d_1 = \lambda_2 d_2$.

Assume by contradiction that $\gcd(p, \lambda_2)=1$.
Then $\deg(Q_{j2}')=\lambda_2-1$, and since $S_2(X)'$ is a constant we have that \[\deg([Q_{j2}(S_2(X))]')=\deg(Q_{j2}'(S_2(X)))=(\lambda_2-1)d_2p.\]
At the same time, 
\[\deg([Q_{j1}(S_1(X))]')= \deg(Q_{j1}'(S_1(X))) \leq (\lambda_1 -1) d_1 p < (\lambda_2-1)d_2p.\]
But $\deg(Q_{j1}(S_1(X)))=\deg(Q_{j2}(S_2(X)))$ and hence a contradiction arises.
\end{proof}

A direct consequence of Lemma \ref{lem:etale} is that a space obtained as a pullback of a space $L_{\lambda p^{n-1},n}$ with $(p, \lambda)=1$ and $\lambda > p-1$ will never be equivalent to the pullback of a standard space.
For example, the spaces $L_{12,2}$ and $L_{15,2}$ constructed in Section \ref{sec:class} give rise by \'etale pullbacks to certain spaces $L_{36d,2}$ and $L_{45d, 2}$ for every positive integer $d$, and the lemma ensures that these spaces are never equivalent to \'etale pullbacks of standard spaces.

\subsubsection{Standard spaces for $p=2$}
When $p=2$, the standard $L_{2^{n-1},n}$-spaces can be also obtained using the techniques of section \ref{sec:p=2}: by Proposition \ref{prop:standard} we have that any such space is generated by a basis arising from $Q_1, \dots, Q_n$ with
\[ Q_i := \mu a_i X + \mu a_i^2 = V_i^2 X + U_i^2\]
for $V_i = (\mu a_i)^{1/2}$ and $U_i = \mu^{1/2} a_i$.
We then have that $U_i = \mu^{-1/2} V_i^2$ and then the condition $(i)$ of Proposition \ref{prop:UV} is met by setting $\alpha=0$ and $\beta= \mu^{-1/2}$.
Recall that $\mu^{2^{n-1}}=\frac{1}{\Moore{n}{a}}$.
We then have that 
\[\beta^{2^{n-1}-1} \Moore{n}{V}= \frac{\Delta_n((\mu a_1)^{1/2}, \dots, (\mu a_n)^{1/2})}{(\mu^{1/2})^{2^{n-1}-1}}= (\mu^{1/2})^{2^{n-1}} \Moore{n}{a}^{1/2}=1,\]
and hence condition $(ii)$ of Proposition \ref{prop:UV} is also met. 
By using Proposition \ref{prop:FM5.1.}, we know that there exist elements $W_1, \dots, W_n$ with $\Moore{n}{W} \neq 0$ giving rise to this space via the construction of Theorem \ref{thm:p=2generaln}.

Conversely, for every $n$-tuple of elements $W_1, \dots, W_n \in k^\times$ such that $\Moore{n}{W} \neq 0$ one can set $V_i=\Delta_{n-1}(\ul{\hat{W_i}}) \in k^\times$ and $U_i= \frac{V_i^2}{\Moore{n}{W}}\in k^{\times}$, and show that the resulting 
\[Q_i = U_i^2 + XV_i^2 = V_i^2 X+\left(\frac{V_i^2}{\Moore{n}{W}}\right)^2\] give rise to a basis of a space $L_{2^{n-1},n}$ whose set of poles is the set of non-zero vectors in $\langle a_1^2, \dots, a_n^2 \rangle_{\F_p}$ with $a_i:=\frac{V_i}{\Delta_n(\ul{W})}$. By definition, this is a standard space.


Finally, we observe that we can also obtain \'etale pullbacks of standard spaces using the the techniques of section \ref{sec:p=2}.
In fact, by Remark \ref{rmk:p=2pullpack}, we can realise \'etale pullbacks of standard spaces with respect to $X \mapsto S(X)$ with $S'(X)=1$ using the the techniques of section \ref{sec:p=2}.
All \'etale pullbacks of standard spaces are equivalent to an \'etale pullback with $S'(X)=1$, so we can realise in this way at least one member from each equivalence class.

\section{Classification of spaces $L_{12,2}$ and $L_{15,2}$ over $\F_3$}\label{sec:class}
The aim of this section is to completely classify spaces $L_{12,2}$ and $L_{15,2}$ up to equivalence and Frobenius equivalence in the case where $p=3$.
We recall that the spaces $L_{3,2}, L_{6,2}$ and $L_{9,2}$ ($p=3$) are classified by Pagot in \cite{Pagot02}, so the results of this section are a natural prosecution of that work.
By exhibiting the existence of a space $L_{15,2}$, we provide the first known example of a space $L_{\lambda p,2}$ where $p-1$ does not divide $\lambda$.
This section relies on computations of Gr\"obner bases to solve polynomial systems in characteristic 3.
The supporting \emph{Macaulay2} code can be found in a public repository\footnote{available at the url \url{https://github.com/DanieleTurchetti/equidistant}}, in a form that can be easily replicated using other computer algebra systems.

\subsection{Classification of spaces $L_{12,2}$}\label{sec:L12,2}

Let $p=3$ and $\lambda=4$.
We set
\[Q_1 :=X^4-t_1X^3+t_2X^2-t_3X+t_4,\] 
\[Q_2 :=a(X^4-s_1X^3+s_2X^2-s_3X+s_4),\]
and we look for conditions such that the pair $(Q_1,Q_2)$ is a prompt for a space $L_{12,2}$.
To this aim, we consider the expressions $R_k$'s of Convention \ref{conv:Ri} as polynomials with coefficients in $\{a,s_1, \dots, s_4, t_1, \dots, t_4\}$
and look for a solution to the system of equations
\begin{equation}\label{eq:Ri4} \begin{cases}
R_1(a, s_i, t_i) \neq 0, \\
R_k(a, s_i, t_i) = 0 & \text{for} \;\; k=2, \dots, 8.\end{cases}
\end{equation}

The main result of this section is the following:

\begin{theorem}\label{thm:class12,2}
Let $k$ be an algebraically closed field containing $\F_3$ and let $a\in k$ be such that $a^2 \not \in \F_3$.
Then, the pair $(Q_{1,a},Q_{2,a})$ with 
\[Q_{1,a}:=X^4-(a^4+a^2-1)X^2+1,\]
\[Q_{2,a} :=a(X^4+(a^4-a^2-1)X^2+a^8),\] 
is a prompt for a space $L_{12,2}$ in $\Omega(k(X))$ denoted by $\Omega_a$. 
Conversely, if $\Omega \subset \Omega(k(X))$ is a space $L_{12,2}$ then there exists $a \in k$ with $a^2 \not \in \F_3$ such that $\Omega$ is equivalent to $\Omega_a$.
\end{theorem}
\begin{proof}
We prove the two statements separately.
\begin{itemize}
\item Let $(Q_{1,a},Q_{2,a})$ be as in the statement.
We can verify that 
\[\left((Q_{1,a}^3-Q_{1,a}Q_{2,a}^2)^2\right)''= -(a^3-a)^{10}(a^2+1)^5.\]
Since we assumed that $a^2 \not \in \F_3$, this is non-zero. We can then apply Proposition \ref{prop:Pagot} and Remark \ref{rmk:givesriseto} to get that the pair $(Q_{1,a},Q_{2,a})$ is a prompt for a space $L_{12,2}$. In the rest of the proof, we will denote this space by $\Omega_a$.
\item Let $\Omega \subset \Omega(k(X))$ be a space $L_{12,2}$, and let $(A,B)$ be a pair of polynomials in $k[X]$ of the form
\[Q_1 =X^4-t_1X^3+t_2X^2-t_3X+t_4,\]
\[Q_2 =a(X^4-s_1X^3+s_2X^2-s_3X+s_4),\] 
such that $(Q_1, Q_2)$ is a prompt for $\Omega$.
Then, we can apply the following successive reductions to get a situation where the coefficients $s_i$'s and $t_i$'s can be retrieved computationally :
\begin{enumerate}[]
\item[`$s_1=t_1=0$':] We know by Lemma \ref{lem:s1=t1} that $s_1=t_1$ and applying the translation $X\to X+s_1$ allows us to suppose $s_1=t_1=0$.
\item[`$s_3=t_3=0$':] Suppose by contradiction that this is not the case. Then, up to replacing $Q_1$ with $\frac{Q_2}{a}$ and $Q_2$ with $a Q_1$, we can assume that $s_3\neq 0$.
In fact, after applying a homothety, we can assume that $s_3=1$.
Then by resolving system (\ref{eq:Ri4}) computationally, one finds that it has no solution (see \cite[Program 6.1]{GithubRepo}).
\item[`$s_4=a^8$, $t_4=1$':] A computation of Gr\"obner basis under the reductions above (see \cite[Program 6.2]{GithubRepo}) returns the condition $s_4=a^8 t_4$.
Since $t_3=0$ we know that $t_4 \neq 0$ otherwise the polynomial $B$ would have multiple roots. We can then pick $\beta \in k$ such that $\beta^4=t_4^{-1}$ and apply the transformation $X\to \beta X$ to get that $t_4=1$.
\end{enumerate}
The reductions above leave us with polynomials of the form
\[Q_1:=X^4+t_2X^2+1,\]
\[Q_2 :=a(X^4+s_2X^2+a^8).\] 
Let us simplify the notation by setting $s:=s_2$ and $t:=t_2$, and compute these coefficients, again using Gr\"obner bases (see \cite[Program 6.3]{GithubRepo}).
By inspecting the first element of the basis, we conclude that we have at most six possible expressions for $s$:
\[s= \begin{cases} \pm(a^4+a^2-a+1) \\
 \pm(a^4-a^2-1) \\ \pm(a^4+a^2+a+1)\end{cases}\]
Let us show how the situation can be further simplified:
first, we note that the transformation $X\to i X$ with $i^2=-1$ sends $(Q_1, Q_2)$ to the pair of polynomials 
\[\left (X^4-tX^2+1, a(X^4-sX^2+a^8) \right).\]
Further, we note that $(Q_1,Q_2)$ is a prompt for a space $L_{12,2}$ if, and only if, the pair $(-Q_1,Q_2)$ is a prompt for the same space.
As a result of these two observations, if we can find a solution to the system (\ref{eq:Ri4}) for a pair $(a,s)$ we also have a solution for pairs $(-a,s)$,$(a,-s)$, and $(-a,-s)$, which reduces our search to the following two situations:
\begin{enumerate}
\item[Case 1:] $s=a^4-a^2-1$.
In this case, Program \cite[Program 6.4]{GithubRepo} returns $t=-(a^4+a^2-1)$ (under the condition that $s \neq 0$), and then we find that $Q_1=Q_{1,a}$ and $Q_2=Q_{2,a}$, as required. If $s=0$, then the program returns $t^2=a^2+1=a^4$, which gives the two values $t= \pm a^2=\pm(a^4+a^2-1)$.
Note that these two values correspond to equivalent spaces, under the transformation $X \mapsto i X$.
We conclude that, in Case 1, the pair $(Q_1,Q_2)$ is a prompt for a space equivalent to $\Omega_{a}$.
\item[Case 2:] $s=a^4+a^2+a+1$. In this case, Program \cite[Program 6.5]{GithubRepo} returns $t=-(a^4-a^3+a^2+1)$ (under the condition that $s \neq 0$).

We now compare the pair $(Q_1,Q_2)$ obtained with the pair $(Q_{1,a-1},Q_{2,a-1})$ which is a prompt for the space $\Omega_{a-1}$, namely
\[Q_{1,a-1}=X^4-(a^4-a^3+a^2+1)X^2+1,\]
\[Q_{2,a-1} =(a-1)\left(X^4+(a^4-a^3-a^2+a-1)X^2+(a-1)^8\right).\] 

One verifies that $Q_1^3Q_2-Q_1Q_2^3=Q_{1,a-1}^3Q_{2,a-1}-Q_{1,a-1}Q_{2,a-1}^3$. Hence the pair $(Q_1,Q_2)$ is a prompt for a space $L_{12,2}$ with the same set of poles as $\Omega_{a-1}$.
By Corollary \ref{coro:poles}, we have then that $(Q_1,Q_2)$ is a prompt for $\Omega_{a-1}$.
If $s=0$, the program returns $t^2 = -a^3-a^2-a-1$, giving the two values $t=\pm (a^3+a) =\pm (a^4-a^3+a^2+1)$, that correspond to equivalent spaces under the transformation $X \to i X$.
Hence in Case 2 the pair $(Q_1,Q_2)$ is a prompt for a space equivalent to $\Omega_{a-1}$.
\end{enumerate}

\end{itemize}
\end{proof}

We can be more explicit about the spaces $\Omega_a$ classified in the theorem, and compute their residues and their writing in logarithmic form.
This is the content of the following result.
\begin{corollary}
Let $k$ be an algebraically closed field containing $\F_3$.
Let $a\in k$ be such that $(a^3-a)(a^2+1)\neq 0$ and
fix $i,j \in k$ such that $i^2=-1$ and $j^2=a^2+1$. 
Then the space $\Omega_a$ is generated by $\frac{df_1}{f_1}$ and $\frac{df_2}{f_2}$ with
\[f_1=\prod_{\substack{\epsilon_1=\pm 1\\ \epsilon_2=\pm 1}} \left(X-(\epsilon_1 (a-1)j+\epsilon_2 a i)\right)^{-\epsilon_1}\prod_{\substack{\epsilon_1=\pm 1\\ \epsilon_2=\pm 1}}\left(X-(\epsilon_1 (a+1)j+\epsilon_2 a i)\right)^{\epsilon_1}\prod_{\substack{\epsilon_1=\pm 1\\ \epsilon_2=\pm 1}} \left(X-(\epsilon_1 j + \epsilon_2(a^2-1))\right)^{\epsilon_1}\]
and
\[f_2=\prod_{\substack{\epsilon_1=\pm 1\\ \epsilon_2=\pm 1}} \left(X-(\epsilon_1 (a-1)j+\epsilon_2 a i)\right)^{\epsilon_1}\prod_{\substack{\epsilon_1=\pm 1\\ \epsilon_2=\pm 1}}\left(X-(\epsilon_1 (a+1)j+\epsilon_2 a i)\right)^{\epsilon_1}\prod_{\substack{\epsilon_1=\pm 1\\ \epsilon_2=\pm 1}} \left(X-(\epsilon_1 a j+\epsilon_2(a^2-1))\right)^{\epsilon_1}.\]
 
In particular, the set of poles and residues of these generators are:

\begin{center}
\begin{tabular}{|c|c|c|}
\hline
$x$ & $\res_1(x)$ & $\res_2(x)$\\
\hline
\hline
$j+a^2-1$ & $1$ & $0$\\
$j-a^2+1$ & $1$ & $0$\\
$-j+a^2-1$ & $-1$ & $0$\\
$-j-a^2+1$ & $-1$ & $0$\\
$(a-1)j+ai $ & $-1$ & $1$\\
$(a-1)j-ai $& $-1$ & $1$\\
$-(a-1)j+ai $ & $1$ & $-1$\\
 $-(a-1)j-ai $ & $1$ & $-1$\\
 \hline
\end{tabular}
\begin{tabular}{|c|c|c|}
\hline
$x$ & $\res_1(x)$ & $\res_2(x)$\\
\hline
\hline
$(a+1)j+ai $ & $1$ & $1$\\
$(a+1)j-ai $& $1$ & $1$\\
$-(a+1)j+ai $ & $-1$ & $-1$\\
 $-(a+1)j-ai $ & $-1$ & $-1$\\
$aj+(a^2-1) $ & $0$ & $1$\\
$aj-(a^2-1)$& $0$ & $1$\\
$-aj+(a^2-1)$ & $0$ & $-1$\\
 $-aj-(a^2-1)$ & $0$ & $-1$\\
 \hline
\end{tabular}
\end{center}

\end{corollary}
\begin{proof}
We start by computing the zeroes of the polynomials $Q_{1,a}$, $Q_{1,a}+Q_{2,a}$, $Q_{1,a}-Q_{2,a}$, $Q_{2,a}$, which can be easily done by applying the formulas for solving biquadratic equations.
We find that
\[Z(Q_{1,a})=\{ aj+(a^2-1), aj-(a^2-1), -aj+(a^2-1), -aj-(a^2-1)\},\]
\[Z(Q_{1,a}+Q_{2,a})=\{(a-1)j+ai, (a-1)j-ai, -(a-1)j+ai , -(a-1)j-ai  \} \]
\[Z(Q_{1,a}-Q_{2,a})=\{(a+1)j+ai , (a+1)j-ai, -(a-1)j+ai, -(a-1)j-ai \} \mbox{, and}\]
\[Z(Q_{2,a})=\{ j+a^2-1, j-a^2+1, -j+a^2-1, -j-a^2+1\},\]
as expected.
By Theorem \ref{thm:class12,2} and Remark \ref{rmk:givesriseto} we know that $\Omega_a$ is generated by the differential forms $\omega_1:=\frac{Q_{2,a} dX}{c^3(Q_{1,a}^3Q_{2,a}-Q_{1,a}Q_{2,a}^3)}$ and $\omega_2:=\frac{-Q_{1,a} dX}{c^3(Q_{1,a}^3Q_{2,a}-Q_{1,a}Q_{2,a}^3)}$ with $c^6(a^2+1)^5(a^3-a)^{10}=1$.
To compute the residues of these forms at their poles, we remark that $c^3=\pm\frac{1}{(a^2+1)^{5/2}(a^3-a)^5}=\pm \frac{1}{j^5(a^3-a)^5}$.
For each $x\in Z(Q_{1,a}) \cup Z(Q_{1,a}+Q_{2,a}) \cup Z(Q_{1,a}-Q_{2,a}) \cup Z(Q_{2,a})$ we can compute the quantities $(Q_{1,a}^2Q_{2,a}-Q_{2,a}^3)'(x)$ and $(Q_{1,a}^3-Q_{1,a}Q_{2,a}^2)'(x)$ and check when they are equal to $\pm j^5(a^3-a)^5$.
Up to possibly replacing $j$ with $-j$, and using the formulas $\res_1(x)=\frac{1}{c^3(Q_{1,a}^3-Q_{1,a}Q_{2,a}^2)'(x)}$ and $\res_2(x)=\frac{1}{c^3(Q_{1,a}^2Q_{2,a}-Q_{2,a}^3)'(x)}$ we get the table of residues above, or equivalently that $\omega_1=\frac{df_1}{f_1}$ and $\omega_2=\frac{df_2}{f_2}$ with $f_1, f_2$ as in the statement.
\end{proof}

\begin{remark}
In the light of the results of Section \ref{sec:standard}, it is natural to wonder if the description of subspaces of standard spaces given by Proposition \ref{prop:standardsubsp} has an analogue in the case of spaces that are not standard. Namely, given a space $L_{\lambda p^{n-1},n}$ that is not standard, and assuming that $p | \lambda$, one can ask whether its $t$-dimensional subspaces for $t<n$ can be obtained as \'etale pullbacks of some space $L_{\mu p^{t-1},t}$.

The classification of spaces $L_{12,2}$ achieved in this section shows that the answer to this question is negative, at least at the level of generality stated above.
More precisely, for every value of $a$ no one dimensional subspace of $\Omega_a$ can be obtained as a non-trivial \'etale pullback of some other space.
First, note that we only need to check this for degree 6 pullbacks of spaces $L_{2,1}$, as there are no spaces $L_{4,1}$. Then, we remark that any space $L_{2,1}$ is equivalent to a standard space, whose set of poles is of the form $\{x, -x\}$.

The \'etale pullback $\phi^\star(\omega)$ is then up to a constant of the form $\frac{dX}{S(X)^2-x^2}$ for $S(X)$ of degree 6 such that $S'(X) \in k^\times$.
In particular, $S(X)^2-x^2$ has always a non zero term of degree 7.
If we fix a differential form $\omega \in \Omega_a$, and we consider the polynomial whose zeroes are the poles of $\omega$, we see that it has only terms of even degree.
As a result, it is not possible to obtain $\omega$ via an \'etale pullback of a differential form in a space $L_{2,1}$.
\end{remark}

\subsection{Classification of spaces $L_{15,2}$}

Let $p=3$ and $\lambda=5$.
In the first part of the section, we exhibit explicitly two vector spaces $L_{15,2}$, one whose poles are all in $\F_{27}$ and another one whose poles are all in $\F_{81}$.

\begin{example}\label{ex:F27}
Let $\F_{27}$ be the finite field with 27 elements, and let us write $\F_{27}=\F_3[\mu]$, with $\mu^3-\mu+1$=0.
We have that $\mu^{13}=-1$, and then $\mu$ is a generator of the cyclic multiplicative group $\F_{27}^\times$.
Following the notation used in the proof of Proposition \ref{prop:Pagot}, we define a subset of $\F_{27}$ indexed by the elements of $\PP^1(\F_3)=\{0,1,2,\infty\}$:
\[\begin{cases}
X_{0} = & \{\mu^2-\mu, \mu+1, -\mu^2-\mu-1, \mu, 0 \} =\{\mu^4,\mu^9, \mu^{19}, \mu, 0\}\\
X_{1} = & \{ -\mu^2+\mu-1, -\mu^2, -\mu^2-\mu+1, -\mu^2-1, \mu^2+1\} =\{\mu^5,\mu^{15}, \mu^{24},\mu^8,\mu^{21}\}\\
X_{2} = & \{-\mu^2+1, -1, \mu^2-\mu+1, \mu^2-1, -\mu^2+\mu \} =\{\mu^{25},\mu^{13}, \mu^{18}, \mu^{12}, \mu^{17}\}\\
X_{\infty} = & \{ \mu^2+\mu, \mu-1, \mu^2, \mu^2-\mu-1, -\mu-1\}=\{\mu^{10}, \mu^3, \mu^2, \mu^7,\mu^{22}\}. \\
\end{cases}\]
To these sets, we associate the corresponding polynomials of $\F_{27}[X]$:
\[\begin{cases}
P_{0}(X) =   \prod_{x\in X_0} (X-x)  & = X^5-(\mu^2+\mu+1)X^3-X^2+(\mu^2-\mu-1)X \\
P_{1}(X) =  \prod_{x\in X_1} (X-x) &  = X^5-(\mu^2-1)X^3-(\mu^2-\mu+1 )X^2-(\mu+1)X-(\mu^2+1)\\
P_{2}(X) =  \prod_{x\in X_2} (X-x) & =  X^5+(\mu^2-\mu-1)X^3+(\mu+1)X^2+(\mu^2-1)X-(\mu^2-\mu-1)\\
P_{\infty}(X) =   \prod_{x\in X_\infty} (X-x)  & = X^5-(\mu^2+\mu)X^3-(\mu+1)X^2-(\mu^2+\mu-1)X-(\mu^2-\mu+1).
\end{cases}\]

Now, let
\[Q_1(X) :=-\mu P_\infty(X) \;\; \text{and} \;\; Q_2(X) := (\mu^2-\mu-1) P_0(X).\]

One verifies that the second derivative of the polynomial $(Q_1^3-Q_1Q_2^2)^2$ is equal to $-1$. Moreover, since $(\mu^2-\mu-1)\pm(-\mu) \notin \F_3$, the degree of $iQ_1+jQ_2$ is equal to 5 for every $[i:j]\in \PP^1(\F_3)$.
Then, we can set $\displaystyle \omega_1=\frac{Q_2}{Q_1^3Q_2-Q_1Q_2^3}dX$ and $\displaystyle \omega_2=\frac{-Q_1}{Q_1^3Q_2-Q_1Q_2^3}dX$ and apply Proposition \ref{prop:Pagot} to show that $\Omega_1 := \langle \omega_1,\omega_2 \rangle$ is a $\F_3$-vector space $L_{15,2}$.

To write the generators of this spaces in logarithmic form, we need to compute their residues at the poles.
In order to do this, we denote by $Z_1$ the set of zeroes of $Q_1^3-Q_1Q_2^2$, by $Z_2$ the set of zeroes of $Q_1^2Q_2-Q_2^3$, and we introduce the functions $\res_1: Z_1 \to \F_3^\times$ and $\res_2(x):Z_2 \to \F_3^\times$ defined by $\res_i(x):=\res_x(\omega_i)$.
Then, we compute them thanks to the formulas $\res_1(x)=\frac{1}{(Q_1^3-Q_1Q_2^2)'(x)}$ and $\res_2(x)=\frac{1}{(Q_1^2Q_2-Q_2^3)'(x)}$.
We find the following table of values:\\
\begin{center}
\begin{tabular}{|c|c|c|}
\hline
$x$ & $\res_1(x)$ & $\res_2(x)$\\
\hline
\hline
$\mu^4$& 1 & 0\\
$\mu^9$ & 1 & 0\\
$\mu^{19}$ & 1 & 0\\
$\mu$ & $-1$ & 0\\
$0$ & $-1$ & 0\\
$\mu^5$ & $1$ & -1\\
$\mu^{15}$ & $-1$ & $1$\\
 $\mu^{24}$ & $-1$ & $1$\\
 $\mu^8$ &  $-1$&$1$ \\
 $\mu^{21}$ & $-1$ & $1$\\
 \hline
\end{tabular}
\begin{tabular}{|c|c|c|}
\hline
$x$ & $\res_1(x)$ & $\res_2(x)$\\
\hline
\hline
$\mu^{25}$ &$-1$ & $-1$ \\
$\mu^{13}$ & $1$&$1$ \\
$\mu^{18}$ & $-1$ & $-1$\\
$\mu^{12}$ & $1$& $1$\\
$\mu^{17}$ & $-1$&$-1$ \\
$\mu^{10}$ & 0 &$1$ \\
 $\mu^3$ & 0& $1$\\
 $\mu^2$ & 0 &$-1$ \\
 $\mu^7$ & 0 & $-1$ \\
 $\mu^{22}$ &0 & 1\\
 \hline
\end{tabular}
\end{center}

Then, $\omega_1=\frac{dF_1}{F_1}$ and $\omega_2=\frac{dF_2}{F_2}$ with
\[F_1(X)= \prod_{x\in Z_1} (X-x)^{res_{1}(x)}  \mbox{ and } F_2(X)=\prod_{x\in Z_2} (X-x)^{res_{2}(x)}.\]

\end{example}

\begin{example}\label{ex:F81}
Let $\F_{81}$ be the finite field with 81 elements, and let us write $\F_{81}=\F_3[\mu]$, with $\mu^4+\mu^3-\mu^2-\mu-1=0$.\footnote{This is equivalent to write $\F_{81}=\F_3[\alpha]$ for a choice of $\alpha$ satisfying $\alpha^4+\alpha^2-1=0$ and set $\mu=\alpha^3+\alpha-1$.}
We have that $\mu^{40}=-1$, and then $\mu$ is a generator of the cyclic multiplicative group $\F_{81}^\times$.
In analogy with Example \ref{ex:F27}, we define for every element of $\PP^1(\F_3)$ a subset of $\F_{81}$ as follows:
\[\begin{cases}
X_{0} = & \{\mu^7, \mu^{30}, \mu^{51}, \mu^{59}, \mu^{63} \} \\
X_{1} = & \{ \mu^{26}, \mu^{50}, \mu^{52}, \mu^{68}, \mu^{74}\} \\
X_{2} = & \{ \mu^{34}, \mu^{60}, \mu^{66}, \mu^{70}, 0 \} \\
X_{\infty} = & \{ \mu^{10}, \mu^{11}, \mu^{19}, \mu^{20}, \mu^{40}\} \\
\end{cases}.\]
Then, we associate to these sets the corresponding monic polynomials.
In this case, it turns out that they have coefficients in $\F_9$. 
More specifically, we set $a=-\mu^{20}\in \F_9$ and one can verify that we have:
\[\begin{cases}
P_{0} =   \prod_{x\in X_0} (X-x)  & = X^5 - X^3 -X^2 +a X -(a+1) \\
P_{1} =  \prod_{x\in X_1} (X-x) &  = X^5 - aX^3 + (a-1) X - (a-1) \\
P_{2} =  \prod_{x\in X_2} (X-x) & =  X^5 - (a-1) X^2 - (a-1) X\\
P_{\infty} =   \prod_{x\in X_\infty} (X-x)  & =X^5-(a+1)X^3+(a+1)X^2+X+a
\end{cases}.\]
Then, we set 
\[Q_1(X) := -(a+1)P_\infty(X)  \;\; \text{and} \;\; Q_2(X) := a P_0(X),\]
and we can verify that the second derivative of the polynomial $(Q_1^3-Q_1Q_2^2)^2$ is equal to $-1$. 
Moreover, the degree of $iQ_1+jQ_2$ is equal to 5 for every $[i:j]\in \PP^1(\F_3)$.
Hence, if we set $\displaystyle \omega_1=\frac{Q_2}{Q_1^3Q_2-Q_1Q_2^3} dX$ and $\displaystyle \omega_2=\frac{-Q_1}{Q_1^3Q_2-Q_1Q_2^3} dX$ and apply Proposition \ref{prop:Pagot} this shows that $\Omega_2 := \langle \omega_1,\omega_2 \rangle$ is a space $L_{15,2}$.

As in Example \ref{ex:F27}, we can easily compute the residues of $\omega_1$ and $\omega_2$ at their poles, and get the following table

\begin{center}
\begin{tabular}{|c|c|c|}
\hline
$x$ & $\res_1(x)$ & $\res_2(x)$\\
\hline
\hline
$\mu^7$ & $-1$ & $0$\\
$\mu^{30}$ & $-1$ & $0$\\
$\mu^{51}$ & $-1$ & $0$\\
$\mu^{59}$ & $-1$ & $0$\\
$ \mu^{63}$ & $-1$ & $0$\\
$\mu^{26}$& $-1$ & $1$\\
$\mu^{50}$ & $1$ & $-1$\\
 $\mu^{52}$ & $1$ & $-1$\\
 $\mu^{68}$ &  $1$ &$-1$ \\
 $\mu^{74}$ & $-1$ & $1$\\
 \hline
\end{tabular}
\begin{tabular}{|c|c|c|}
\hline
$x$ & $\res_1(x)$ & $\res_2(x)$\\
\hline
\hline
$\mu^{34}$ &$-1$ & $-1$ \\
$ \mu^{60}$ & $1$&$1$ \\
$\mu^{66}$ & $-1$ & $-1$\\
$\mu^{70}$ & $1$& $1$\\
$0$ & $1$&$1$ \\
$\mu^{10}$ & $0$ & 1 \\
 $\mu^{11}$ & $0$& 1\\
 $\mu^{19}$ & $0$ &1 \\
 $\mu^{20}$ &$0$ & 1\\
 $\mu^{40}$ &$0$ & $-1$\\
 \hline
\end{tabular}
\end{center}
and as in the previous example we have $\omega_1=\frac{dF_1}{F_1}$ and $\omega_2=\frac{dF_2}{F_2}$ with
\[F_1(X)= \prod_{x\in Z_1} (X-x)^{res_{1}(x)}  \mbox{ and } F_2(X)=\prod_{x\in Z_2} (X-x)^{res_{2}(x)}.\]
\end{example}

The main result of this section states that all the spaces $L_{15,2}$ can be obtained from the examples above.
\begin{theorem}\label{thm:class15,2}
Let $k$ be an algebraically closed field containing $\F_3$ and let $\Phi:\Omega(k(X))\to \Omega(k(X))$ be the relative Frobenius operator.
Every $\Omega\subset \Omega(k(X))$ vector space $L_{15,2}$ is equivalent to one of the following spaces, each representing a distinct equivalence class: $\Omega_1$ (defined in Example \ref{ex:F27}), $\Phi(\Omega_1)$, $\Phi^2(\Omega_1)$, $\Omega_2$ (defined in Example \ref{ex:F81}), or $\Phi(\Omega_2)$.
\end{theorem}

To prove the theorem, we set
\[Q_1:=X^5-t_1X^4+t_2X^3-t_3X^2+t_4X-t_5\]
and
\[Q_2 :=a(X^5-s_1X^4+s_2X^3-s_3X^2+s_4X-s_5).\] 

Then, we consider the expressions $R_k$'s of Convention \ref{conv:Ri} as polynomials with coefficients in $\{a,s_1, \dots, s_5, t_1, \dots, t_5\}$
and we aim to find a solution to the system of equations
\begin{equation}\label{eq:Ri} \begin{cases}
R_1(a, s_i, t_i) \neq 0, \\
R_k(a, s_i, t_i) = 0 & \text{for} \;\; k=2, \dots, 10.\end{cases}
\end{equation}
The existence of a space $L_{15,2}$ gives rise to a solution to the system (\ref{eq:Ri}) (see Convention \ref{conv:Ri}).
Conversely a solution to (\ref{eq:Ri}) produces a space $L_{15,2}$. 
In fact, such a solution would imply that the second derivative of $(Q_1^3-Q_1Q_2^2)^2$ is equal to the nonzero constant $-R_1(a, s_i, t_i)$.
It follows that $(Q_1, Q_2)$ is a prompt for a space $L_{15,2}$.

Unfortunately, solving (\ref{eq:Ri}) is not an easy task even when assisted by a computer (a brute-force calculation of Gr\"obner basis turns out to be hopeless without further assumptions), so we need to simplify the equations before we go further with our strategy.
To this aim, we apply Lemma \ref{lem:s1=t1} to get that $s_1=t_1$ and we reparametrize $\PP^1_k$ in such a way to have $s_1=t_1=0$.
Then, we obtain the following result.

\begin{lemma}\label{lem:s3t3}
Let $(Q_1,Q_2)$ be a pair of polynomials as above, and suppose that they yield a solution of the system (\ref{eq:Ri}). Then, either $s_3 \neq 0$ or $t_3 \neq 0$.
\end{lemma}
\begin{proof}
We can show this using a Gr\"obner basis computation: Program \cite[Program 6.6]{GithubRepo} computes the ideal generated by all the relations in (\ref{eq:Ri}) with $s_1=t_1=0$ and $s_3=t_3=0$, and verifies that the system (\ref{eq:Ri}) has no solution under these assumptions. As a result, we either have $s_3 \neq 0$ or $t_3 \neq 0$.
\end{proof}
Thanks to Lemma \ref{lem:s3t3}, we can suppose that $s_3\neq 0$, and use our last degree of freedom to reparametrize $\PP^1_k$ in such a way that $s_3=1$. This is the content of our next Lemma.
\begin{lemma}\label{lem:s3=1}
Let $(Q_1,Q_2)$ be a pair of polynomials as above and suppose that they yield a solution of the system (\ref{eq:Ri}), hence being a prompt for a space $L_{15,2}$ denoted by $\Omega$.
Then, there exists a pair $(Q_1^\sharp,Q_2^\sharp)$ whose coefficients satisfy $s_1=t_1=0$, $s_3=1$ and (\ref{eq:Ri}).
In particular, $(Q_1^\sharp,Q_2^\sharp)$ is a prompt for a space equivalent to $\Omega$.
\end{lemma}
\begin{proof} 
We know by Lemma \ref{lem:s1=t1} that $s_1=t_1$ and applying the translation $X\to X+s_1$ allows us to suppose $s_1=t_1=0$.
To get to $s_3=1$, we apply Lemma \ref{lem:s3t3} to get to the situation where $s_3 \neq 0$, which we can do up to possibly swapping $Q_1$ and $Q_2$.
Then, we consider $\alpha \in k$ such that $\alpha^3=s_3^{-1}$ and we apply the transformation $X\to \alpha X$.
Under this transformation we have
\[P_0^\sharp(X) = \prod_{x\in X_{[1:0]}} (X-\alpha x)=X^5+\alpha^2 s_2X^3-\alpha^3 s_3X^2+\alpha^4 s_4X-\alpha^5 s_5.\] 
and
\[P_\infty^\sharp(X) = \prod_{x\in X_{[0:1]}} (X-\alpha x) =X^5+\alpha^2 t_2X^3-\alpha^3 t_3X^2+\alpha^4 t_4X-\alpha^5 t_5,\]
so that the coefficient of $X^2$ in $P_0'(X)$ is -1 as needed.
Since $(Q_1,Q_2)$ is a prompt for a space $L_{15,2}$ and we applied a homothety to the zeroes of $Q_1^3Q_2-Q_1Q_2^3$, then by setting $Q_1^\sharp=P_\infty^\sharp$ and $Q_2^\sharp=a^\sharp P_0^\sharp$ for a suitable nonzero constant $a^\sharp$ we have that $(Q_1^\sharp,Q_2^\sharp)$ is a prompt for a space $L_{15,2}$ that is equivalent to $\Omega$.
In particular, the coefficients of $Q_1^\sharp$ and $Q_2^\sharp$ satisfy the system (\ref{eq:Ri}).
\end{proof}

Lemma \ref{lem:s3=1} allows us then to consider without loss of generality the situation where $s_1=t_1=0$ and $s_3=1$.
Recall that we denoted by $a$ the leading coefficient of the polynomial $Q_1$. 
The following lemma shows that one can obtain the same space by applying a suitable substitution only to the parameter $a$.

\begin{lemma}\label{lem:transa}
Let $(Q_1,Q_2)$ be a pair of polynomials with
\[Q_1:=X^5+t_2X^3-t_3X^2+t_4X-t_5\]
and
\[Q_2 :=a(X^5+s_2X^3-X^2+s_4X-s_5)\] 
which is a prompt for a space $L_{15,2}$, denoted by $\Omega$.
For every $\alpha \in  \{-1, \frac{1}{a+1}, -\frac{1}{a+1}, \frac{1}{a-1}, -\frac{1}{a-1}\}$ there exists a pair $(Q_1^\sharp,Q_2^\sharp)$ which is a prompt for the same space $\Omega$ such that $Q_1^\sharp$ is monic and $Q_2^\sharp=\alpha Q_2$.
\end{lemma}

\begin{proof}
By Definition \ref{defn:givesriseto}, there exists a constant $c\in k^\times$ such that the space $\Omega$ is generated by $\omega_1$ and $\omega_2$ with
$\omega_1:= \frac{dX}{c^3(Q_1Q_2^2-Q_1^3)}$ and $\omega_2:= \frac{dX}{c^3(Q_1^{2}Q_2-Q_2^3)}$.
We first consider the following situations:
\begin{enumerate}
\item If we consider the basis $(\omega_1, -\omega_2)$ of $\Omega$ and apply Proposition \ref{prop:QiMi}, we find that the pair $(Q_1,-Q_2)$ is a prompt for $\Omega$.
\item If we consider the basis $(\omega_1, \omega_1+\omega_2)$ of $\Omega$ and apply Proposition \ref{prop:QiMi}, we find that the pair $(Q_1 - Q_2,Q_2)$ is a prompt for $\Omega$.
Hence, the pair $(\frac{Q_1 - Q_2}{1-a},\frac{Q_2}{1-a})$ is also a prompt for $\Omega$ and its first term is monic.
\end{enumerate}
The coefficients $\frac{1}{a-1}, \frac{1}{a+1},$ and $-\frac{1}{a+1}$ can be obtained similarly by considering all the other possible bases of $\Omega$.
\end{proof}

\subsubsection*{Proof of Theorem \ref{thm:class15,2}}
Using the simplifications above, we are now ready to prove the main theorem.
More precisely, we use Lemma \ref{lem:s3=1} to assume that $s_1=t_1=0$ and $s_3=1$.
Then, the program \cite[Program 6.7]{GithubRepo} tells us that any solution to the system (\ref{eq:Ri}) needs to satisfy the following condition on the parameter $a$:
\[\scriptstyle{(a^3-a^2-a-1)(a^3-a+1)(a^2+1)(a^3-a^2+1)(a^3+a^2-a+1)(a^3-a^2+a+1)(a^3-a-1)(a^2+a-1)(a^3+a^2-1)(a^3+a^2+a-1)(a^2-a-1)=0.}\]
By applying Lemma \ref{lem:transa} we can consider only an irreducible polynomial for each orbit under the group action on $(Q_1,Q_2)$ generated by $a\mapsto -a$ and $a\mapsto \frac{a}{a+1}$.
Namely, we are left with studying the following cases:
\begin{enumerate}
\item[Case 1:] $a^2+1=0$
\item[Case 2:] $a^3-a^2+1=0$
\item[Case 3:] $a^3-a+1=0$.
\end{enumerate}

In all these cases, the successive elements of a Gr\"obner basis for the system (\ref{eq:Ri}) can be computed, and show us all the possibilities for the coefficients of $Q_1$ and $Q_2$.
Let $\Delta_2(Q_1,Q_2)=Q_1Q_2^3-Q_1^3Q_2$ be the Moore determinant of $(Q_1,Q_2)$.
We remark that, if $(Q_1,Q_2)$ and $(S_1,S_2)$ are prompts for spaces $L_{15,2}$ and $\Delta_2(Q_1,Q_2)=\Delta_2(S_1,S_2)$, then these pairs are prompts for the same space by virtue of Corollary \ref{coro:poles}.
With this in mind, we denote by $\Delta^{[1]}$ the Moore determinant of the pair $(Q_1, Q_2)$ appearing in Example \ref{ex:F27}, by $\Delta^{[2]}$ the Moore determinant of the pair $(Q_1, Q_2)$ appearing in Example \ref{ex:F81}, and by $\Delta^\bullet$ with an appropriate superscript the Moore determinants appearing in the cases below (e.g. $\Delta^{1.A}$ is the Moore determinant of the polynomials given by the program in case $1.A$).
We show that we can get a complete classification up to equivalence by studying the cases outlined above and comparing the respective Moore determinants.

In \textbf{Case 1}, we let $a$ be a square root of $-1$. 
For simplicity and consistency, we assume that $a=-\mu^{20}$ where $\mu$ is the generator of $\F_{81}^\times$ appearing in Example \ref{ex:F81}.
Then, the program \cite[Program 6.8]{GithubRepo} returns the following subcases:\\[.2cm]
\begin{tabular}{|l||c|c|c|c||c|c|c|c|}
      \hline
      Case $1.A$ &$s_2$ & $s_3$ & $s_4$ & $s_5$ &$t_2$&$t_3$&$t_4$&$t_5$\\ \hline
      & 0 &1 & $a-1$& 0 &$-a+1$&$a$&$-1$&$a+1$ \\ \hline 
    \end{tabular}\\[.2cm]
    \begin{tabular}{|l||c|c|c|c||c|c|c|c|}
      \hline
          Case $1.A'$ &  $s_2$ & $s_3$ & $s_4$ & $s_5$ &$t_2$&$t_3$&$t_4$&$t_5$\\ \hline
    &  0 &1 & $-a-1$& 0 &$a+1$&$-a$&$-1$&$-a+1$\\  \hline 
    \end{tabular}\\[.2cm]
In Case $1.A$ we get a space that is equivalent to $\Omega_2$.
In fact, one can verify that, after applying the homothety $X\mapsto -(a+1)X$ the Moore determinant $\Delta^{1.A}$ is equal to $\Delta^{[2]}$.
    
    Since Case $1.A'$ is obtained from Case $1.A$ by applying the transformation $a \mapsto -a$, and we have that $a^3=-a$, then it results that the polynomials appearing in Case $1.A'$ are prompts for the space $\Phi(\Omega_2)$.\\[.2cm]
    \begin{tabular}{|l||c|c|c|c||c|c|c|c|}
      \hline
          Case $1.B$ &  $s_2$ & $s_3$ & $s_4$ & $s_5$ &$t_2$&$t_3$&$t_4$&$t_5$\\ \hline
       & $-1$ &1 & $a$& $1+a$ &$-a$&0&$a-1$&$a-1$ \\ \hline 
    \end{tabular}\\[.2cm]
\begin{tabular}{|l||c|c|c|c||c|c|c|c|}
      \hline
      Case $1.B'$ &$s_2$ & $s_3$ & $s_4$ & $s_5$ &$t_2$&$t_3$&$t_4$&$t_5$\\ \hline
      & $-1$ &1 & $-a$& $1-a$ &$a$&0&$-(a+1)$&$-(a+1)$ \\ \hline 
    \end{tabular}\\[.2cm]
    In Case $1.B$, one gets the space $\Omega_2$ without the need to reparametrize. 
    In fact, the zeroes of $A$ (resp. of $B$) in this case are the same as those of $P_0$ (resp. $P_1$) in the example. 
    If we apply the transformation $a \mapsto -a$, we land in Case $1.B'$ which corresponds to the space $\Phi(\Omega_2)$.\\[.2cm]
    \begin{tabular}{|l||c|c|c|c||c|c|c|c|}
      \hline
          Case $1.C$ &  $s_2$ & $s_3$ & $s_4$ & $s_5$ &$t_2$&$t_3$&$t_4$&$t_5$\\ \hline
       & $a-1$ &1 & $-1$& $a-1$ &0 &$-a$&$a-1$&0 \\ \hline 
    \end{tabular}\\[.2cm]
\begin{tabular}{|l||c|c|c|c||c|c|c|c|}
      \hline
      Case $1.C'$ &$s_2$ & $s_3$ & $s_4$ & $s_5$ &$t_2$&$t_3$&$t_4$&$t_5$\\ \hline
      & $-(a+1)$ &1 & $-1$& $-(a+1)$&0&$a$&$-(a+1)$&0 \\ \hline 
    \end{tabular}\\[.2cm]
In Case $1.C$, one gets a space that is equivalent to $\Omega_2$, as one can see by applying the transformation $X \mapsto (a-1)X$.
If we apply the transformation $a \mapsto -a$, we land in Case $1.C'$ which then corresponds to the space $\Phi(\Omega_2)$, as above.

In \textbf{Case 2}, let $a$ be a root of $X^3-X^2+1$.
For simplicity and consistency, we assume that $a=\mu^{-1}$ where $\mu$ is the generator of $\F_{27}^\times$ appearing in Example \ref{ex:F27}.
Then we have that $\Phi(a)=\frac{1}{\mu+1}=-a^2-a-1$ and $\Phi^2(a)=\frac{1}{\mu-1}=a^2-1$.
Program \cite[Program 6.9]{GithubRepo} returns the following subcases:\\[.2cm]
\begin{tabular}{|l||c|c|c|c||c|c|c|c|}
      \hline
      Case $2.A$ & $s_2$ & $s_3$ & $s_4$ & $s_5$ &$t_2$&$t_3$&$t_4$&$t_5$\\ \hline
      & $-1$ &1 & $a^2$& $-(a^2+a+1)$ &$-(a+1)$ &$a^2+a+1$&$a^2-a$&$-(a^2-a+1)$ \\ \hline 
    \end{tabular}\\[.2cm]
    \begin{tabular}{|l||c|c|c|c||c|c|c|c|}
      \hline
      Case $2.B$ & $s_2$ & $s_3$ & $s_4$ & $s_5$ &$t_2$&$t_3$&$t_4$&$t_5$\\ \hline
      & $-1$ &1 & $-(a^2+a)$& $a$ &$-(a+1)$&$a^2$&$-(a^2-a)$&$0$\\  \hline 
    \end{tabular}\\[.2cm]
            \begin{tabular}{|l||c|c|c|c||c|c|c|c|}
      \hline
      Case $2.C$& $s_2$ & $s_3$ & $s_4$ & $s_5$ &$t_2$&$t_3$&$t_4$&$t_5$\\ \hline
      & $-1$ &1 & $a+1$& $a^2-1$ &$-(a+1)$&$a^2-1$&$1$&$a^2-a-1$\\  \hline 
    \end{tabular}\\[.2cm]

In Case $2.A$ we get a vector space equivalent to $\Omega_1$.
In fact, one can verify that, after applying the homothety $X\mapsto (\mu^2+\mu+1)X = -(a^2+1) X$, the Moore determinant $\Delta^{2.A}$ is equal to $\Delta^{[1]}$.
We can also verify that $\Phi(\Delta^{2.A})$ has the same zeroes of $\Delta^{2.C}$ and $\Phi^2(\Delta^{2.A})$ has the same zeroes of $\Delta^{2.B}$.
Hence, Case $2.C$ corresponds to $\Phi(\Omega_1)$ and Case $2.B$ corresponds to $\Phi^2(\Omega_1)$.

Finally, in \textbf{Case 3} let $a$ be a root of $X^3-X+1$. 
For simplicity and consistency, we assume that $a=\mu$, the generator of $\F_{27}^\times$ appearing in Example \ref{ex:F27}.
Then we have that $\Phi(a)=a-1$ and $\Phi^2(a)=a+1$.
Program \cite[Program 6.10]{GithubRepo} returns the following subcases:\\[.2cm]
\begin{tabular}{|l||c|c|c|c||c|c|c|c|}
      \hline
      Case $3.A$ & $s_2$ & $s_3$ & $s_4$ & $s_5$ &$t_2$&$t_3$&$t_4$&$t_5$\\ \hline
      & $a$ &1 & 1& $-(a^2-a+1)$ &$-(a^2+a+1)$ &$-(a+1)$&$a^2+a$&$-(a^2+a+1)$ \\ \hline 
    \end{tabular}\\[.2cm]
    \begin{tabular}{|l||c|c|c|c||c|c|c|c|}
      \hline
      Case $3.B$ & $s_2$ & $s_3$ & $s_4$ & $s_5$ &$t_2$&$t_3$&$t_4$&$t_5$\\ \hline
     &  $-(a^2+a)$ &1 & $a^2-a$& $a$ &$-a^2$&$a-1$&$a^2-1$&$a^2$\\  \hline 
    \end{tabular}\\[.2cm]
            \begin{tabular}{|l||c|c|c|c||c|c|c|c|}
      \hline
      Case $3.C$ & $s_2$ & $s_3$ & $s_4$ & $s_5$ &$t_2$&$t_3$&$t_4$&$t_5$\\ \hline
     &  $-(a^2-a+1)$ &1 & $a^2+1$& $0$ &$-(a^2+a)$&$a^2$&$-a$&$a^2+a-1$\\  \hline 
    \end{tabular}\\

In Case $3.A$ we get a vector space equivalent to $\Omega_1$ under the homothety $X\mapsto (1-a)X$.

In Case $3.B$ we get a vector space equivalent to $\Phi^2(\Omega_1)$.
We can see this by applying $\Phi$ and then the homothety $X \mapsto (a-1)X$ to the coefficients of our polynomials, and noticing that the resulting pair is a prompt for the space $\Omega_1$.

Finally, if we apply the transformation $\Phi$ to the coefficients of $P_0$ and $P_1$ of Example \ref{ex:F27}, we see that we get the coefficients of the table of Case $3.C$.
The pair $(Q_1, Q_2)$ in this latter case then is a prompt for the space $\Phi(\Omega_1)$.

This exhausts all possible cases, and since the spaces $\Omega_1$, $\Phi(\Omega_1)$, $\Phi^2(\Omega_1)$, $\Omega_2$, and $\Phi(\Omega_2)$ make up distinct equivalence classes of spaces $L_{15,2}$, this concludes our proof of the classification theorem. \qed

\begin{remark}
We believe that the classification of spaces $L_{12,2}$ and $L_{15,2}$ in this section is interesting in itself, as the nature of these examples is quite different from anything previously known, for example these spaces can not be obtained from standard spaces.
Moreover, we remark that, by applying \'etale pullbacks to these examples, we can generate spaces $L_{36 d,2}$  and $L_{45 d,2}$ for every integer $d\geq 1$, resulting in large classes of examples useful for future investigation.
\end{remark}

\section{Classification of spaces $L_{4p,2}$}\label{sec:L4p,2}

The classification of spaces $L_{8,2}$ in characteristic 2 (cf. Remark \ref{rmk:L_2lambda,2}), of spaces $L_{12,2}$ in characteristic 3 (\ref{sec:L12,2}) and the examples of standard spaces $L_{20,2}$ in characteristic 5 prompted us to wonder if we can achieve a full classification of spaces $L_{4p,2}$ for every $p$.
This section is devoted to complete such a classification, which thanks to Theorem \ref{thm:generic} and the settled classifications for $p=2, 3$ boils down to the following statement.

\begin{theorem}\label{thm:classL4p,2}
    Let $p$ be a prime number and let $\Omega$ be a space $L_{4p, 2}$. Then $p \in \{2,3,5\}$.
    Moreover, if $p=5$, then $\Omega$ is equivalent to a standard $L_{20, 2}$-space.
\end{theorem}

The strategy for proving the theorem can be summarized as follows: starting from $\Omega$, we construct a non-zero polynomial of small degree and establish that it has more roots than its degree when $p \in \{7, 11\}$ and when $p = 5$ and the space under consideration is not equivalent to a standard space.
Since no single polynomial construction works uniformly for a given value of $p$, we partition the analysis into three subcases 
and employ a tailored strategy for each of them.

Let $\Omega=\langle \omega_1, \omega_2 \rangle_{\F_p}$ be a space $L_{4p,2}$ and recall that it can be built from suitable polynomials $Q_1, Q_2$ of degree 4 such that $\displaystyle \omega_1=\frac{Q_2 dX}{Q_1Q_2^p-Q_1^pQ_2}$ and $\displaystyle \omega_2=\frac{-Q_1 dX}{Q_1Q_2^p-Q_1^pQ_2}$.
We write $Q_1=-c P_\infty$ and $Q_2=a c P_0$ with $P_0, P_\infty \in k[X]$ monic, $c \in k^\times$ and $a \in k - \F_p$.
For every $j=0, \dots, p-1 $ we consider $P_j:= \frac{aP_0+jP_\infty}{a+j}$, the monic polynomial whose zeroes are the elements of $\calP(\Omega) - \calP(\omega_1 +j \omega_2)$.
Recall from Remark \ref{rmk:simplepolesn=2} that $Q_1$ and $Q_2$ have no common root. As a result, all the $P_j$'s are pairwise coprime, a fact that will be used repeatedly in the proofs that follow.
We denote its zeroes and coefficients by writing
\[P_j=X^4 -e_1X^{3}+e_{j,2}X^{2}-e_{j,3}X+e_{j,4}= \prod_{i=1}^4 (X-x_{ij}),\] 
recalling that the $e_1$ is independent of $j$ by Lemma \ref{lem:s1=t1}).

Finally, we denote the residues of $\omega_2$ at $x_{i,j}$ by writing
\begin{equation}\label{eq:omega2}
\displaystyle \omega_2=\frac{-dX}{Q_2^p-Q_1^{p-1}Q_2}=\frac{ -\;dX}{c^p(a^p-a)\prod_{j=0}^{p-1} P_j}=\sum_{j=0}^{p-1}\sum_{i=1}^4\frac{h_{j,i}}{X-x_{j,i}}dX.
 \end{equation}

The relationship between these residues and the polynomial 
\[ \Theta:=P_{\infty}-P_{0}=(e_{\infty,2}-e_{0,2}) X^2-(e_{\infty,3}-e_{0,3})X+(e_{\infty,4}-e_{0,4}).\]
is shown in the following lemma

\begin{lemma}\label{lem:hji}
    For every $0 \leq j \leq p-1$ and $1 \leq i \leq 4$ we have
    \[ h_{j,i} = a^{-(p-1)}c^{-p}\frac{(a+j)^{p-2}}{P_j'(x_{j,i})\Theta^{p-1}(x_{j,i})}.\]
\end{lemma}
\begin{proof}
From the writing (\ref{eq:omega2}), we have that $\displaystyle h_{j,i}=\frac{-1}{c^p(a^p-a) P_j'(x_{j,i})\prod_{k\neq j}P_k(x_{j,i})}$.
Using that $0=(a+j)P_j(x_{j,i})=aP_0(x_{j,i})+jP_\infty(x_{j,i})$ we deduce for every $k\neq j$ the relation
\[P_k(x_{j,i})=\frac{aP_0(x_{j,i})+kP_\infty(x_{j,i})}{a+k}=\frac{(k-j)P_\infty(x_{j,i})}{a+k}\]
which plugged into the expression for the residues yields $\displaystyle h_{j,i}=\frac{-1}{c^p(a+j) P_j'(x_{j,i})P_\infty^{p-1}(x_{j,i})}$.
We can then conclude by remarking that $P_\infty(x_{j,i})=(P_\infty-P_j)(x_{j,i})=\frac{a}{a+j}(P_\infty-P_0)(x_{j,i})$.
\end{proof}

\begin{proposition}\label{prop:Sum}
For every $0\leq j\leq p-1$ we have the following equations:
\begin{enumerate}[label=\ref{prop:Sum}(\roman*)]
    \item 
\[ \sum_{i=1}^4 h_{j,i} \Theta^{p-1}(x_{j,i}) x_{j,i}^k=0 \; \; \mbox{for all} \; \; 0\leq k\leq 2\]\label{prop:Sumi}
\item 
\[ \sum_{i=1}^4 h_{j,i} \Theta^{p-1}(x_{j,i})x_{j,i}^3=\frac{(a+j)^{p-2}}{a^{p-1}c^p}\]\label{prop:Sumii}
\item 
\[ \sum_{i=1}^4 h_{j,i} \Theta(x_{j,i})=0\]\label{prop:Sumiii}
\item 
\[ \sum_{i=1}^4 \frac{1}{P_j'(x_{j,i})\Theta^{p-2}(x_{j,i})}=0.\]\label{prop:Sumiv}
\end{enumerate}
\end{proposition}
\begin{proof}
In order to prove $(i)$ and $(ii)$, we let $0 \leq k \leq 3$ and compute the partial fraction decomposition of $\frac{X^k}{P_j}$.
Since the $x_{j,i}$'s are all distinct, this is
$ \frac{X^k}{P_j}=\sum_{i=1}^4\frac{x_{j,i}^k}{P'_j(x_{j,i})}\frac{1}{(X-x_{j,i})}$, which gives the formula 
\[X^k= \sum_{i=1}^4 \left( \frac{x_{j,i}^k}{P_j'(x_{j,i})}\prod_{\ell \neq i} (X-x_{j,\ell}) \right).\]
By comparing the coefficients of degree $3$ on the two sides of this equation, we get that $\sum_{i=1}^4 \frac{x_{j,i}^k}{P_j'(x_{j,i})}=0$ when $k=0,1,2$ as well as $\sum_{i=1}^4 \frac{x_{j,i}^3}{P_j'(x_{j,i})}=1$.
At the same time, by Lemma \ref{lem:hji} we have
\[\sum_{i=1}^4 h_{j,i} \Theta^{p-1}(x_{j,i}) x_{j,i}^k = \sum_{i=1}^4 a^{-(p-1)}c^{-p}(a+j)^{p-2}\frac{x_{j,i}^k}{P_j'(x_{j,i})}, \]
which is 0 if $k\leq 2$ and $a^{-(p-1)}c^{-p}(a+j)^{p-2}$ for $k=3$, hence proving parts $(i)$ and $(ii)$.

In order to prove $(iii)$, we write 
\[ \left(\sum_{i=1}^4 h_{j,i} \Theta(x_{j,i})\right)^p= \sum_{i=1}^4 h_{j,i} \Theta^{p}(x_{j,i}) = \sum_{i=1}^4 h_{j,i} \Theta^{p-1}(x_{j,i})\Theta(x_{j,i})\]
and since $\Theta$ is a polynomial of degree at most 2, we have by part $(i)$ that 
\[\sum_{i=1}^4 h_{j,i} \Theta^{p-1}(x_{j,i})\Theta(x_{j,i})= \sum_{i=1}^4 h_{j,i} \Theta^{p-1}(x_{j,i}) \left((e_{\infty,2}-e_{0,2})x_{j,i}^2 +(e_{\infty,3}-e_{0,3})x_{j,i} +(e_{\infty,4}-e_{0,4})\right)=0.\]

Finally, part $(iv)$ is obtained by combining Lemma \ref{lem:hji} with part $(iii)$.
\end{proof}

Our main theorem is proved by separately studying three subcases, according to the degree of $\Theta$.
The most relevant quantities that we will be using in our proofs are the Newton sums
\[ N_j(r):= \sum_{i=1}^4 h_{j,i}x_{j,i}^r\]
which satisfy the equation
 \begin{equation}\label{eq:NewtonSums}
  N_j(k+4)-e_{1}N_j(k+3)+e_{j,2}N_j(k+2)-e_{j,3} N_j(k+1) + e_{j,4}N_j(k)=0
 \end{equation}
for all $k\in \N$.

\begin{remark}
All the results of this section are valid more generally when the degree of the $Q_j$'s (and hence of the $P_j$'s) is any $\lambda>0$.
\end{remark}

\subsubsection*{The case $\deg(\Theta)=0$}

If we assume that $\deg(\Theta)=0$, the classification of spaces $L_{4p,2}$ is provided by the following proposition.
\begin{proposition}\label{prop:degtheta=0}
Let $\deg(\Theta)=0$. Then
\begin{itemize}
    \item If $p=5$ then $\Omega$ is equivalent to a standard $L_{20,2}$ space
    \item If $p\in \{7,11\}$ then there is no space $L_{20,2}$.
\end{itemize}
\end{proposition}
\begin{proof}
Since $\deg(\Theta)=0$, we have $e_{\infty,2}=e_{0,2}$ and $e_{\infty,3}=e_{0,3}$. We will denote these by $e_2$ and $e_3$ respectively. Recalling that $(a+j)P_j=aP_0+jP_\infty$, we can write
\[P_j = X^4 -e_1 X^3 +e_2 X^2 - e_3 X +e_{j,4}.\]
After applying a suitable affine transformation, we can assume that $e_1=0$ and $\Theta=e_{\infty,4}-e_{0,4}=1$.
In this way, Proposition \ref{prop:Sumi} rewritten in terms of Newton sums gives $N_j(0)=N_j(1)=N_j(2)=0$, and similarly Proposition \ref{prop:Sumii} gives $N_j(3)=\gamma (a+j)^{p-2}$, with $\gamma\in k^\times$ independent of $j$.
In particular, $N_j(3)\neq 0$ for every $j$.
The higher Newton sums can be obtained using the relation (\ref{eq:NewtonSums}): we have 
\[ \begin{cases} N_j(4)=0\\ N_j(5)=-e_2N_j(3)\\ N_j(6)=e_3 N_j(3)\\  N_j(7)=-e_2 N_j(5)-e_{j,4} N_j(3)= (-e_{j,4}+e_2^2)N_j(3)\\
\vdots \\N_j(10)=(- 2e_3e_{j,4} +3e_2^2e_3)N_j(3), \\N_j(11)= 
(e_{j,4}^2+8e_2^2e_{j,4}+e_2^4+8e_2e_3^2)N_j(3).
\end{cases}\]
We now discuss, according to the value of $p$:
\begin{itemize}
    \item Let $p=5$. Then $N_j(5)=N_j(1)^5=0$ and from it follows that $e_2=0$. Similarly, $N_j(10)=N_j(2)^5=0$ and hence $e_3=0$. 
    Using this, we find that $N_j(15)=-e_{j,4}N_j(11)=-e^3_{j,4}N_j(3)$, and we know at the same time that we have $N_j(15)=N_j(3)^5$, so comparing the two expressions we get that $N_j(3)^4=-e_{j,4}^3$. 
    We can rewrite this equation explicitly to obtain that $\frac{(a+j)^{12}}{a^{16}c^{20}} = - \frac{(a e_{0,4} +j e_{\infty,4})^3}{(a+j)^3}$, which after reordering becomes
    \[ \frac{(a+j)^{15}}{a^{16}c^{20}} = - (a e_{0,4} +j e_{\infty,4})^3 \;\; \forall \;j=0, \dots, 4. \]
    This expression when $j=0$ returns $e_{0,4}^3=-\frac{1}{a^4 c^{20}}$, which plugged in the left hand side of the same expression gives
    \[ \frac{(a^5+j)^{3}}{a^{12}} (-e_{0,4}^3) = - (a e_{0,4} +j e_{\infty,4})^3 \;\; \forall \;j=0, \dots, 4. \]
    We then notice that the two sides of the last equation are polynomials of degree $3$ in $j$ and that this relation is satisfied for 5 distinct values of $j$, hence the two expressions must coincide as polynomials in $j$. In particular, the degree 1 coefficients coincide and we have $-\frac{3 e_{0,4}^3}{a^2}=-3 a^2 e_{0,4}^2 e_{\infty,4}$, which implies that $e_{0,4}= a^4 e_{\infty,4}$.\\
    To conclude, we pick $a_1, a_2 \in k^\times$ such that $a_1^4 = -e_{\infty,4}$ and $a_2=-a a_1$ so that we have that $a_2^4 = -a^4 e_{\infty,4}=-e_{0,4}$.
    We can then choose $\mu = -\frac{c}{a_1}$ and verify that the polynomials $Q_1, Q_2$ satisfy 
    \[\begin{cases} 
    Q_1=-c P_\infty = \mu(a_1 X^4 - a_1^5) \\
    Q_2 = ac P_0 = \mu(a_2 X^4 -a_2^5).
    \end{cases}\]
    By Proposition \ref{prop:standard} and \cite[Proposition 3.1]{FresnelMatignon23}, we conclude that $\Omega$ is equivalent to a standard $L_{20,2}$-space.
    \item Let $p=7$. Then $N_j(7)=N_j(1)^7=0$ and from it follows that $-e_{j,4}+e_2^2=0$ for every $j=0,\dots, 6$.
    But the set $\{e_{j,4}: 0 \leq j \leq 6\}$ has 7 elements because otherwise there would be $i \neq j$ with $P_i=P_j$ and that is not possible. It follows that the polynomial $-Z+e_2^2$ has 7 distinct roots, and hence a contradiction.
    \item Let $p=11$. Then the argument is completely analogous: we have $N_j(11)=0$ and then we find that the polynomial $Z^2+8e_2^2Z+e_2^4+8e_2e_3^2$ has 11 distinct roots and hence a contradiction.
\end{itemize}
\end{proof}

\subsubsection*{The case $\deg(\Theta)=1$}
\begin{proposition}\label{prop:degtheta=1}
    Let $\deg(\Theta)=1$ and $p\in \{5,7,11\}$. Then there are no spaces $L_{4p,2}$.
\end{proposition}
\begin{proof}
As in the previous case, we start by applying an affine transformation of $k$ so that we can assume $e_{0,4}=e_{\infty,4}$ (which we denote by $e_4$) and $e_{\infty,3}-e_{0,3}=-1$.
We have then that $\Theta=X$ as well as
\[ P_j = X^4-e_1X^3+e_2X^2-e_{j,3}X+e_4.\]
The $P_j$'s are coprime, so that $e_4 \neq 0$, and that $e_{j,i}=\frac{ae_{0,i}+je_{\infty,i}}{a+j}$, hence $e_{\infty,3}\neq e_{0,3}$ ensures that $e_{j,3} \neq e_{i,3}$ for every $i\neq j$.

We now look at the Newton sums: by Proposition \ref{prop:Sumiii} we have that $N_j(1)=0$, by Proposition \ref{prop:Sumi} we have that $N_j(p-1)=N_j(p)=N_j(p+1)=0$, and by Proposition \ref{prop:Sumii} we have that $N_j(p+2)\neq 0$ for all $0\leq j \leq p-1$.

Using equation (\ref{eq:NewtonSums}), we first get that $N_j(p+2)=-e_4 N_j(p-2)$ and hence Proposition \ref{prop:Sumii} that $N_j(p-2) \neq 0$.
The same equation for other values of $k$ gives the system
\begin{equation}\label{eq:recursiveNewt}
    \begin{cases}
    -e_{j,3} N_j(p-2) +e_4 N_j(p-3)=0\\
    e_2N_j(p-2)-e_{j,3}N_j(p-3)+e_4 N_j(p-4)=0\\
       -e_1N_j(p-2)+e_{2}N_j(p-3)-e_4 N_j(p-4)+e_5 N_j(p-5)=0\\
    \vdots \\
    N_j(5)-e_1 N_j(4)+e_2 N_j(3)-e_{j,3} N_j(2)+e_4N_j(1)=0,
\end{cases}
\end{equation}
from which we can recursively retrieve the expression of $N_j(1)$ in terms of $N_j(p-2)$ (see detailed calculations in \cite[Program 1]{GithubRepo2}).
This is discussed according to the different values of $p$:
\begin{itemize}
    \item Let $p=5$. Then from (\ref{eq:recursiveNewt}) we get that \[0=e_4^2N_j(1)=(e_{j,3}^2-e_2 e_4)N_j(3).\] Since $N_j(3) \neq 0$, we have $e_{j,3}^2-e_2e_4=0$ for $0\leq j \leq 4$. As the set $\{ e_{j,3} : 0\leq j \leq 4\}$ is of cardinality $5$, this leads to a contradiction and hence there is no space $L_{20,2}$
    \item Let $p=7$. Then from (\ref{eq:recursiveNewt}) we get that \[0=e_4^4 N_j(1)=(e_{j,3}^4-3e_2e_4e_{j,3}^2+2e_1e_4^2 e_{j,3}-e_2^2e_4^2+e_4^3 )N_j(5).\] Since $N_j(5) \neq 0$, we have that the expression in parenthesis vanishes for $0\leq j \leq 6$. This expression is a unitary polynomial of degree 4 in $e_{j,3}$. As the set $\{ e_{j,3} : 0\leq j \leq 6\}$ is of cardinality $7$, this leads to a contradiction and hence there is no space $L_{28,2}$
    \item Let $p=11$. Then from (\ref{eq:recursiveNewt}) we get that \begin{align*}
    0=e_4^8 N_j(1)=(e_{j,3}^8+4e_2e_4e_{j,3}^6-5e_1e_4^2e_{j,3}^5+(4e_2^2e_4^2-5e_4^3)e_{j,3}^4+2e_1e_2e4^3e_{j,3}^3+ \\
    +(e_2^3e_4^3-5e_1^2e_4^4+e_2e_4^4)e_{j,3}^2+(e_1e_2^2e_4^4+5e_1e_4^5)e_{j,3}+e_2^4e_4^4-3e_1^2e_2e_4^5-3e_2^2e_4^5+e_4^6)N_j(9).
    \end{align*}
    As in the cases above, this entails the vanishing of a unitary polynomial of degree 8 when evaluated in the $e_{j,3}$'s. As the set $\{ e_{j,3} : 0\leq j \leq 10\}$ is of cardinality $11$, this leads to a contradiction and hence there is no space $L_{44,2}$.
\end{itemize}
 \end{proof}

\subsubsection*{The case $\deg(\Theta)=2$}\label{sec:degtheta=2}

Treating this case is more involved, as we can not assume that the polynomial $\Theta$ is a monomial anymore. 
However, after possibly translating the set of poles, we can assume that $e_{0,3}=e_{\infty,3}$, and we can then discuss two separate cases.

\begin{proposition}\label{prop:degtheta=2monomial}
    Let $p\in \{5,7,11\}$ and $\deg(\Theta)=2$. Assume moreover that $e_{0,3}=e_{\infty,3}$ and $e_{0,4}=e_{\infty,4}$. Then there are no spaces $L_{4p,2}$.
\end{proposition}
\begin{proof}
    In this case, we can apply a homothety to the set of poles to get that $e_{\infty,2}-e_{0,2}=1$ (notice that this does not affect the conditions that $e_{0,3}=e_{\infty,3}$ and $e_{0,4}=e_{\infty,4}$). It results that $\Theta=X^2$ and we can compute Newton sums as before: by Proposition \ref{prop:Sumiii} we have that $N_j(2)$=0, by Proposition \ref{prop:Sumi} we have that $N_j(2p-2)=N_j(2p-1)=N_j(2p)=0$ and Proposition \ref{prop:Sumii} gives that $N_j(2p+1)= \frac{(a+j)^{p-2}}{a^{p-1}c^p}\neq 0$ for all $0 \leq j \leq p-1$.
    As in the previous proof, we can apply equation (\ref{eq:NewtonSums}) to get that $N_j(2p-3)\neq 0$, and we apply (\ref{eq:NewtonSums}) recursively to express $N_j(2)$ as a function of $N_j(2p-3)$ (see detailed calculations in \cite[Program 2]{GithubRepo2}). 
    More specifically, this recursion gives the formula
    \[N_j(2)= V(e_{2,j}) N_j(2p-3), \]
        where $V\in k(e_4)[e_1,e_3][Z]$ is a polynomial of degree $p-3$.
        Since $N_j(2)=0$ and $N_j(2p-3)\neq 0$, the formula above implies that $V(e_{2,j})=0$ for every $j=0,\dots, p-1$, and since this is satisfied for $p$ distinct values of $e_{2,j}$, we have that all the coefficients of $V$ vanish. 
        The computation of the two coefficients of higher degree done in \cite[Program 2]{GithubRepo2} returns the following table.
        
        \begin{center}
\begin{tabular}{|c|c|c|c|}
\hline
$p$ & $\deg(V)$ & $\mathrm{coeff}(V,Z^{p-3})$ & $\mathrm{coeff}(V,Z^{p-4})$\\
\hline
\hline
$5$ & $2$ & $-2  e_4^{-3} e_3$ & $e_4^{-4}e_3^3  + 3e_4^{-2} e_1 $ \\
$7$ & $4$ & $-2 e_4^{-5}e_3$ & $e_4^{-6}e_3^3+3e_4^{-4}e_1 $\\
$11$ & $8$ & $-2e_4^{-9}e_3$ & $e_4^{-10}e_3^3+ 3e_4^{-8}e_1$\\

 \hline
\end{tabular}
\end{center}
Since $0$ can not be a common root of the $P_j$'s, we have that $e_4\neq 0$, and then the vanishing of the coefficients of $V$ implies that $e_1=e_3=0$. In other words, we have that
        \[P_j = X^4+e_{j,2}X^2+e_4\]
is a biquadratic polynomial for every $j$. It easily follows that $N_j(k)=0$ for every even $k$, while for odd $k$ we can compute these from $N_j(2p-3)$ using the same recursive process: we obtain formulas \begin{equation}\label{eq:W1W2}
    \begin{cases}
        N_j(1)= W_{1}(e_{2,j}) N_j(2p-1) = -\frac{W_{1}(e_{2,j})}{e_4} \frac{(a+j)^{p-2}}{a^{p-1}c^p} \\
        N_j(p)= W_2(e_{2,j}) N_j(2p-1) = -\frac{W_{2}(e_{2,j})}{e_4} \frac{(a+j)^{p-2}}{a^{p-1}c^p},
    \end{cases}
\end{equation}

with $W_{1}$ and $W_{2}$ in $k(e_4)[Z]$ 
as in the following table (obtained in \cite[Program 3]{GithubRepo2}):
            \begin{center}
\begin{tabular}{|c|c|c|}
\hline
$p$ & $W_1(Z)$ & $W_2(Z)$\\
\hline
\hline
 $5$ &$ \frac{Z^3}{e_4^3} + 3\frac{Z}{e_4^2}$ &$\frac{Z}{e_4}$\\
   $7$ & $\frac{Z^5}{e_4^5} + 3\frac{Z^3}{e_4^4} + 3\frac{Z}{e_4^3} $&$ 6\frac{Z^2}{e_4^2} + \frac{1}{e_4}$ \\
  $11$ & $\frac{Z^9}{e_4^{9}} + 3\frac{Z^7}{e_4^8} + 10\frac{Z^5}{e_4^7} + 2\frac{Z^3}{e_4^6} + 5\frac{Z}{e_4^5}  $&$10\frac{Z^4}{e_4^4} + 3\frac{Z^2}{e_4^3} + \frac{10}{e_4^2}$ \\

 \hline
\end{tabular}
\end{center}

From formulas (\ref{eq:W1W2}), using the fact that $e_{j,2} = e_{\infty,2}-\frac{a}{a+j}$, and raising the first line to the power $p$, we get that

\[ \begin{cases}
N_j(1)^p= \widehat{W}_1(j)\\
N_j(p)= \widehat{W}_2(j),
\end{cases} \]
where $\widehat{W}_1, \widehat{W}_2 \in k(e_4)[Z]$ are polynomials of degree $p-2$.
As $N_j(1)^p=N_j(p)$ for every $j\in\{0, \dots, p-1\}$, the polynomials $\widehat{W}_1$ and $\widehat{W}_2$ must coincide.
However, $\widehat{W}_1$ has simple roots, as one can verify through a computation of the discriminant (see \cite[Program 3]{GithubRepo2}), while $\widehat{W}_2$ has a root in $-a$ with multiplicity at least 2 (because $\deg(W_2)<p-3$ for $p\in \{5,7,11\}$).
This leads to a contradiction.
\end{proof}

It remains to study the case where $\Theta$ is a binomial.
To do this, we first establish a result that will help us with our strategy.
\begin{lemma}\label{lem:gcdex}
    Let $A, B \in k[X]$ be coprime polynomials and assume that $B$ has simple roots $x_1, \dots, x_b$. In particular $b=\deg(B)$.
    Let $(U, V) \in k[X]^2$ be the unique pair satisfying $\deg(U)<b$, $\deg(V)<\deg(A)$ and $AU+BV=1$ and write $U(X)=\sum_{i=0}^{b-1} u_i X^i$. Then
    \[ u_{b-1}= \sum_{i=1}^{b} \frac{1}{A(x_i)B'(x_i)}.\]
\end{lemma}
\begin{proof}
By partial fraction decomposition, we have 
$\frac{U(X)}{B(X)} = \sum_{i=1}^b \frac{\alpha_i}{X-x_i},$
with $\alpha_i= \frac{U(x_i)}{B'(x_i)} = \frac{1}{A(x_i)B'(x_i)},$ and expanding this decomposition gives 
\[ \frac{U(X)}{B(X)}=\sum_{i=1}^{b}\frac{\alpha_i}{X-x_i}=\frac{\sum_{i=1}^{b}\alpha_i\prod_{j\neq i}(X-x_j)}{B(X)}. \]
As a result, $U(X)=\sum_{i=1}^b \alpha_i\prod_{j\neq i}(X-x_j)$, and its coefficient of degree $b-1$ is
\[ u_{b-1}= \sum_{i=1}^b\alpha_i=\sum_{i=1}^{b} \frac{1}{A(x_i)B'(x_i)}.\]
\end{proof}

We are now going to prove that there are no spaces $L_{4p,2}$ in the last remaining case.
Our conclusion relies on computations carried out in \cite{GithubRepo3}, which could not be performed with \emph{Macaulay2} due to lack of a suitable function computing the coefficients of a B\'ezout identity in polynomial rings. 
As a result, we chose to employ the software \emph{Maple}.

\begin{proposition}\label{prop:degtheta=2binomial}
    Let $p\in \{5,7,11\}$, $\deg(\Theta)=2$, and assume that $e_{0,3}=e_{\infty,3}$ and $e_{0,4} \neq e_{\infty,4}$. Then there are no spaces $L_{4p,2}$.
\end{proposition}
\begin{proof}
We first apply a homothety to the set of poles in such a way that $e_{\infty,2}-e_{0,2}=-(e_{\infty,4}-e_{0,4})$ and we set $\nu:=e_{\infty,2}-e_{0,2}$.
Then $\Theta=\nu (X^2-1)$.
We note that 
\[e_{j,4}+e_{j,2} = \frac{a (e_{0,4}+ e_{0,2})+j(e_{\infty,4} +e_{\infty,2})}{a+j} = (e_{0,4}+e_{0,2})\] and in particular it is independent of $j$, and so there exists an element $\mu \in k$ such that $\mu=e_{j,4}+e_{j,2}$. 

In the first part of the proof, we show that we have $e_1=e_3=0$.
Since $\Theta$ is a binomial, the methods used so far can not be applied in the same way, and we need to use Lemma \ref{lem:gcdex}:
for every $j \in \{0,\dots,p-1\}$ we let $A=\Theta^{p-2}$ and $B=P_j$.
Then $A$ and $B$ are coprime, because by definition $\Theta= P_\infty - P_0$ and hence any common root between $\Theta$ and $P_j=\frac{aP_0 + j P_\infty}{a+j}$ must also be a root of $P_0$ and $P_\infty$, contradicting the fact that the $P_0$ and $P_\infty$ have distinct roots.
By applying the lemma we find that 
\[u_{3} = \sum_{i=1}^4 \frac{1}{\Theta^{p-2}(x_{ji})P_j'(x_{ji})}=0\]
where $u_3$ is the leading term of the polynomial $U$ appearing in the Lemma and the vanishing of the sum on the right is given by \ref{prop:Sumiv}.
The polynomial $U$ can be computed using \emph{Maple}'s function ``gcdex'' \footnote{We note that a similar function, called \emph{gcdCoefficients} exists in \emph{Macaulay2} but in its current state experiences issues when applied to polynomials with coefficients in the field of fractions of a polynomial ring. As a result, the use of Maple was deemed necessary by the authors for this step.}.
As a result, we obtain the coefficient $u_3$ (cf. Program \cite{GithubRepo3}) as an element of $k[\nu, \mu, e_1, e_3][e_{j,2}]$ of degree at most $p-3$ in $e_{j,2}$.
Since this vanishes for $p$ distinct values of $e_{j,2}$, then the coefficients of such polynomial (which are element of $k[\nu, \mu, e_1, e_3]$ must all vanish, giving a system of $p-2$ equations).
Solving this system computationally (cf. Program \cite[Program 7.1]{GithubRepo3}) gives the following conditions:

\begin{itemize}
    \item For $p=5$, we have $e_1=e_3$ and $e_1(e_1^2-3)=0$.
    \item for $p \in \{7,11\}$, we have $e_1=e_3=0$, i.e. $P_j$ is a biquadratic polynomial for all $j$.
\end{itemize}

Recall that, if we consider the expressions $R_k$'s of Convention \ref{conv:Ri} as polynomials whose variables correspond to the coefficients of $P_0$ and $P_\infty$, then the existence of a solution to the system 
\[ \begin{cases}
    R_1 = 1\\
    R_i = 0 & \mbox{for} \;\;2 \leq i \leq 4 (p-1)
\end{cases}\]
corresponds to the existence of a space $L_{4 p, 2}$.
Thanks to the simplifications obtained so far, this system becomes computationally tractable and its resolution leads us to the following: 

\begin{itemize}
    \item Let $p=5$. Program \cite[Program 7.2]{GithubRepo3} shows that this system has no solutions. Hence there are no spaces $L_{20,2}$ under the assumptions of the Proposition.
    \item Let $p \in \{7,11\}$. As $P_0$ is biquadratic, we can apply a homothety to the set of poles to have $s_{0,4}=1$. This does not change the fact that $P_j$ is biquadratic for all $j$. We then find a system in the variables $s_{0,2}, s_{\infty,2}$ and $s_{\infty,4}$, that Program \cite[Program 7.3]{GithubRepo3} shows to have no solution. Hence there are no spaces $L_{28,2}$ when $p=7$ and no spaces $L_{44,2}$ when $p=11$.
\end{itemize}
\end{proof}

\appendix

\section{Moore determinants and invariant theory}\label{app:Moore}

In this section, we collect several results on Moore determinants.
With the exception of Lemma \ref{lem:Moore} and Lemma \ref{lem:Anew}, for which we provide a proof, these results are already known (from work of Dickson \cite{Dickson11}, Elkies in \cite{Elkies99} and Fresnel-Matignon in \cite{FresnelMatignon23}), and are therefore recalled with minimal or no proof.

Let $k$ be a field of characteristic $p>0$, and denote by $F:k\to k$ the Frobenius automorphism $x \mapsto x^p$.
The \emph{Moore determinant} of an $n$-tuple $\underline{a}:=(a_1, \dots, a_n) \in k^n$ is defined as
\[\Moore{n}{a}:= \begin{vmatrix}
a_1 & a_2 & \dots & a_n \\
a_1^p & a_2^p & \dots & a_n^p \\
\vdots & \vdots & \dots & \vdots \\
a_1^{p^{n-1}} & a_2^{p^{n-1}} & \dots & a_n^{p^{n-1}} \\
\end{vmatrix}.  \]

\begin{remark}\label{rmk:elementaryMoore}
Here we list the first elementary results on Moore determinants.
By multilinearity of determinants, for every $\alpha \in k$, we have that
\[\Delta_{n}(\alpha \underline{a})=\alpha^{1+p+\dots+p^{n-1}} \Moore{n}{a}.\] 
Moreover, given an invertible matrix $M \in GL_n(\F_p)$, we have (\cite[p. 80]{Elkies99}) that
\[\Delta_{n}(\underline{a} M) =\Delta_{n}(\underline{a}) \det(M).\]
These relations are used in the proof of Proposition \ref{prop:QiMiii}.
Finally, Moore shows the following identity:
\begin{equation}\label{eq:Mooreprod}
\Moore{n}{a}= \prod_{i=1}^n \prod_{\epsilon_{i-1}\in \F_p} \dots \prod_{\epsilon_{1}\in \F_p} (a_i+\epsilon_{i-1}a_{i-1}+\dots+\epsilon_1a_1).
\end{equation}
As a result, we have that $\Moore{n}{a}\neq 0$ if, and only if, the $a_i$'s are $\F_p$-linearly independent.
\end{remark}

\begin{theorem}\label{thm:FM4.1.}
For every $n$-tuple $\underline{a} \in k^n$, and for every $1\leq i \leq n$, we define the $n-1$-tuple $\underline{\hat{a_i}} := (a_1, \dots, a_{i-1}, a_{i+1}, \dots, a_n)$.
Then, we have the formula
\[\Delta_n\left(\Delta_{n-1}(\underline{\hat{a}_1}), \dots,(-1)^{i-1}\Delta_{i-1}(\underline{\hat{a}_i}), \dots, (-1)^{n-1}\Delta_{n-1}(\underline{\hat{a}_n})\right) = \Moore{n}{a}^{1+p+\dots+p^{n-2}}. \]
\end{theorem}
\begin{proof}
This is the special case $m=0$ of \cite[Theorem 4.1]{FresnelMatignon23}. 
\end{proof}

\begin{definition}\label{defn:strucpoly}
For every $V \subset k$ is a $\F_p$-vector space of dimension $n$, the \emph{structural polynomial} of $V$ is
\[ P_V(X):=\prod_{v\in V} (X-v) \in k[X].\]
It is the unique monic polynomial of degree $p^n$ such that $V$ is the set of zeroes of $P_V$.
\end{definition}

\begin{lemma}[Proposition 2.2. of \cite{FresnelMatignon23}]\label{lem:FM2.2.}
Let $V \subset k$ be a $\F_p$-vector space of dimension $n$.
The structural polynomial $P_V$ is additive and for every choice of basis $\ul{v}=(v_1, \dots, v_n)$ of $V$ it satisfies the identity
\[P_V(X)=\frac{\Delta_{n+1}(\ul{v},X)}{\Moore{n}{v}}=X^{p^n}+\cdots + (-1)^n \Moore{n}{v}^{p-1}X.\]
\end{lemma}
\begin{proof}
Consider $\Delta_{n+1}(\ul{v},X)$ as a polynomial in $k[X]$. 
The development of the determinant along the last column gives 
\[ \Delta_{n+1}(\ul{v},X)= \Moore{n}{v} X^{p^n}+\cdots + (-1)^n \Moore{n}{v}^{p}X,  \]
which results in the second equality of the lemma.
On the other hand, we have by definition that $\Delta_{n+1}(\ul{v},v)=0$ for every $v\in V$, which proves the first equality.
Additivity then follows from the fundamental theorem of additive polynomials, as $V$ is an additive subgroup of $k$.
\end{proof}

The structural polynomial is tightly related to the invariant theory of finite groups, via its tight connections to Dickson invariants: if $\F_p[X_1, \dots, X_n]$ is a polynomial ring on $n$ variables over $\F_p$, and $V_n$ is the $\F_p$-vector space $\langle X_1, \dots, X_n \rangle_{\Fp}$, one defines the Dickson invariants $c_{n,i} \in \F_p[X_1, \dots, X_n]$ for $0 \leq i\leq n$ as the coefficients of the structural polynomial of $V_n$.
More precisely, we have

\[P_{V_n}(T)=\frac{\Delta_{n+1}(\underline X,T)}{\Delta_n(\ul{X})}=\prod_{v\in V_n}(T-v)=\sum_{i=0}^{i=n}(-1)^{n-i}c_{n,i}T^{p^i}\in \F_p[X_1,\dots,X_n][T].\]

Note that the $c_{n,i}$'s are invariant under the natural action of $\GL_n(\F_p)$ on $\F_p[X_1, \dots, X_n]$ and Dickson proved that they generate the algebra of polynomials that are invariant by this action.
In the following proposition, we collect the properties of Dickson invariants that will be used in Section \ref{subsec:applyingPagotn}.

\begin{proposition}\label{prop:Dickson}
    Let $\F_p[X_1, \dots, X_n]$ be the polynomial ring on $n$ variables over $\F_p$, and for every $1\leq j \leq n$ let $P_{V_j}$ be the structural polynomial of the $\F_p$-vector space $V_j= \langle X_1, \dots, X_j \rangle_{\Fp}$.
    For $0 \leq i < j \leq n$ denote by $c_{j,i}$ the $i$-th Dickson invariant in the subring $\F_p[X_1, \dots, X_j]$. Then, we have the following identities
   \begin{enumerate}[label=(\roman*),ref=\ref{prop:Dickson}(\roman*)]
\item \label{eq:Dickson.i}\[P_{V_{n-1}}(X_n)^{p-1}=\left(\frac{\Delta_n(X_1,\dots,X_n)}{\Delta_{n-1}(X_1,\dots,X_{n-1})}\right)^{p-1}=c_{n,n-1}-
c_{n-1,n-2}^p.\]
\item \label{eq:Dickson.ii} \[c_{n,i}= \frac{\det (\underline X,\underline{X}^p,\cdots,\widehat{\underline {X}^{p^i}},\cdots,{\underline X}^{p^n})}{ \Delta_n(X_1,\dots,X_n) }.\]
\item \label{eq:Dickson.iii}\[c_{n,n-1}= \sum_{\lceil\underline  x\rceil \in \mathbb P(\F_p^n)}\ \prod_{\lceil\underline  \epsilon\rceil \in \mathbb P (\F_p^n)|\sum_i\epsilon_ix_i\neq 0} (\sum_i\epsilon_iX_i)^{p-1}.\]
    \end{enumerate}
\end{proposition}

\begin{proof}
    \begin{enumerate}
        \item See \cite[Proposition 1.3 b)]{Wilkerson83}
        \item See \cite[Proposition 1.3 a)]{Wilkerson83}
        \item This is a special case of a theorem by Stong and Tamagawa as stated in \cite[Theorem 6.1.7.]{NeuselSmith10}
    \end{enumerate}
\end{proof}

\begin{lemma} \label{lem:Moore} Let $n\geq 1,\ m\geq 1$ and $Y_1,\dots,Y_n$, $X_1,\dots,X_{m}$ be variables over $\F_p$. Then
\begin{align} \Delta_{n}(\Delta_{m+1}(Y_1,X_1,X_2,\dots,X_{m}),\dots,\Delta_{m+1}(Y_n,X_1,X_2,\dots,X_{m}))= \nonumber \\ = \Delta_{m}(X_1,\dots,X_{m})^{p+p^2+\dots+p^{n-1}}\Delta_{n+m}(Y_1,\dots,Y_n, X_1,\dots, X_{m}).\label{F}
\end{align}

\end{lemma}

\begin{proof}
We proceed by induction on $n$.
If $n=1$, we interpret the expression $p+p^2+\dots+p^{n-1}$ as 0. 
If $n=2$, then the Lemma coincides with \cite[Corollary 4.2]{FresnelMatignon23} (also specialized to the case $n=2$ in the notation of the Corollary).
Let us then assume $n>2$ and that the Lemma is satisfied when replacing $n$ with $n-1$.

We denote by $V_m$, the $\F_p$-vector space $\langle X_1,X_2,\dots, X_m\rangle_{\Fp}$, and by $P_{V_m}$ the structural polynomial of $V_m$.
It then follows from Lemma \ref{lem:FM2.2.} that
\[\Delta_{n}(\Delta_{m+1}(Y_1,X_1,X_2,\dots,X_{m}),\dots,\Delta_{m+1}(Y_n,X_1,X_2,\dots,X_{m}))=\]\[\Delta_{n}((-1)^m \Delta_{m}(\underline X)  P_{V_m}(Y_1),\cdots,(-1)^m \Delta_{m}(\underline X)P_{V_m}(Y_n))=\]\[(-1)^{mn} \Delta_{m}(\underline X)^{1+p+\cdots+p^{n-1}}\Delta_{n}(P_{V_m}(Y_1),\cdots,P_{V_m}(Y_n)). \]

From this, we deduce that the identity \eqref{F} is equivalent to
\begin{align}
 \Delta_{n+m}(Y_1,\dots,Y_n, X_1,\dots, X_{m})=(-1)^{mn} \Delta_{m}(\underline X)\Delta_{n}(P_{V_m}(Y_1),\cdots,P_{V_m}(Y_n)).\label{G}
\end{align}

Moreover, we remark that 
\[ \Delta_{n+m}(Y_1,\dots,Y_n, X_1,\dots, X_{m})=(-1)^{m} \Delta_{n+m}(Y_1,\dots,Y_{n-1}, X_1,\dots, X_{m},Y_n)=\] \[(-1)^{m}\Delta_{n-1+m}(Y_1,\dots,Y_{n-1}, X_1,\dots,X_{m})\prod_{v \in \langle Y_1,\dots,Y_{n-1}, X_1,\dots,X_{m}\rangle_{\Fp}}(Y_n+v),\]
and then we can apply the identity \eqref{G} replacing $n$ with $n-1$ (satisfied by the inductive hypothesis) to show that \eqref{F} is equivalent to
\[\Delta_{n}(P_{V_m}(Y_1),\cdots,P_{V_m}(Y_n))=\Delta_{n-1}(P_{V_m}(Y_1),\cdots,P_{V_m}(Y_{n-1}))\prod_{v \in \langle Y_1,\dots,Y_{n-1}, X_1,\dots,X_{m} \rangle}(Y_n+v)\]
which, by applying again Lemma \ref{lem:FM2.2.} is equivalent to 
\[\prod_{w\in \langle P_{V_m}(Y_1),\cdots,P_{V_m}(Y_{n-1})\rangle_{\Fp}}(P_{V_m}(Y_n)+w) =\prod_{v\in \langle Y_1,\dots,Y_{n-1}, X_1,\dots,X_{m}\rangle_{\Fp}}(Y_n+v).\]
\end{proof}

\begin{remark}
When $m=1$ the identity of Lemma \ref{lem:Moore} is already known. It appears for example in work of Elkies \cite[p.81]{Elkies99} and it is used in the proof of Proposition \ref{prop:standard}.
\end{remark}

To produce certain formulas that we need in Section \ref{sec:Pagotn}, let us introduce the following determinants.
For every non-zero $n$-tuple $\underline{\epsilon}=(\epsilon_1, \dots, \epsilon_n) \in \F_p^n - \underline{0}$ we define 

\[ \Delta_{\underline{\epsilon}} (\underline{a},X):= \begin{vmatrix}
\epsilon_1& \epsilon_2 & \dots & \epsilon_n & 0\\
a_1 & a_2 & \dots & a_n &X\\
a_1^p & a_2^p & \dots & a_n^p &X^p\\
\vdots & \vdots & \ddots & \vdots & \vdots \\
a_1^{p^{n-1}} & a_2^{p^{n-1}} & \dots & a_n^{p^{n-1}} &X^{p^{n-1}}\\

\end{vmatrix} \]
and
\[ \delta_{\underline{\epsilon}} (\underline{a}):= \begin{vmatrix}
\epsilon_1& \epsilon_2 & \dots & \epsilon_n\\
a_1 & a_2 & \dots & a_n \\
a_1^p & a_2^p & \dots & a_n^p\\
\vdots & \vdots & \ddots & \vdots  \\
a_1^{p^{n-2}} & a_2^{p^{n-2}} & \dots & a_n^{p^{n-2}} \\
\end{vmatrix} \]

\begin{proposition}[Proposition 2.3. of \cite{FresnelMatignon23}]\label{prop:FM2.3.}
Let $W\subset k$ be a $\F_p$ vector space of dimension $n$.
Let $\underline{a}$ be an $\F_p$-basis of $W$ and $\underline{a}^\star=(a_1^\star, \dots, a_n^\star)$ be its dual basis.
For every non-zero $n$-tuple $\underline{\epsilon}=(\epsilon_1, \dots, \epsilon_n) \in \F_p^n - \underline{0}$, we denote by $\varphi_{\underline{\epsilon}}$ the linear homomorphism $W \to \F_p$ given by $\varphi_{\underline{\epsilon}}:=\sum_{i=1}^n \epsilon_i a_i^\star$.
Recall that we have defined $P_{\ker \varphi_{\underline{\epsilon}}}(X) := \prod_{w \in \ker \varphi_{\underline{\epsilon}}} (X-w) $.

Then $\delta_{\underline{\epsilon}} (\underline{a})\neq 0$ and we have the identity
\begin{equation}\label{eq:deltae}
P_{\ker \varphi_{\underline{\epsilon}}}(X) =\frac{\Delta_{\underline{\epsilon}} (\underline{a},X)}{\delta_{\underline{\epsilon}} (\underline{a})}=X^{p^{n-1}}+\dots+(-1)^{n-1}\delta_{\underline{\epsilon}}(\underline{a})^{p-1} X.
\end{equation}
Moreover, denoting by $\underline{e_i}$ the $i$-th element of the standard basis of $\F_p^n$, we have the following formulas:
\begin{eqnarray}
\Delta_{\underline{e_i}}(\underline{a},X)=(-1)^{i-1}\Delta_{n}(\underline{\hat{a_i}}, X) & \;\; \forall \; \; i=1, \dots, n\\
\Delta_{\underline{\epsilon}} (\underline{a},X)=\sum_{i=1}^n \epsilon_i \Delta_{\underline{e_i}}(\underline{a},X) = \sum_{i=1}^n (-1)^{i-1} \epsilon_i \Delta_{n}(\underline{\hat{a_i}}, X) \\
\delta_{\underline{\epsilon}} (\underline{a}) = \sum_{i=1}^n (-1)^{i-1} \epsilon_i\Moore{n-1}{\hat{a_i}}. \label{eq:deltaepsilona}
\end{eqnarray}
\end{proposition}

\begin{proposition}[Proposition 5.1. of \cite{FresnelMatignon23}]\label{prop:FM5.1.}
Let $k$ be an algebraically closed field containing $\Fp$. Let $V(\Delta_n):=\{(a_1, \dots, a_n) | \Moore{n}{a}=0 \}$. Then the map

\[
\begin{tikzcd}[row sep=0.2cm]
\varphi: & k^n \arrow[r] & k^n \\
&\ul{a} \arrow[r, mapsto] & \left((-1)^{i-1}\Moore{n-1}{\hat{a_i}}\right)_{i=1, \dots, n}
\end{tikzcd}\]

induces a surjective function $k^n -  V(\Delta_n) \to k^n -  V(\Delta_n)$.
This satisfies the property that
\[\varphi^2(\underline{a}) = (-1)^{n-1}\Moore{n}{a}^{1+p+\dots+p^{n-3}}(\underline{a})^{p^{n-2}}.\]
Moreover, $\underline{a}, \underline{a}'$ are such that $\varphi(\underline{a})=\varphi(\underline{a}')$ if, and only if, $\underline{a}=\theta\underline{a}'$ for some $\theta$ satisfying \[\theta^{1+p+\dots+p^{n-2}}=1.\]
\end{proposition}

Finally, we prove two results that are not in \cite{FresnelMatignon23} and that are used in Section \ref{sec:Pagotn}.

\begin{lemma}\label{lem:structuralMoore} Let $X_1,\dots,X_n$ be variables over $\Fp$, and  for every $0\leq t \leq n$ let $V_{n-t}=\langle X_{1},\dots,X_{n-t} \rangle_{\F_p}$. 
Then, we have that
\[\Delta_n(X_1,\dots,X_n)=\Delta_{n-t}(X_{1},\dots,X_{n-t} )\Delta_{t}(P_{V_{n-t}}(X_{n-t+1}),\dots,P_{V_{n-t}}(X_n)), \]
where $P_{V_{n-t}}(X)$ is the structural polynomial of $V_{n-t}$ (see Definition \ref{defn:strucpoly}).
\end{lemma}
\begin{proof}
From Moore's formula (\ref{eq:Mooreprod}) we have 
\begin{align*}
    \Moore{n}{X}& = \prod_{i=1}^n \prod_{\epsilon_{i-1}\in \F_p} \dots \prod_{\epsilon_{1}\in \F_p} (X_i+\epsilon_{i-1}X_{i-1}+\dots+\epsilon_1X_1)  = A\cdot B,
\end{align*}
where 
\[ A :=  \prod_{i=1}^{n-t} \prod_{\epsilon_{i-1}\in \F_p} \dots \prod_{\epsilon_{1}\in \F_p} (X_i+\epsilon_{i-1}X_{i-1}+\dots+\epsilon_1X_1) \] 
and
\[
B := \prod_{i=n-t+1}^n \prod_{\epsilon_{i-1}\in \F_p} \dots \prod_{\epsilon_{1}\in \F_p} (X_i+\epsilon_{i-1}X_{i-1}+\dots+\epsilon_1X_1).
\]

Moore's formula ensures that $ A =  \Delta_{n-t}(X_1, \dots, X_{n-t})$, while the definition of the structural polynomial gives that
\[B= \prod_{i=n-t+1}^n\prod_{\epsilon_{i-1}\in \F_p} \dots \prod_{\epsilon_{n-t+1}\in \F_p} P_{V_{n-t}}(X_i+\epsilon_{i-1}X_{i-1}+\dots+\epsilon_{n-t+1}X_{n-t+1}).\]

By Lemma \ref{lem:FM2.2.}, $P_{V_{n-t}}$ is an additive polynomial and therefore the above is also equal (after reindexing) to 
\[ B = \prod_{i=1}^t \prod_{\epsilon_{i-1}\in \F_p} \dots \prod_{\epsilon_{1}\in \F_p} (P_{V_{n-t}}(X_{n-t+i})+ \epsilon_{i-1} P_{V_{n-t}}(X_{n-t+i-1})+
\dots+\epsilon_1 P_{V_{n-t}}(X_{n-t+1})),\]
and we can see from Moore's formula that this is precisely equal to
\[\Delta_{t}(P_{V_{n-t}}(X_{n-t+1}),\dots,P_{V_{n-t}}(X_n)),\]
which completes the proof of the claim.
\end{proof} 

\begin{corollary}\label{cor:structuralMoore} 
Under the hypotheses of Lemma \ref{lem:structuralMoore}, for every $n-t+1\leq i\leq n$ we have that
\[\frac{\Delta_{n-1}(X_1,\dots,\widehat{X_i},\dots,X_n)}{\Delta_n(X_1,\dots,X_n)}=\frac{\Delta_{t-1}(P_{V_{n-t}}(X_{n-t+1}),\dots,\widehat{P_{V_{n-t}}(X_i)},\dots,P_{V_{n-t}}(X_n))}{\Delta_{t}(P_{V_{n-t}}(X_{n-t+1}),\dots,P_{V_{n-t}}(X_n))}.\]
 \end{corollary}

\begin{lemma}\label{lem:Anew}
Let $n \geq 1$, let $M \in \GL_n(\F_p)$ and let $X_1,\dots,X_n$ be variables over $\F_p$.
Denote by $Y_i=\Delta_{n-1}(\hat{\ul X_i})$ and by $Y_i^M=\Delta_{n-1}\left(\widehat{(\ul{X}M)_i}\right)$.
Then, we have
\[ (Y_1^M, \dots, Y_n^M) = \ul{Y}M^c,\]
where $M^c \in GL_n(\F_p)$ is the cofactor matrix of $M$
\end{lemma}
\begin{proof}
This is proved with a direct computation.
We write $M = \begin{pmatrix}
m_{ij}
\end{pmatrix}_{ij}$ and, for simplicity, we show only that $Y_1^M = \ul{Y} M^c_{\bullet,1}$ where $M^c_{\bullet,1}$ is the first column of $M^c$, the proof that $Y_j^M = \ul{Y} M^c_{\bullet,j}$ being completely analogous.
By definition, we have
\[Y_1^M= \Delta_{n-1}\left(\sum_{i=1}^n m_{i,2}X_2, \dots, \sum_{i=1}^n m_{i,n}X_n\right)\]
and, using the multi-linear properties of Moore determinants, this can be rewritten as
\begin{align*}
Y_1^M = \sum_{i_2, \dots, i_n} m_{i_2, 2} \cdots m_{i_n, n} \Delta_{n-1}(X_{i_2}, \dots, X_{i_n}) & = \sum_{\sigma \in \frakS_n} m_{\sigma(2),2} \cdots m_{\sigma(n),n} \Delta_{n-1}(X_{\sigma(2)}, \dots, X_{\sigma(n)}) \\
& = \sum_{\sigma \in \frakS_n} (-1)^{\sgn(\sigma)} m_{\sigma(2),2} \cdots m_{\sigma(n),n} Y_{\sigma(1)},
\end{align*}
where $\frakS_n$ is the symmetric group over the set with $n$ elements (note that we used that $m_{i,j}^p=m_{i,j}$ for getting rid of the exponents).
Since for every $i$, the coefficient of $Y_i$ is the minor $M_{i1}$ of the matrix $M$, it results that $Y_1^M = \ul{Y} M^c_{\bullet,1}$.
\end{proof}

\bibliographystyle{siam}                         
\bibliography{VectorDiffForms.bib}

\vfill

\end{document}